\documentclass[11pt,a4paper]{article}

\usepackage[utf8]{inputenc}
\usepackage[english]{babel}
\usepackage{amsmath,amssymb,mathrsfs,latexsym,xcolor}
\usepackage{todonotes}
\usepackage{authblk}

\usepackage{fullpage}

\usepackage{amsthm}
\usepackage{amsfonts}
\usepackage{amssymb}
\newtheorem{e}{Example}
\usepackage{mathtools,bm}
\usepackage{comment}
\usepackage{stmaryrd}
\usepackage[shortlabels]{enumitem}
\usepackage{hyperref}
\usepackage[nameinlink,capitalize]{cleveref}
\newcommand{\assign}{:=}
\newcommand{\backassign}{=:}
\newcommand{\nobracket}{}

\newcommand{\tmcolor}[2]{{\color{#1}{#2}}}
\newcommand{\tmmathbf}[1]{\ensuremath{\boldsymbol{#1}}}
\newcommand{\tmop}[1]{\ensuremath{\operatorname{#1}}}

\newcommand{\tmtextit}[1]{\text{{\itshape{#1}}}}

\newtheorem{thm}{Theorem}[section]
\newtheorem{theorem}[thm]{Theorem}

\newtheorem{lem}[thm]{Lemma}

\newtheorem{remark}[thm]{Remark}
\newtheorem{notation}[thm]{Notation}
\newtheorem{prop}[thm]{Proposition}
\newtheorem{proposition}[thm]{Proposition}

\newtheorem{definition}[thm]{Definition}

\newtheorem{assumptioninner}{Assumption}
\newenvironment{assumption}[1]
  {\renewcommand\theassumptioninner{#1}%
   \assumptioninner}
  {\endassumptioninner}

\newcommand{\R}{\mathbb{R}}
\newcommand{\E}{\mathbb{E}}
\newcommand{\Y}{\mathbb{Y}}

\newcommand{\be}{\begin{equation}}
\newcommand{\ee}{\end{equation}}
\newcommand{\bea}{\begin{eqnarray}}
\newcommand{\eea}{\end{eqnarray}}

\newcommand{\vars}{\varsigma}

\newcommand{\BY}{\mathbf{Y}}

\newcommand{\MF}{\mathcal{F}}
\newcommand{\MB}{\mathcal{B}}

\newcommand{\ML}{\mathcal{L}}
\newcommand{\MG}{\mathcal{G}}

\newcommand{\SC}{\mathscr{C}}
\newcommand{\SD}{\mathscr{D}}

\newcommand{\PP}{\mathbb{P}}

\newcommand{\hmu}{{\hat{\mu}}}

\newcommand{\ome}{\omega}

\DeclareFontEncoding{FMS}{}{}
\DeclareFontSubstitution{FMS}{futm}{m}{n}
\DeclareFontEncoding{FMX}{}{}
\DeclareFontSubstitution{FMX}{futm}{m}{n}
\DeclareSymbolFont{fouriersymbols}{FMS}{futm}{m}{n}
\DeclareSymbolFont{fourierlargesymbols}{FMX}{futm}{m}{n}
\DeclareMathDelimiter{\vvert}{\mathord}{fouriersymbols}{152}{fourierlargesymbols}{147}
\DeclarePairedDelimiter{\nn}{\vvert}{\vvert}

\begin{document}


\title{Rough stochastic filtering}

\date{}

\author{Fabio Bugini}
\affil{TU Berlin}

\author{Peter K.~Friz}
\affil{TU Berlin and WIAS Berlin}

\author{Khoa L\^e}
\affil{School of Mathematics, University of Leeds}

\author{Huilin Zhang}
\affil{Shandong U. and Humboldt U.}

\maketitle

\begin{abstract}
This article is concerned with the well-posedness of the ``filtering equations'', due to Zakai and Kushner--Stratonovich, arising in nonlinear stochastic filtering. 
In general situations, notably in correlated diffusion models and when signal coefficients depend on the observation process, the well-posedness is a difficult problem, mainly due to conflicting martingale structures of the involved forward and backward equations.
Crisan--Pardoux (2024) address this classical problem with BSPDE techniques, Du et al. (2013), a Sobolev-based approach that however requires increasingly strong regularity assumptions in high dimensions. In this work, we take a new mixed rough stochastic perspective which allows us to derive well-posed rough counterparts of the filtering equations. 
Importantly, the rough filtering equations are seen, upon randomization, to coincide with the classical filtering equations. Our framework yields well-posedness (existence, uniqueness, stability) under dimension-independent regularity assumptions, providing a robust and conceptually unified solution to a longstanding problem in stochastic filtering theory. To illustrate the flexibility of the method, we also treat rough versions of the classical Kalman--Bucy filter, with characteristics described by a new class of RDEs of rough Riccati type. 
\end{abstract}



\tableofcontents

\section{Introduction}


This article is concerned with the well-posedness of the filtering equations arising in nonlinear stochastic filtering theory. Providing a just historic account is a research task by itself (e.g. \cite{crisan2014stochastic}; with due credit to R.\,E.~Kalman, G.~Kallianpur, H.~Kushner, R.\,L.~Stratonovich, N.~Wiener and many others). Clarifying  the probabilistic, analytic, and the SPDE-based structure of filtering equations has been a major community effort. 
With regard to the goals of this article, we may point specifically to A.~Bensoussan, M.~Clark, D. Crisan, H.~Kunita, M. Davis, N.~Krylov, B.~Rozovskii, R.~Karandikar, and  E.~Pardoux. %
Selected references of expository nature include  \cite{bensoussan1992stochastic,pardoux2006filtrage,BC09,RozovskiiLototsky2018, Yau24}.

The filtering equations of interest, the so-called Zakai and Kushner--Stratonovich equations, are linear resp. non-linear stochastic partial differential equations of parabolic type, of form\footnote{Notation consistent with \cref{sec:review}.}
\begin{align*} 
\mu_t (\varphi) & = \mu_0 (\varphi) + \int_0^t \mu_r  (A_r \varphi) dr \nonumber   + \int_0^t \mu_r  (\nabla \varphi f_r  + \varphi h_r )  
  d Y_r, \label{equ:ZI}
   \\
   \varsigma_t (\varphi)  & =
  \varsigma_0 (\varphi) + \int_0^t \vars_r  (A_r \varphi) dr \notag  + \int_0^t \left( \varsigma _r  (\nabla \varphi f_r  +
    \varphi h_r ) - \varsigma_r(\varphi)\varsigma_r(h_r) \right) \left( \mathrm{d} Y_r - \frac{d\langle Y, Y \rangle_r}{dr} \varsigma_r (h_r^\top) \, dr \right) ,  
\end{align*} 
with $\R^{d_Y}$-dimensional observation noise $Y$, and diffusion generator $A$ of a $\R^{d_X}$-dimensional signal  process $X$. 
Well-posedness (existence, uniqueness, stability) of these filtering equations is a classical problem. Existence is naturally obtained via $\mathcal{F}^Y$-conditional Feynman--Kac representations (involving the likelihood process $Z$) and the Kallianpur--Striebel formula,
$$
 \mu_t (\varphi) = \mathbb{E} [ Z_t \varphi (X_t)  \mid
     \mathcal{F}^Y_t], \ \  
 \varsigma_t (\varphi) = \mathbb{E}^o [\varphi (X_t) \mid \mathcal{F}^Y_t] =  
     \frac{\mu_t (\varphi)}{\mu_t (1)}, 
$$
where $\mathbb{E}^o$ denotes expectation under an ``original'' measure, connected to $\mathbb{E}$ via $Z$, by the Girsanov theorem. 
Stability is then essentially a question about robust representations. In the latter, the role of rough paths is well-appreciated  \cite{Lyo98, CDFO13,AC20, CNNR23, allan2025roughsdesrobustfiltering}.  The uniqueness  problem for the Kushner--Stratonovich equation can be reduced to the uniqueness of the Zakai equation. 
A general approach to the uniqueness of linear equations is the duality method. {\em Without observation} ($h=0$) the filtering equations reduce to the  Kolmogorov forward equation; the duality method then amounts to pairing a measure-valued solution to the forward equation with a (sufficiently regular) solution of the corresponding backward equation. An application of the product rule shows that this pairing remains constant in time, which immediately yields uniqueness for (measure-valued) solutions of the forward equation. 

{\em With observation}, however, i.e. in the case when the Zakai equation is a genuine {\em stochastic} partial differential equation, the classical duality argument breaks down. The reason is that any (backward) martingale structure of a backward stochastic equation is in general incompatible with that of the forward equation, making it impossible to apply (It\^o) product rules. This has, amongst other things, inspired the development of a forward-backward stochastic calculus \cite{PP87}, although the standard proof of uniqueness is, to the best of our knowledge,
found in \cite{bensoussan1992stochastic}. Here, the measure valued solution of the Zakai equation is paired with a regular solution 
of a deterministic backward PDE; the duality property is shown through the use of a collection of exponential martingales first introduced in the filtering context in \cite{KR81}.
As pointed out in \cite{CP24}, this method fails (as would \cite{PP87} if suitably extended to infinite dimensions\footnote{Cf. Section 7.4. in \cite{PP87}, the problem being the absence of a (forward/backward)semimartingality generality for the integrands. We thank Etienne Pardoux for a related email exchange.})
when the signal coefficients are also allowed to depend on the observation process, a feature that is both natural and important from a filtering perspective.

Surprisingly enough, only recently Crisan--Pardoux \cite{CP24} could solve this problem, also in a remarkable generality of degenerate observation noise, by constructing backward equations adapted to the forward filtration, using ideas from backward stochastic differential equations in Sobolev spaces \cite{DTZ13}. (We summarize their results in  \cref{sec:review}.) As a consequence of Sobolev embeddings, the construction of sufficiently regular backward solutions requires increasingly strong regularity assumptions on the coefficients as the dimension of the signal increases. The purpose of a uniqueness results (here for Zakai or Kushner--Stratonovich) is the guaranty that any solution, typically obtained via numerical methods, indeed provides (or at least approximates) the object of interest: the filter $\mathbb{E}^o [\varphi (X_t) \mid \mathcal{F}^Y_t]$. 
Kalman--Bucy theory describes situations when the filter is conditionally Gaussian, the Kushner--Stratonovich is then effectively replaced by a finite-dimensional Riccati SDEs for mean and variance of the conditional law; cf e.g. \cite{HP88}, \cite[Sec 6.3]{pardoux2006filtrage} or \cite[Sec 8.1]{BC09}.

\medskip
{\em Contributions of this article.} In a precise sense, we show in this work the well-posedness of the filtering equations under dimension-independent regularity assumptions. (This opens up the possibility to treat infinite-dimensional signal processes, although such is left to future investigations.) A key insight of this work is that such well-posedness results are possible on the level of the {\em rough filtering equations}, valid for generic observation rough paths $\BY$, of the form 
\begin{align*} 
\mu^{\BY}_t (\varphi) &  =  {\mu _0}^{\BY} (\varphi) +  \int_0^t \mu^{\BY}_r  (A^{\BY}_r \varphi) dr + \int_0^t  \mu^{\BY}_r (\Gamma^{\BY}_r \varphi)  \, d \BY_r \\
   \varsigma_t^{\BY} (\varphi) & = \varsigma_0^{\BY} (\varphi) + \int_0^t \varsigma_r^\BY  (A^{\BY}_r \varphi) \, d r + \int_0^t (\varsigma_r^{\BY}  (\Gamma^{\BY}_r \varphi) - \varsigma_r^{\BY} (\varphi)  \varsigma_r^{\BY} (h_r^{\BY}) (d  \BY_r - \dot{[\BY]}_r  \varsigma_r (h^\BY)^\top) d r). 
\end{align*} 
Underlying our analysis, given in Section \ref{sec:RoughFiltering}, are rough counterparts of conditional Feynman-Kac and Kallianpur--Striebel formulae, of the form 
$$
 \mu^\BY_t (\varphi) = \mathbb{E} [Z^\BY_t \varphi (X^\BY_t)], \ \  
 \varsigma^\BY_t (\varphi) =  
     \frac{\mu^\BY_t (\varphi)}{\mu^\BY_t (1)}. 
$$
In contrast to the stochastic case, there is no more change of measure, for $\BY$ is deterministic. Here $X^\BY, Z^\BY$ denote the rough version of signal and likelihood process, formally obtained by replacing observation noise $Y$ by a generic rough path $\BY$. We rigorously treat these (in Section \ref{sec:RSDXZ}) via {\em rough stochastic differential equations} in the recent sense of \cite{FHL21, BFS24}, with brief summary in 
\cref{sec:elem}.
In particular, we give in Sections \ref{sec:roughZakai_existence} and \ref{sec:RKS} rigorous definitions of the {\em rough Zakai} and {\em Kushner--Stratonovich equation}, of the form displayed above. (Section \ref{subsec:degen-case} for the degenerate observation noise case). These are, respectively, linear and non-linear {\em rough}  partial differential equations of parabolic type.
By ``removing probability'' (to the extent possible) from the filtering equations it becomes possible to perform a direct pairing of rough forward and backward equations, at the price of moving to a space-time setting of strongly controlled rough paths. 
Contrary to many previous works on rough partial differential equations, including \cite{GT10, HH18, HN19, GH19}, we do not restrict to geometric rough paths, nor do we make ellipticity assumptions (related to mild formulation, energy methods etc.), 
instead we 
work in the for filtering problems natural generality of spaces of measures. 

In Section \ref{section:roughKB} we consider (affine-)linear rough stochastic signal dynamics, employing rough stochastic differential equations with linear coefficients \cite{BCN24}, and develop a  rough counterpart of Kalman--Bucy theory. The rough counterpart of the filter, $\varsigma^\BY$ is seen to be Gaussian, with arguments rather different from the stochastic case, including a method of controlled approximations, with exponential integrability of second Wiener It\^o chaos also at play (\cref{app:EWC}). The evolution of mean and covariance is seen to follow a {\em Riccati rough differential equation}, which appears here (to our best knowledge) for the first time. We give a direct well-posedness (non-explosion) argument, outside the scope of existing results \cite{RS17,LY25}. In view of  future uses (e.g.\ in linear-quadratic control problems or mean field games with common noise), the relevant material in Section \ref{sec:RRE} is kept maximally self-contained. 



Last not least, in 
Section \ref{section:bridgingroughandstochasticfiltering}, also end of Sections \ref{subsec:degen-case} and \ref{sec:RKBF}, we connect the ``rough'' objects or previous sections with those from stochastic filtering theory (cf. \cref{sec:review}). In particular, rough signal dynamics are seen to become the correct stochastic signal dynamics under randomization of the rough path, i.e. when $\BY$ is taken as typical realization of the It\^o lift of the observation process, 
$$
   \BY^{\text{It\^o}}(\omega) = (Y,\Y)(\omega), \ \ \Y_{s,t}(\omega) = \Big( \int_s^t \delta Y_{s,r} \otimes dY_r \Big) (\omega).
$$
By consequence, the randomization of $\varsigma^\BY_t (\varphi) = \mathbb{E} [Z^\BY_t \varphi (X^\BY_t)]/ \mathbb{E} [Z^\BY_t ]$ is seen to be a version of the overall  object of interest, the (``stochastic'') filter $\mathbb{E}^o [\varphi (X_t) \mid \mathcal{F}^Y_t]$. In Section \ref{sec:SvsRFW} we obtain analagous statements directly on the filtering equations. The uniquess result for the  rough Kushner--Stratonovich equation then guarantees that any solution $\varsigma^\BY_t$, possibly obtained by numerical rough PDE methods, indeed provides the correct 
(or at least approximates the desired) rough filter, which agrees - upon randomization - with the stochastic filter, object of interest. 
In this sense the uniqueness results obtained in this work for the rough filtering equations are just as useful as those for the classical stochastic equations, with the additional benefit of dealing with deterministic (rough) equations which inherit many robustness properties of rough analysis.

\medskip



{\bf Acknowledgement.}
FB is supported by DFG 
- Project-ID 410208580 - IRTG2544 (``Stochastic Analysis in Interaction''). 
PKF and HZ acknowledge support from DFG CRC/TRR 388 ``Rough
Analysis, Stochastic Dynamics and Related Fields''. 
Part of this work was carried out during a visit of the second author to Shandong University. 
KL acknowledges support from EPSRC
[grant number EP/Y016955/1]. HZ is partially supported by the Fundamental Research Funds for the Central Universities, NSF of China and Shandong (Grant Numbers 12031009, ZR2023MA026), Young Research Project of Tai-Shan (No.tsqn202306054).

\section{Notation}  \label{sec:notation}
Let $d,d_1,d_2 \in \mathbb{N}_{\ge 1}$. 
We shall denote by $\mathcal{M}_F(\mathbb{R}^d)$ the set of finite measures on $\mathbb{R}^d$ and by $\mathcal{P}(\R^d)$ the set of probability measures on $\R^d$. For any $\mu \in \mathcal{M}_F(\R^d)$ and for any $\psi:\R^{d} \to \R$ Borel measurable and bounded, we define \begin{equation*}
    \mu(\psi) := \int_{\R^d} \psi(x) \, \mu(dx).
\end{equation*}
We equip $\mathcal{M}_F(\mathbb{R}^d)$ with the topology of weak convergence (i.e.\ $\mu_n \to \mu$ if $\langle \mu_n, \varphi \rangle \to \langle \mu, \varphi \rangle$, for all $\varphi \in \mathcal{C}^0_b(\R^d)$, the set of continuous bounded functions on $\mathbb{R}^d$). $\mathcal{B}(\R^d)$ denotes the Borel $\sigma$-algebra on $\R^d$ and $\mathcal{B}_b (\mathbb{R}^d)$ is the space of bounded measurable functions from $\R^d$ to $\R$.

In the (stochastic) filtering context, $\mathbb{P}^o$ stands for the original probability measure, under which the filtering dynamics is specified, $\mathbb{P}$ the (Girsanov) changed measure. 
Write $\varsigma_t$ for the $\mathcal{P}(\R^d)$-valued filter process, $\mu_t$ for the $\mathcal{M}_F(\R^d)$-valued unnormalized conditional distribution process. 
Level-2 rough paths are generically denoted by $\mathbf{Y} = (Y, \mathbb{Y})$, elements of a rough path space $(\mathscr{C}, \rho)$, cf.  Section \ref{sec:RPs}.

Given $\varphi: \R^{d_1} \to \R^{d_2}$, we denote by $|\varphi|_\infty := \sup_{x \in \R^{d_1}} |\varphi(x)|$ and by $D^j \varphi$ its $j$-th (Fréchet) derivative, if it exists. Partial Fréchet derivatives with respect to the variable $x_i$ are denoted by $D^j_{x_i} \varphi$. 
Let $N \in \mathbb{N}_{\ge 0}$
and let $\gamma \in (0,1]$.
We say that $\varphi \in \mathcal{C}^{N+\gamma}(\R^{d_1};\R^{d_2})$ if it is $N$-times continuously differentiable in classical (Fréchet) sense and its $N$-th (Fréchet) derivative is $\gamma$-H\"older continuous. 
We say that $\varphi \in \mathcal{C}^{N+\gamma}_b (\R^{d_1};\R^{d_2})$ if, in addition, \begin{equation*}
    |\varphi|_{\mathcal{C}^{N+ \gamma}_b} := |\varphi|_\infty + \sum_{j=1}^N |D^j \varphi|_\infty + \sup_{x\ne y} \frac{|D^N \varphi(x) - D^N\varphi(y)|}{|x-y|^\gamma} <+\infty 
\end{equation*}
In particular, $\mathcal{C}^1(\R^{d_1};\R^{d_2})$ denotes the space of Lipschitz continuous functions from $\R^{d_1}$ to $\R^{d_2}$, not the space of continuously (Fréchet) differentiable functions. 
Similarly, $\mathcal{C}^2(\R^{d_1};\R^{d_2})$ is the space of continuously differentiable functions whose first derivative is Lipschitz continuous.
{We write $C^k$ for functions with continuous $k$-th derivatives for any $k \in  \mathbb{N}_{\ge 0}$, and $C^\alpha$ for $\alpha$-H\"older continuous paths with $\alpha \in (0,1)$.}

Given any $(Y_t)_t$ we denote the increment process as $\delta Y_{s,t} = Y_t - Y_s$. 
We write $\Psi \lesssim \Phi$ if there is a constant $C >0$ such that $\Psi \le C \Phi$.  
We denote by $\mathrm{Lin}(\R^{d_1},\R^{d_2})$ the space of linear (and bounded) functions from $\R^{d_1}$ to $\R^{d_2}$, with associated operator norm $|\varphi|_{\mathrm{Lin}} := \sup_{|x| \le 1} |\varphi(x)|$. 
Let $V$ be a Banach space and denote by $\{e_i, \, i=1,\dots,d\}$ the canonical basis of $\R^d$. 
For any map $f:V \to \R^d$, we denote by $f^i : V \to \R$ the corresponding $i$-th component map, $i=1,\dots,d$. 
Given $f : V \to \mathrm{Lin}(\R^m,\R^d)$ and for any $\kappa = 1,\dots,m$, we denote by $f_\kappa :V \to \R^d, \ f_\kappa(v) := f(v) e_\kappa$. In particular, $f(v) y = \sum_{\kappa=1}^m f_\kappa(v) y^\kappa$ for any $y \in \R^m$ and for any $v \in V$. Similarly, given $f: V \to \mathrm{Lin}(\R^m \otimes \R^m,\R^d)$, we denote by $f_{\kappa \lambda}(v) := f(v) (e_\kappa \otimes e_\lambda)$.

\medskip

\medskip


\section{Rough stochastic filtering} \label{sec:RoughFiltering}

\subsection{Rough stochastic dynamics of signal and observation } \label{sec:RSDXZ}
In classical filtering theory of diffusions, as reviewed in \cref{sec:review}, the signal/observation pair $(X,Y)$ is enhanced with likelihood process $Z$, which captures a Girasonv change of measure, from some original measure $\mathbb{P}^o$ to a new measure $\mathbb{P}$, which offers a better independence structure. With momentary focus on the %
(classical\footnote{Eg. \cite{pardoux2006filtrage}; also \cite[Rmk 2.1]{HP88} for reduction to case $k=I$.}) case of non-degenerate observation noise, 
consider general (``correlated'') joint dynamics of the form  
\begin{align*}
    d X_t & =  b (t, X_t, Y_t) d t + \sigma (t, X_t,  Y_t) d B_t (\omega) +  f (t, X_t, Y_t) \, d Y_t (\omega)
     \\
  d Y_t & =  
     k(t, Y_t) \, d W_t, \qquad \qquad \  Y_0 = 0  \\
     d Z_t & =  Z_t h (t, X_t, Y_t) \, dY_t, 
     \qquad   Z_0 = 1.
\end{align*}
Here $(B, W)$ are $\mathbb{P}$-independent Brownian motions, of respective dimension $d_W,d_B$. The process triplet $(X,Y,Z)$ takes values in Euclidean space of dimension $d_X+d_Y+1$. The observation process $Y$ is a local martingale under $\mathbb{P}$. The pair $(B, Y)$ are $\mathbb{P}$-independent and jointly constitute the driving noise for the signal $X$. Note that $Z$ can be written explicitly as stochastic exponential, involving the bracket of $Y$ with derivative 
$$ \dot{\langle Y \rangle}_t = \frac{d\langle Y \rangle_t}{dt} = (kk^\top)(t,Y).
$$

\medskip

\noindent {\em Rough paths.} This suggests to consider observation sample paths of $Y$ lifted to a space non-geometric rough paths with suitable bracket structure 
$$  
\mathscr{C }^{0, \alpha , 1} = \{ \BY \in
\mathscr{C}^{0, \alpha} ([0, T]) \mid t \mapsto \delta [\BY ]_{0, t} \in C^1
([0, T]) \}, \quad \dot{[\BY]}_t = \frac{d[\mathbf{Y}]_t}{dt}.
$$
With deterministic observation paths, we work on a filtered stochastic basis $(\Omega, \mathcal{F}, \{\mathcal{F}_t\}_{t \in [0,T]}, \mathbb{P})$ that supports a Brownian motion $B$. This is precisely the stochastic setup for rough stochastic differential equations (reviewed in \cref{sec:elem}), which allows us
to introduce the $(d_X+1)$-dimensional rough signal and likelihood process 
\begin{align} 
    d X^{\BY}_t &=  b (t, X^{\BY}_t; \BY )     d t + \sigma (t, X^{\BY}_t;\BY) d B_t (\omega) +     f (t, X^{\BY}_t ; \BY) \, d\BY_t  \label{eq:roughSDE_filtering} \\ 
    dZ^\BY_t &= Z^\BY_t h(t,X^\BY_t;\BY) \, d\BY_t, \qquad Z_0^\BY = 1 \in \R. \label{eq:Girsanovexponential_filtering} 
\end{align}
On the coefficients appearing in \eqref{eq:roughSDE_filtering} and \eqref{eq:Girsanovexponential_filtering}, we impose the following:

\begin{assumption}{E}\label{assum:E}

For any $\BY \in \mathscr{C}^{0,\alpha,1}$, the measurable functions 
    \begin{align*}
        b (\cdot, \cdot  \, ;\BY) &: [0,T] \times \R^{d_X}  \to \R^{d_X} \\
        \sigma (\cdot, \cdot  \, ;\BY) &: [0,T] \times \R^{d_X}  \to \mathrm{Lin}(\R^{d_B},\R^{d_X}) \equiv \R^{d_X \times d_B} \\
        f (\cdot, \cdot  \, ;\BY) &: [0,T] \times \R^{d_X}  \to \mathrm{Lin}(\R^{d_Y},\R^{d_X}) \equiv \R^{d_X \times d_Y} \\
        h (\cdot, \cdot  \, ;\BY) &: [0,T] \times \R^{d_X}  \to \mathrm{Lin}(\R^{d_Y},\R) \equiv \R^{1 \times d_Y} 
    \end{align*}
have the following properties: 
\begin{itemize}
    \item $(t,x) \mapsto (b(t,x ;\BY), \sigma(t, x ;\BY)) $ are bounded and uniformly Lipchitz in the sense that
    there exists a finite constant $\|b(\cdot, \cdot \, ; \BY)\|_{Lip}$ such that (and similar for $\sigma$) 
    \begin{equation*}    
    \sup_{t \in [0,T]} \sup_{\substack{x,y \in \R^{d_X}\\x \neq y}} \frac{|b(t,x;\BY) - b(t,y;\BY)|}{|x-y|} \le \|b(\cdot, \cdot \, ; \BY)\|_{Lip}. 
    \end{equation*} 
    \item  there is a measurable function $f'(\cdot, \cdot  \, ;\BY):[0,T] \times \R^{d_X}  \to \mathrm{Lin}(\R^{d_Y},\mathrm{Lin}(\R^{d_Y},\R^{d_X})) \equiv \mathrm{Lin}(\R^{d_Y} \otimes \R^{d_Y} , \R^{d_X})$ and there is a constant $\beta \in \left( \frac{1}{\alpha}, 3 \right]$ such that 
    \begin{equation*}
        [t \mapsto (f(t,\cdot \, ;\BY), f'(t,\cdot \, ;\BY) )] \in \mathscr{D}_Y^{2\alpha} \mathcal{C}^\beta_b  
    \end{equation*}
in the sense of \cref{def:stochasticcontrolledvectorfields_intro} with $Y = \pi_1 \BY$;
\item  there is a measurable function $h'(\cdot, \cdot  \, ;\BY):[0,T] \times \R^{d_X} \to \mathrm{Lin}(\R^{d_Y},\mathrm{Lin}(\R^{d_Y},\R))$ such that $[t \mapsto (h(t,\cdot \, ; \BY), h'(t,\cdot \, ;\BY) )] \in \mathscr{D}_Y^{2\alpha} \mathcal{C}^2_b $
in the sense of \cref{def:stochasticcontrolledvectorfields_intro}.
\end{itemize}
\end{assumption}

\begin{theorem} \label{thm:wellposedness_roughfiltering}
  Let $X^\BY_0 \in L_0(\Omega,\mathcal{F}_0;\R^{d_X})$. Under \cref{assum:E}, there exists a unique $L_{2,\infty}$-integrable solution $X^\BY$ to equation \eqref{eq:roughSDE_filtering} in the sense of \cref{def:integrablesolutionsRSDEs}
and also a unique $L_2$-integrable solution  $Z^\BY$ to the linear RSDE \eqref{eq:Girsanovexponential_filtering}, given by the \emph{rough stochastic exponential} \begin{equation} \label{eq:roughstochasticexponential}
    Z^\BY_t := \exp(I_t^\BY) := \exp \left( \int_0^t h(r,X^\BY_r;\BY) \, d\BY_r - \frac12 \int_0^t h(r,X^\BY_r ; \BY) \dot{[\BY]}_r h(r,X^\BY_r ; \BY)^\top \, dr \right)
    . 
\end{equation} 

\end{theorem}

\begin{proof}
    The well-posedness of equation \eqref{eq:roughSDE_filtering} under \cref{assum:E} follows directly from \cite[Theorem 4.6]{FHL21}. 
    In particular, let us recall that $X^\BY$ is a continuous and adapted process such that, for any $0 \le s \le t \le T$,  \begin{align*}
        X^\BY_t - X^\BY_s = & \int_s^t b(r,X^\BY_r;\BY) \, dr + \sum_{\theta=1}^{d_B} \int_s^t \sigma_\theta(r,X^\BY_r;\BY) \, dB^\theta_r + \sum_{\kappa=1}^{d_Y} f_\kappa(s,X^\BY_s;\BY) \delta Y^\kappa_{s,t} \\
       & + \sum_{\kappa,\lambda=1}^{d_Y} \big( \sum_{j=1}^{d_X} \partial_{x^j} f_\lambda(s,X_s^\BY;\BY) f^j_\kappa (s,X_s^\BY;\BY) + f'_{\lambda \kappa}(s,X_s^\BY;\BY)    \big) \mathbb{Y}^{\kappa \lambda}_{s,t} + X^\natural_{s,t}
    \end{align*}
    with $$ \sup_{0 \le s < t \le T} \frac{\|\E(|X^\natural_{s,t}|^2 \mid \mathcal{F}_s^B)^\frac{1}{2}\|_\infty}{|t-s|^{2\alpha}} < +\infty, \quad  \sup_{0 \le s < t \le T} \frac{\|\E(X^\natural_{s,t} \mid \mathcal{F}_s^B)\|_\infty}{|t-s|^{3\alpha}} < +\infty.$$ 

    By \cite[Lemma 3.11]{FHL21}, $[t \mapsto (h(t,X^\BY_t;\BY), D_x h(t,X^\BY_t;\BY) f(t,X^\BY_t;\BY) + h'(t,X^\BY_t;\BY))]$ is a stochastic controlled rough path in $\mathbf{D}_Y^{2\alpha}L_{2,\infty}$. 
    Hence, by John-Nirenberg inequality \cite[Proposition 2.8]{FHL21}, \begin{equation*}
        \sup_{t \in [0,T]} \E(e^{p |\int_0^t h(r,X^\BY_r;\BY) \, d\BY_r|}) < +\infty, \qquad \forall p<\infty .
    \end{equation*} 
    In particular, defining $Z^\BY$ as in \eqref{eq:roughstochasticexponential}, $Z^\BY_t \in \bigcap_{p \in [1,\infty)} L_p(\Omega;\R)$ for any $t \in [0,T]$ and one sees that $Z^\BY$ is an $L_2$-integrable solution to \eqref{eq:Girsanovexponential_filtering} via an application of the rough stochastic It\^o formula (cf. \cref{eq:roughItoformula_roughItoprocess}.3).
    We therefore have, for any $0 \le s \le t \le T$,
    \begin{align*}
        Z^\BY_t - Z^\BY_s = & \sum_{\kappa=1}^{d_Y} Z_s ^\BY h_\kappa(s,X^\BY_s;\BY) \delta Y^\kappa_{s,t}  + \sum_{\kappa,\lambda=1}^{d_Y} \big( Z^\BY_s h_\lambda (s,X^\BY_s;\BY) h_\kappa(s,X^\BY_s;\BY) \\ 
        & + \sum_{j=1}^{d_X}  \partial_{x^j} h_\lambda(s,X^\BY_s;\BY) Z^\BY_s f^j_\kappa(s,X^\BY;\BY) +  Z^\BY_s h'_{\lambda \kappa}(s,X^\BY_s;\BY  \big) \mathbb{Y}_{s,t}^{\kappa \lambda} + Z^\natural_{s,t}
    \end{align*}
    with 
    $\sup_{0 \le s < t \le T} {\|Z^\natural_{s,t}\|_2}/{|t-s|^{2\alpha}} < +\infty$ and $\sup_{0 \le s < t \le T} {\|\E(Z^\natural_{s,t} \mid \mathcal{F}_s^B)\|_2}/{|t-s|^{3\alpha}} < +\infty$.
    To prove uniqueness for \eqref{eq:Girsanovexponential_filtering}, let $\tilde Z^\BY$ be any other $L_2$-integrable solution. With $Z_t^\BY >0$ for any $t \in [0,T]$, we conclude from \begin{align*}
        d \left( \frac{\tilde Z^\BY}{Z^\BY} \right)_t &= \tilde Z^\BY_t \left( -\frac{1}{Z_t^\BY} h(t,X^\BY_t ; \BY) \, d\BY_t + \frac{1}{Z_t^\BY} h(t,X^\BY_t ; \BY) \dot{[\BY]}_t h(t,X^\BY_t ; \BY)^\top \, dt  \right) \\
        & \quad + \frac{1}{Z_t^\BY} \, d \tilde Z^\BY_t + \tilde Z^\BY_t h(t, X^\BY_t ; \BY) \dot{[\BY]}_t \left( -\frac{1}{Z^\BY_t} h(t,X^\BY_t ; \BY) \right)^\top =0,
    \end{align*}
    consequence of the rough stochastic It\^o formula \cref{eq:roughItoformula_roughItoprocess}.2 .
\end{proof}

\begin{remark} \label{rem:afterE} At this stage, no measurability, causality or anything else is required on the coefficients with respect to $\BY = (\pi_1 \BY, \pi_2 \BY)$, which is fixed. This will enter the picture in \cref{section:bridgingroughandstochasticfiltering} when we make the link to the original stochastic problem.
 This setup indeed covers the case $b (t, x; \BY) = b ( t, x, (\pi_1 \BY)_t)$, similar for $\sigma, f, h_1$ and $h$. 
 \end{remark}

\subsection{Rough Kallianpur--Striebel}




    

Let $\BY \in \mathscr{C}^{0,\alpha,1}$ and $\varphi\in \mathcal{B}_b(\R^{d_X})$. Motivated by the corresponding stochastic objects (cf. \cref{sec:stochKSformula}, classical in filtering theory),  \cref{thm:wellposedness_roughfiltering} allows us to
define the {\em rough unnormalised filter}
\begin{equation} \label{eq:roughunnormfilter}
    \mu_t^\BY (\varphi) := \E( \varphi(X_t^\BY) Z^\BY_t), \qquad t \in [0,T],
\end{equation}
and then the {\em rough Kallianpur--Striebel map}
\begin{equation} \label{eq:roughKSform} 
\BY \mapsto \varsigma^\BY_t (\varphi)
   := \frac{\mu^{\BY}_t (\varphi) }{
   \mu^{\BY}_t (1)}. 
\end{equation}
Note $\mu_t^\BY \in \mathcal{M}_F(\R^{d_X}), \, \varsigma_t^\BY \in
  \mathcal{P}(\R^{d_X})$, the space of finite (resp. probability) Borel measures. 

\begin{notation}
    Given $Y,\bar Y \in C^\alpha([0,T];\R^{d_Y})$ and $(g,g') \in \mathscr{D}_Y^{2\alpha} \mathcal{C}^\beta_b$, $(\bar g, \bar g') \in \mathscr{D}_{\bar Y}^{2\alpha} \mathcal{C}^\beta_b$ in the sense of \cref{def:stochasticcontrolledvectorfields_intro} for some $\beta = N + \kappa > 1$, $N \in \mathbb{N}_{\ge 0}$ and $\kappa \in (0,1]$,  we denote by  
\begin{equation*}
    \llbracket g,g'; \bar g,  \bar g' \rrbracket_{Y,\bar Y;2\alpha;\beta} := \sum_{i=0}^{N} \llbracket \delta D_x^i(g- \bar g) \rrbracket_\alpha + \sum_{j=0}^{N-1} ( \llbracket \delta D_x^j (g' - \bar g') \rrbracket_\alpha + \llbracket R^{D_x^j g} - \bar R^{D_x^j \bar g} \rrbracket_{2\alpha}  )
\end{equation*}
with $\bar R^{D_x^j \bar g}_{s,t}(x) :=D_x^j \bar g(t,x) - D_x^j\bar g(s,x) - D_x^j\bar g'(s,x) \delta \bar{Y}_{s,t}$. 
Similarly, given $(G,G') \in \mathbf{D}_Y^{2\alpha}L_{p} \mathrm{Lin}$ and $(\bar G, \bar G') \in \mathbf{D}_{\bar Y}^{2\alpha} L_{p} \mathrm{Lin}$ in the sense of \cref{def:stochasticcontrolledlinearvectorfield}, we denote by \begin{equation*}
    \llbracket G,G'; \bar G,  \bar G' \rrbracket_{Y,\bar Y;2\alpha;p}^{\mathrm{Lin}} :=  \sup_{0 \le s < t \le T} \Big( \frac{\| \delta (G - \bar G)_{s,t}\|_p^\mathrm{Lin}}{|t-s|^\alpha} + \frac{\| \delta (G' - \bar G')_{s,t}\|_p^\mathrm{Lin}}{|t-s|^\alpha} + \frac{\| R^G_{s,t} - \bar R^{\bar G}_{s,t}\|^\mathrm{Lin}_p}{|t-s|^{2\alpha}} \Big) 
\end{equation*}
with $\| \cdot \|_p^\mathrm{Lin} := \E ( |\cdot|_\mathrm{Lin}^p)^\frac1p$, $R^G_{s,t}x := G_t x - G_s x - G'_s x \delta Y_{s,t}$ and $\bar R^{\bar G}_{s,t}x := \bar G_t x - \bar G_s x - \bar G'_s x \delta \bar Y_{s,t}$.
\end{notation}

\begin{assumption}{R} \label{assumptionR}   
    There is a finite constant $C$ such that, for any $\BY, \bar \BY \in \mathscr{C}^{0,\alpha,1}$, 
     the following conditions hold:
    \begin{itemize}
        \item $\sup_{\mathbf{Z} \in \mathscr{C}^{0,\alpha,1} } \|b(\cdot, \cdot \, ; \mathbf{Z})\|_{Lip} < +\infty$ and $$\sup_{t \in [0,T] } \sup_{x \in  \R^{d_X}} |b(t,x;\BY) - b(t,x; \bar \BY)| \le C \rho_\alpha(\BY, \bar \BY),$$ with similar conditions for $\sigma$;
        \item $\sup_{\mathbf{Z} \in \mathscr{C}^{0,\alpha,1}}\|(f(\cdot,\cdot \, ; \mathbf{Z}) , f'(\cdot,\cdot \, ; \mathbf{Z})) \|_\beta + \llbracket f(\cdot,\cdot \, ; \mathbf{Z}) , f'(\cdot,\cdot \, ; \mathbf{Z}) \rrbracket_{Z;2\alpha;\beta} < +\infty$ and, denoting by $(f,f') = (f(\cdot,\cdot \, ; \BY), f'(\cdot,\cdot\, ; \BY)) $ and $(\bar f, \bar f') = (f(\cdot,\cdot \, ; \bar \BY),  f'(\cdot,\cdot\, ; \bar \BY)) $, 
        \begin{equation*}
            \|(f - \bar f , f' - \bar f' )\|_{\beta-1} + \llbracket f,f' ; \bar f, \bar f' \rrbracket_{Y, \bar Y; 2\alpha; \beta-1} \le C \rho_\alpha(\BY, \bar \BY) ;
        \end{equation*} 
        
        \item $[t \mapsto (h(t,\cdot \, ; \BY), h'(t,\cdot \, ;\BY) )] \in \mathscr{D}_Y^{2\alpha} \mathcal{C}^3_b$ with $$\sup_{\mathbf{Z} \in \mathscr{C}^{0,\alpha,1}}\|(h(\cdot,\cdot \, ; \mathbf{Z}) , h'(\cdot,\cdot \, ; \mathbf{Z})) \|_3 + \llbracket h(\cdot,\cdot \, ; \mathbf{Z}) , h'(\cdot,\cdot \, ; \mathbf{Z}) \rrbracket_{Z;2\alpha;3} < +\infty$$ and, denoting by $(h,h') = (h(\cdot,\cdot \, ; \BY), h'(\cdot,\cdot\, ; \BY)) $ and $(\bar h, \bar h') = (h(\cdot,\cdot \, ; \bar \BY),  h'(\cdot,\cdot\, ; \bar \BY)) $, 
        \begin{equation*}
            \|(h - \bar h , h' - \bar h' )\|_{1} + \llbracket h,h' ; \bar h, \bar h' \rrbracket_{Y, \bar Y; 2\alpha; 2} \le C \rho_\alpha(\BY, \bar \BY);
        \end{equation*} 
        \item $\| X^\BY_0 - X^{\bar \BY}_0 \|_2 \le C \rho_\alpha(\BY, \bar \BY)$. 
    \end{itemize}
\end{assumption}
\ 

\begin{lem} \label{lemma:Lipschitzestimates_RSDEs}
    Let \cref{assum:E} and \cref{assumptionR} be in force. Then, for any $R > 0$ and for any $\BY, \bar \BY \in (\mathscr{C }^{0, \alpha , 1}, \rho_\alpha)$ with $\nn{\BY}_\alpha \vee \nn{\bar \BY}_\alpha \le R$, there exists a constant $K = K_R > 0$ such that \begin{equation*}
        \sup_{t \in [0,T]} \|X_t^\BY - X_t^{\bar \BY} \|_2 \le K \rho_\alpha(\BY, \bar \BY) \qquad \text{and} \qquad \sup_{t \in [0,T]} \|Z_t^\BY - Z_t^{\bar \BY} \|_1 \le K \rho_\alpha(\BY, \bar \BY). 
    \end{equation*}
\end{lem}


\begin{proof} 
    The first bound follows as a straightforward application of the stability estimate in \cite[Theorem 4.9]{FHL21}, thanks to the assumptions on $b,\sigma$ and $(f,f')$. 
    For definiteness, note that \begin{equation*}
        \|X_t^\BY - X_t^{\bar \BY}\|_2 \le \| X^\BY_0 - X^{\bar \BY}_0 \|_2 + \| \sup_{t \in [0,T]} |\delta X^\BY_{0,t} - \delta X^{\bar \BY}_{0,t}| \|_2  . 
    \end{equation*}

    To deduce the second bound, we recall \eqref{eq:roughstochasticexponential} and  note that $|Z^\BY_t - Z^{\bar \BY}_t| = |e^{I_t^\BY} - e^{I^{\bar \BY}_t} | \le |I_t^\BY - I^{\bar \BY}_t| \max\{e^{I_t^\BY}, e^{I^{\bar \BY}_t} \}$. Hence, 
    $
        \|Z^\BY_t - Z^{\bar \BY}_t\|_1 \le \|I_t^\BY - I^{\bar \BY}_t\|_2  \|\max\{Z_t^\BY, Z^{\bar \BY}_t \}\|_2 $
    and the conclusion follows combining local Lipschitz continuity of rough stochastic integrals (see \cite[Corollary 3.5]{FHL21}) and stability for stochastic controlled rough paths (see \cite[Proposition 3.13]{FHL21}). 
\end{proof}

\begin{theorem} Under \cref{assum:E} and \cref{assumptionR} the following statements hold, uniformly in $t \in [0,T]$: \begin{enumerate} 
    \item[(i.1)] for any fixed bounded and Lipschitz function $\varphi: \R^{d_X} \to \R$ , the map $\BY \mapsto \mu_t^\BY(\varphi) $ is locally Lipschitz on  $(\mathscr{C }^{0, \alpha , 1},\rho_\alpha)$; 
    \item[(i.2)] for any fixed bounded and Lipschitz function $\varphi: \R^{d_X} \to \R$ , the map $\BY \mapsto \varsigma_t^\BY(\varphi) $ is locally Lipschitz on  $(\mathscr{C }^{0, \alpha , 1},\rho_\alpha)$;
  \item[(ii.1)] the map $\BY \mapsto \mu_t^\BY \in
  \mathcal{M}_F(\R^{d_X})$ is continuous  on  $(\mathscr{C }^{0, \alpha , 1},\rho_\alpha)$, if $\mathcal{M}_F(\R^{d_X})$ is equipped with topology of weak convergence; 
  \item[(ii.2)] the map $\BY \mapsto \varsigma_t^\BY \in
  \mathcal{P}(\R^{d_X})$ is continuous  on  $(\mathscr{C }^{0, \alpha , 1},\rho_\alpha)$, if $\mathcal{P}(\R^{d_X})$ is equipped with topology of weak convergence. 
\end{enumerate}
\end{theorem}

\begin{proof}
Let $\BY, \bar \BY \in \mathscr{C}^{0,\alpha,1}$ and let $t \in [0,T]$. \begin{itemize}
    \item[\textit{(i.1)}] Let $\varphi$ be a bounded Lipschitz function. Denote by $|\varphi|_{Lip}$ its Lipschitz constant and let $|\varphi|_\infty := \sup_{x \in \R^{d_X}} |\varphi(x)| <+\infty$. A trivial application of H\"older's inequality shows that
    \begin{align*}
        |\mu_t^\BY(\varphi) - \mu_t^{\bar \BY}(\varphi)| &\le \E(|\varphi(X_t^\BY) - \varphi(X_t^{\bar \BY})| Z^\BY_t) + \E(|\varphi(X_t^{\bar \BY})| |Z^\BY_t - Z^{\bar\BY}_t|) \\
        &\le |\varphi|_{Lip} \|X_t^\BY - X_t^{\bar \BY}\|_2 \|Z^\BY_t\|_2 + |\varphi|_\infty \|Z^\BY_t - Z^{\bar\BY}_t\|_1
    \end{align*} 
    and the conclusion follows from \cref{lemma:Lipschitzestimates_RSDEs}. 
    \item[\textit{(i.2)}] Note that $\mu_t^\BY(1) = \E(e^{I_t^\BY}) >0$. By \textit{(i.1)}, $\BY \mapsto \mu_t^\BY(1)$ is continuous and therefore it is locally uniformly bounded away from zero. The conclusion follows by recalling \eqref{eq:roughKSform} and noting that \begin{equation*}
        |\varsigma_t^\BY(\varphi) - \varsigma_t^{\bar \BY}(\varphi)| \le \frac{1}{|\mu_t^\BY(1)|} |\mu_t^\BY(\varphi) - \mu_t^{\bar \BY}(\varphi)| + \frac{|\mu_t^{\bar \BY}(\varphi)|}{|\mu_t^\BY(1)| |\mu_t^{\bar \BY}(1)|} |\mu_t^\BY(1) - \mu_t^{\bar \BY}(1)|.
    \end{equation*}
    \item[\textit{(ii.1)}] Let $(\BY^n)_n$ be a sequence in $\mathscr{C}^{0,\alpha,1}$ such that $\rho_\alpha(\BY^n, \BY) \to 0$ as $n \to +\infty$. Let $\varphi \in C^0_b(\R^{d_X};\R)$ be arbitrary. 
    By the continuous mapping theorem and dominated convergence, $\varphi(X_t^{\BY^n}) \to \varphi(X_t^\BY)$ in $L_2(\Omega;\R)$.
    Hence, \begin{align*}
        |\mu_t^{\BY^n}(\varphi) - \mu_t^\BY(\varphi)| &\le \E(|\varphi(X^{\BY^n}_t) - \varphi(X^{\BY}_t)| Z^\BY_t) + \E(|\varphi(X^{\BY^n}_t)| |Z^{\BY^n}_t - Z^\BY_t|) \\
        &\le  \|\varphi(X^{\BY^n}_t) - \varphi(X^{\BY}_t)\|_2 \|Z^\BY_t\|_2 + |\varphi|_\infty \|Z^{\BY^n}_t - Z^\BY_t\|_1 \to 0 
    \end{align*}
    as $n \to +\infty$. 
    \item[\textit{(ii.2)}] The continuity of $\BY \mapsto \varsigma^\BY_t(\varphi)$ for any continuous and bounded $\varphi$ follows from \textit{(ii.2)} and from the fact that $\BY \mapsto \mu_t^\BY(1)$ is continuous and positive.
\end{itemize}
\end{proof}

\subsection{Rough Zakai equation in spaces of measures: existence} \label{sec:roughZakai_existence}

We show that $\mu^\BY$ defined in \eqref{eq:roughunnormfilter}
is the (unique) solution to the rough Zakai equation
\begin{equation*}
        \mu^{\BY}_t (\varphi)   =  {\mu _0}^{\BY} (\varphi) +  \int_0^t \mu^{\BY}_r  (A^{\BY}_r \varphi) dr + \int_0^t  \mu^{\BY}_r (\Gamma^{\BY}_r \varphi)  \, d \BY_r , \qquad 
        \mu_0^\BY = \text{Law}(X_0^\BY)
\end{equation*}
where
\begin{equation} \label{eq:operatorA^Y}
    A^{\BY}_t \varphi (x) = \sum_{i = 1}^{d_X}
   (\bar{b}^{[\BY]})^i (t, x; \BY) \partial_i \varphi (x) + \frac{1}{2}   \sum_{i, j = 1}^{d_X} (\bar{a}^{[\BY]})^{i j} (t, x; \BY)   \partial^2_{i j} \varphi (x),
\end{equation}
with
\begin{align*}
  (\bar{a}^{[\BY]})^{ij} (t,x;\BY) &: = \sum_{\theta = 1}^{d_B} \sigma^i_\theta
  \sigma_\theta^j  (t,x;\BY) + \sum_{\kappa, \lambda = 1}^{d_Y} f^i_{\kappa}  (t,  x; \BY)  \dot{[\BY]}_t^{\kappa \lambda} f^j_{\lambda}  (t,x;\BY), \\
  (\bar{b}^{[\BY]})^i (t,x;\BY) &:= b^i (t,x;\BY) +
   \sum_{\kappa, \lambda = 1}^{d_Y} f^i_{\kappa} (t,x;\BY)
   \dot{[\BY]}_t^{\kappa \lambda} h^{\lambda} (t,x;\BY),
\end{align*}
and 
$\Gamma^{\BY}_t : =
((\Gamma^{\BY}_t)_1, \ldots, (\Gamma^{\BY}_t)_{d_Y})$ with
\begin{equation} \label{eq:operatorGamma^Y}
    (\Gamma^{\BY}_t)_{\kappa} \varphi (x) = \sum_{i = 1}^{d_X}
   f^i_{\kappa}  (t,x;\BY)  \partial_i \varphi (x) +  h _{\kappa}   (t,x;\BY) \varphi (x) , \qquad \kappa = 1, ..., d_Y .
\end{equation} 
Formally, the operator families $A^\BY_t$, $\Gamma^{\BY}_t$ are  as in the stochastic case (i.e. \eqref{eq:operatorA} and \eqref{eq:stochGamma}), except that $Y =\pi_1 \BY$ and $\langle Y, Y \rangle = [\BY]$ are now deterministic, in contrast to the stochastic case where $Y = Y (\omega)$ was random. 

\noindent The following approximations (cf.\ \cref{assum:E} for a rigorous meaning) \begin{align*}
    f^i_\kappa (t,x; \BY) - f^i_\kappa(s,x;\BY) &\approx \sum_{\eta=1}^{d_Y} (f'_{\kappa \eta})^i (s,x;\BY) \delta Y^\eta_{s,t} \qquad i=1,\dots,d_X \\
    h_\kappa(t,x;\BY) - h_\kappa(s,x;\BY) & \approx \sum_{\eta=1}^{d_Y} h'_{\kappa \eta} (s,x;\BY) \delta Y^\eta_{s,t}
\end{align*}
justify the definition of
\[
(\Gamma^{\BY}_t)' \assign 
\left(
  \begin{array}{cccc}
    (\Gamma^{\BY}_t)_{11}' & \cdots  & (\Gamma^{\BY}_t)'_{1 d_Y} \\
    \vdots & \ddots &         \vdots \\
    (\Gamma^{\BY}_t)_{d_Y 1}'  & \cdots & (\Gamma^{\BY}_t)'_{d_Y d_Y}
  \end{array}
\right)
\]
with
\[
    (\Gamma^{\BY}_t)_{\kappa \lambda}' \varphi (x)  :=  \sum_{i  = 1}^{d_X}  (f_{\kappa \lambda}')^i (t, x; \BY) \partial_i \varphi (x) +  h_{\kappa  \lambda}' (t, x;\BY) \varphi (x) .
\]

\begin{definition}[Rough Zakai] \label{def:solutionroughZakai}
  A measure-valued solution to
  \begin{equation} \label{eq:roughZakai_CITE}
    \mu_t (\varphi) = \mu _0 (\varphi) + \int_0^t \mu_r (A^\BY_r \varphi) \, dr + \int_0^t \mu_r (\Gamma^{\BY}_r \varphi) \,  d \BY_r 
  \end{equation}
  is a weakly continuous \footnote{The curve $\mu:[0,T] \to \mathcal{M}(\R^{d_X})$ is weakly continuous if, for any $\varphi \in \mathcal{C}^0_b(\R^{d_X};\R)$, the map $[0,T] \ni t \mapsto \mu_t(\varphi) = \int_{\R^{d_X}} \varphi(x) \, \mu_t(dx)  \in \R$ is continuous.} curve $\mu : [0, T] \rightarrow
  \mathcal{M}_F (\mathbb{R}^{d_X})$ such that, for any test function $\varphi \in \mathcal{C}^3_b (\mathbb{R}^{d_X}
  ; \mathbb{R})$, the
  following properties are satisfied: 
  \begin{enumerate}[(i)]
    \item the map $t \mapsto \mu_t  (A^\BY_t \varphi)$ 
    belongs to $L_1  ([0, T], dt)$;
    
    \item
    there is $C = C_{\varphi} > 0$ uniform over bounded sets of
    $\varphi$ in $\mathcal{C}^3_b$ such that, for any $0 \le s \le t
    \le T$ and for any $\kappa, \lambda = 1, \ldots, d_Y$,    
    \begin{align*}
      & | \mu_t ((\Gamma_t)_{\kappa}^{\BY}
      (\Gamma_t)_{\lambda}^{\BY} \varphi +
      (\Gamma_t^{\BY})_{\lambda \kappa}' \varphi) - \mu_s
      ((\Gamma_s)_{\kappa}^{\BY} (\Gamma_s)_{\lambda}^{\BY}
      \varphi + (\Gamma_s^{\BY})_{\lambda \kappa}' \varphi) | \le
      C_{\varphi}  |t - s|^{\alpha}\\
      & | \mu_t  ((\Gamma_t)_{\kappa}^{\BY} \varphi) - \mu_s 
      ((\Gamma_s)_{\kappa}^{\BY} \varphi) - \sum_{\eta = 1}^{d_Y}
      \mu_s  ((\Gamma_s)_{\eta}^{\BY}
      (\Gamma_s)_{\kappa}^{\BY} \varphi +
      (\Gamma_s^{\BY})_{\kappa \eta}' \varphi) \delta Y_{s, t}^{\eta}
      | \le C_{\varphi}  |t - s|^{2 \alpha} ;
    \end{align*}
    \item
    the following Davie-type expansion holds for any $0 \le s \le t \le
    T$, 
    \begin{multline*} \label{eq:Davieexpansion_roughZakai}
    \mu_t (\varphi) - \mu_s (\varphi)  =  \int_s^t 
    \mu_r
      (A^\BY_r \varphi) \,   dr \\ + \sum_{\kappa = 1}^{d_Y} \mu_s 
      ((\Gamma^{\BY}_s)_{\kappa} \varphi) \delta Y^{\kappa}_{s, t} 
       + \sum_{\kappa, \lambda = 1}^{d_Y} \mu_s 
      (((\Gamma^{\BY}_s)_{\kappa} (\Gamma_s^{\BY})_{\lambda}
      + (\Gamma_s^{\BY})'_{\lambda \kappa}) \varphi) \mathbb{Y}_{s,
      t}^{\kappa \lambda} + \mu_{s, t}^{\varphi, \natural} 
    \end{multline*}
    where $\mu_{s, t}^{\varphi, \natural} = o (| t - s |)$ as $| t - s |
    \rightarrow 0$.
  \end{enumerate}
  If the initial value $\mu_0 = \nu \in \mathcal{M}_F (\mathbb{R}^{d_X})$ is
  specified, we say that $\mu$ is a solution starting from $\nu$.
\end{definition}

\begin{proposition} \label{prop:roughZakai_integralform}
Let $\mu : [0, T] \to \mathcal{M}_F (\mathbb{R}^{d_X})$ satisfy condition (ii) of the previous definition. Then the following limit exists for any $t
  \in [0, T]$ along any sequence of partitions $\pi$ of $[0, t]$ whose mesh
  tends to zero:
  \begin{multline*}
    \int_0^t \mu_r (\Gamma^{\BY}_r \varphi) d \mathbf{Y}_r \assign \lim_{| \pi | \to
    0}  \sum_{[s, u] \in \pi} (\mu_s ((\Gamma^{\BY}_s)_{\kappa} \varphi) \nobracket
    \delta Y^{\kappa}_{s, u} + \nobracket \mu_s ((\Gamma^{\BY}_s)_{\kappa}
    (\Gamma^{\BY}_s)_{\lambda} \varphi + ((\Gamma^{\BY}_s)')_{\lambda \kappa} \varphi)
    \mathbb{Y}^{\kappa \lambda}_{s, u}) .
  \end{multline*}
In particular, condition $(iii)$ of the previous definition can be equivalently
  replaced by
  \begin{itemize}
    \item[(iii)'] for any $t \in [0, T]$
    \[ \mu_t (\varphi) = \mu_0 (\varphi) + \int_0^t   \mu_r
      (A^\BY_r \varphi) \, dr + \int_0^t \mu_r (\Gamma_r^\BY \varphi) \, d\BY_r,\footnote{The middle integral on the right-hand side is a Lebesgue integral with $L_1$-integrand, cf. (i), the far right integral is a bona fide rough integral, with controlled integrands, cf. (ii).}
    \]
  \end{itemize}
  and the quantity $\mu_{s, t}^{\varphi,\natural}$ in fact satisfies $|
  \mu_{s, t}^{\varphi,\natural} | \le K_{\varphi}  |t - s|^{3 \alpha}$ for
  any $s, t \in [0, T]$, where $K_{\varphi} > 0$ is uniformly over bounded
  sets of $\varphi$ in $\mathcal{C}^3_b (\mathbb{R}^d ; \mathbb{R})$. 
\end{proposition}

\begin{proof}
  The proof is essentially a consequence of the sewing lemma (see \cite[Lemma 4.2]{FH20}). Let $t \in [0, T]$. For any $\varphi \in \mathcal{C}^3_b (\mathbb{R}^d ;
  \mathbb{R})$ and for any $s, u \in
  [0, t]$ define
  \[ \Xi_{s, u}^{\varphi} \assign \langle \mu_s, (\Gamma_s)_{\kappa} \varphi
     \rangle \delta Y^{\kappa}_{s, u} + \langle \mu_s, ((\Gamma_s)_\kappa 
     (\Gamma_s)_{\lambda} + (\Gamma'_s)_{\lambda \kappa}) \varphi \rangle
     \mathbb{Y}^{\kappa \lambda}_{s, u} . \]
  Recalling Chen's relation for rough paths, for any $s, u, v \in [0, t]$ we
  have
  \begin{align*}
    \Xi_{s, v}^{\varphi} - \Xi_{s, u}^{\varphi} - \Xi_{u, v}^{\varphi} & =
    (\langle \mu_s, (\Gamma_s)_{\kappa} \varphi \rangle - \langle \mu_u,
    (\Gamma_u)_{\kappa} \varphi \rangle) \delta Y^{\kappa}_{u, v} \\
    & \quad + \langle (\mu_s, (\Gamma_s)_{\kappa}  (\Gamma_s)_\lambda +
    (\Gamma'_s)_{\lambda \kappa}) \varphi \rangle  (\mathbb{Y}^{\kappa
    \lambda}_{u, v} + \delta Y^{\kappa}_{s, u} \delta Y^{\lambda}_{u, v}) \\
    & \quad - \langle \mu_u, ((\Gamma_u)_{\kappa}  (\Gamma_u)_{\lambda} +
    (\Gamma'_u)_{\lambda \kappa}) \varphi \rangle \mathbb{Y}^{\kappa
    \lambda}_{u, v} .
  \end{align*}
  From condition (ii) of we deduce that $| \Xi_{s, v}^{\varphi} - \Xi_{s,
  u}^{\varphi} - \Xi_{u, v}^{\varphi} | \le C_{\varphi}  | \BY
  |_{\alpha} |t - s|^{3 \alpha}$ with $C_{\varphi}$ uniform over bounded sets
  of $\varphi \in \mathcal{C}^3_b (\mathbb{R}^d ; \mathbb{R})$. 
\end{proof}

\begin{theorem}[Existence] \label{thm:existence}
  Under \cref{assum:E}, the path  $[0, T] \ni t \mapsto \mu_t^{\mathbf{Y}} \in \mathcal{M}_F(\mathbb{R}^{d_X})$ defined by \eqref{eq:roughunnormfilter}  is a measure-valued solution to the rough Zakai equation \eqref{eq:roughZakai_CITE} starting from $\mu^\BY_0 = \text{Law}(X_0^\BY)$.
  Moreover, for any $R>0$ and for any $\theta \in (0,1]$, \begin{equation} \label{eq:condition1}
      \sup_{\substack{\phi \in \mathcal{C}^\theta_b (\R^{d_X};\R) : \\ |\phi|_{\mathcal{C}^{\theta}_b} \le R}} \ \sup_{0 \leqslant s < t \leqslant T} \frac{|\mu^{\BY}_t (\phi) - \mu^{\BY}_s (\phi)|}{|t-s|^{\theta\alpha}} < + \infty.
  \end{equation}    
\end{theorem}

\begin{proof}
    The proof is essentially an application of the rough stochastic It\^o formula \cref{eq:roughItoformula_roughItoprocess} (part 1 and 2). 
    Let 
    $\varphi \in \mathcal{C}^3_b(\mathbb{R}^{d_X};\mathbb{R})$ be arbitrary.

    For sake of simplicity, let us first consider the case $h = 0$ (i.e.\ $Z^{\BY}_t = I$ for any $t$).
    Notice that, by definition of $L_{2,\infty}$-integrable solutions (e.g. \cref{def:integrablesolutionsRSDEs}), $X^{\BY}$ a.s.-continuous and $\mu^{\BY}$ is therefore weakly continuous.
    By construction, the map $[0,T] \ni t \mapsto \mu^{\BY}_t(A^\BY_t \varphi) = \mathbb{E}(A^\BY_t \varphi (X_t^{\BY}) )$ 
    is measurable and uniformly bounded. That is, condition \textit{(i)}\ of \cref{def:solutionroughZakai} is satisfied. 
    Recall that $(X^{\BY}_\cdot, f(\cdot \, ,X^{\BY};\BY))$ is a stochastic controlled rough path in $\mathbf{D}_Y^{2\alpha} L_p$ for any $p \in [2,\infty)$. 
    Hence, for any $\kappa,\lambda=1,\dots,d_Y$ and for any $0 \le s \le t \le T$, we can write 
    \begin{align*}
        &\sum_{i,j=1}^{d_X} ( \partial^2_{ij} \varphi(X^{\BY}_t) f^i_\kappa(t,X^{\BY}_t;\BY) f^j_\lambda(t,X^{\BY}_t;\BY) + \partial_i \varphi(X^{\BY}_t) \partial_{x^j} f_\kappa^i(t,X^{\BY}_t;\BY) f_\lambda^j(t,X^{\BY}_t;\BY)   \\
        & \quad  + \sum_{i=1}^{d_X} \partial_i \varphi(X^{\BY}_t) (f'_{\kappa\lambda})^i(t,X^{\BY}_t;\BY)  \\
        & = \sum_{i,j=1}^{d_X} ( \partial^2_{ij} \varphi(X^{\BY}_s) f^i_\kappa(s,X^{\BY}_s;\BY) f^j_\lambda(s,X^{\BY}_s;\BY) + \partial_i \varphi(X^{\BY}_s) \partial_{x^j} f_\kappa^i(s,X^{\BY}_s;\BY) f_\lambda^j(s,X^{\BY}_s;\BY)    \\
        & \quad  + \sum_{i=1}^{d_X} \partial_i \varphi(X^{\BY}_s) (f'_{\kappa\lambda})^i(s,X^{\BY}_s;\BY) + Q^\varphi_{s,t}
    \end{align*} and 
    \begin{align*}
        & \sum_{i=1}^{d_X} (\partial_i \varphi (X^{\BY}_t) f_\kappa^i (t,X_t^{\BY};\BY) - \partial_i \varphi (X^{\BY}_s) f_\kappa^i (s,X_s^{\BY};\BY))   \\
        & = \sum_{i=1}^{d_X} (\partial_i \varphi (X^{\BY}_t) - \partial_i \varphi (X^{\BY}_s)) f_\kappa^i (s,X_s^{\BY};\BY) + \sum_{i=1}^{d_X} \partial_i \varphi (X^{\BY}_s) (f_\kappa^i (t,X_t^{\BY};\BY) - f_\kappa^i (s,X_s^{\BY};\BY))  \\
        & = \sum_{i,j=1}^{d_X} \partial^2_{ij} \varphi (X^{\BY}_s) \Big( \sum_{\eta=1}^{d_Y} f_\eta^j(s,X^{\BY}_s;\BY) \delta Y^\eta_{s,t} \Big) f_\kappa^i (s,X_s^{\BY};\BY) \\
        &  \quad + \sum_{i=1}^{d_X} \partial_i \varphi (X^{\BY}_s) \sum_{\eta=1}^{d_Y} \Big( \sum_{j=1}^{d_X} \partial_{x^j} f_\kappa^i (s,X_s^{\BY};\BY) f_\eta^j(s,X^{\BY}_s;\BY) + (f'_{\kappa \eta})^i(s,X^{\BY}_s;\BY) \Big) \delta Y^\eta_{s,t} + R^{\varphi}_{s,t}
    \end{align*}
    with $| \E(Q^\varphi_{s,t})| \le \|Q^\varphi_{s,t}\|_2 \le c_\varphi |t-s|^{\alpha}$ and $ |\E( R^\varphi_{s,t})| \le  \| \mathbb{E}( R^\varphi_{s,t} \mid \mathcal{F}_s ) \|_2 \le c_\varphi |t-s|^{2\alpha}$ for some constant $c_\varphi>0$, which is uniform over bounded sets in $\mathcal{C}^3_b(\R^{d_X};\R)$. 
    Taking expectation on both sides of the previous two identities, we have that condition \textit{(ii)} of \cref{def:solutionroughZakai} is satisfied. 
    We can apply the rough stochastic It\^o formula to the process $X^\BY$ in \eqref{eq:roughSDE_filtering} and we deduce that, for any $0 \le s \le t \le T$,  
    \begin{equation} \label{eq:roughItoformula_existence} \begin{split}
        &\varphi(X^{\BY}_t) - \varphi(X^{\BY}_s)  \\
      &=  \int_s^t \sum_{i=1}^{d_X} \partial_{i} \varphi(X^{\BY}_r) b^i (r,X_r^{\BY};\BY)  \\
      & \quad + \frac{1}{2} \sum_{i,j=1}^{d_X} \partial^2_{ij} \varphi (X^{\BY}_r) \Big( \sum_{\theta=1}^{d_B} \sigma_\theta^i(r,X_r^{\BY};\BY) \sigma_\theta^j(r,X_r^{\BY};\BY) + \sum_{\kappa, \lambda = 1}^{d_Y} f^i_\kappa(r,X_r^\BY;\BY)  f^j_\lambda(r,X_r^\BY;\BY) \dot{[\BY]}_r^{\kappa \lambda} \Big)  \, dr \\
      & \quad +  \sum_{\theta=1}^{d_B} \int_s^t  \sum_{i=1}^{d_X} \partial_{i} \varphi(X^{\BY}_r) \sigma_\theta^i (r,X_r^{\BY};\BY) \, dB^{\theta}_r +
      \sum_{\kappa=1}^{d_Y} \sum_{i=1}^{d_X}  \partial_{i} \varphi(X^{\BY}_s) f^i_\kappa(s,X_s^{\BY};\BY) \delta Y_{s,t}^\kappa \\
      & \quad + \sum_{\kappa, \lambda=1}^{d_Y} \sum_{i=1}^{d_X} \Big( \sum_{j=1}^{d_X} ( \partial^2_{ij} \varphi (X^{\BY}_s) f_\lambda^i(s,X_s^{\BY};\BY) f_\kappa^j(s,X_s^{\BY};\BY)  \\
      & \quad + \partial_i \varphi(X_s^{\BY})  \partial_{x^j} f_\lambda^i(s,X_s^{\BY};\BY) f_\kappa^j(s,X_s^{\BY};\BY)  ) +  \partial_i \varphi(X_s^{\BY}) (f_{\lambda\kappa}')^i (s,X_s^{\BY};\BY)  \Big)  \mathbb{Y}_{s,t}^{\kappa \lambda} + \varphi^\natural_{s,t}
    \end{split}
    \end{equation}
  with $|\mathbb{E}(\varphi^\natural_{s,t})| \le \|\mathbb{E}(\varphi^\natural_{s,t} \mid \mathcal{F}_s)\|_2 \le c_\varphi |t-s|^{3\alpha}$. 
  Hence, by taking the expectation on both sides of \eqref{eq:roughItoformula_existence} and by the martingale property of the It\^o integral, we conclude that condition \textit{(iii)} of \cref{def:solutionroughZakai} is also satisfied.

  The proof when $h$ is not zero follows as a generalization of the previous argument.
  Indeed, recall that by construction $(Z^\BY, Z^\BY h(\cdot \, , X^\BY; \BY)) \in \bigcap_{p \in [1,\infty)} \mathbf{D}_Y^{2\alpha} L_p(\mathbb{R}) $. 
  By the product formula (cf.\ \cref{eq:roughItoformula_roughItoprocess} part 2 and 3), we conclude that, for any $0 \le s \le t \le T$, 
  \begin{align}
   & \varphi (X_t^{\BY}) Z^\BY_t  -  \varphi (X_s^{\BY}) Z_s^\BY \notag \\
    &   = \int_s^t A_r^{\BY} \varphi 
    (X_r^{\BY}) Z^\BY_r 
    \, dr   + \sum_{\theta=1}^{d_B} \int_s^t \sum_{j=1}^{d_X} \partial_j \varphi    (X_r^{\BY}) \sigma_\theta^j (r,X_r^{\BY};\BY) Z^\BY_r \, dB^\theta_r    \label{e:2nd} \\
    &  \quad  + \sum_{\kappa=1}^{d_Y} (\Gamma_s^{\BY})_\kappa \varphi     (X_s^{\BY}) Z^\BY_s \delta Y^\kappa_{s,t}  + \sum_{\kappa, \lambda=1}^{d_Y} \Big( (\Gamma^{\BY}_s)_\kappa (\Gamma^{\BY}_s)_\lambda \varphi  (X_s^{\BY})    +    (\Gamma^{\BY}_s)'_{\lambda \kappa} \varphi  (X_s^{\BY})    \Big) Z^\BY_s  \mathbb{Y}_{s,t}^{\kappa \lambda} + \varphi^\natural_{s,t} \notag
  \end{align}
  with $|\mathbb{E} (\varphi^\natural_{s,t})| \le \| \E(\varphi^\natural_{s,t} \mid \mathcal{F}_s)\|_2 \le {c_\varphi} |t-s|^{3\alpha}$ for some constant $c_\varphi>0$, which is uniform over bounded sets in $\mathcal{C}^3_b(\R^{d_X};\R)$. 
  The conclusion analogously follows by taking expectation on both sides, recalling that $\mu^{\BY}_t(\varphi) = \mathbb{E} (\varphi(X^{\BY}_t) {Z_t^{\BY}} )$.

  Finally, it is almost immediate to see that, for any $\theta \in (0,1]$ and for any $\phi \in \mathcal{C}^\theta_b(\R^{d_X};\R)$, the map $[0,T] \ni t \mapsto \phi(X_t^{\BY}) {Z_t^{\BY}} \in L_1(\Omega)$ is $(\theta\alpha)$-H\"older continuous. 
  It follows that, up to constants not depending on $\phi$, \begin{equation*}
      |\mu^{\BY}_t(\phi) - \mu^{\BY}_s(\phi)| \lesssim |\phi|_{\mathcal{C}^\theta_b}  |t-s|^{\theta\alpha},
  \end{equation*}  
  which yields \eqref{eq:condition1}.
\end{proof}

\subsection{Uniqueness} \label{section:Zakai_uniqueness}
\begin{definition} \label{def:solutionbackwardroughPDE}
    A function-valued solution to the backward equation \begin{equation} \label{eq:dualroughequation}
        u_t (x) = u_T(x) +\int_t^T \left( A_r^{\BY} u_r (x) - (\Gamma_r^\BY)'u_r(x) \dot{[\BY]}_r \right)  dr + \int_t^T \Gamma^{\BY}_r u_r(x) \, d\BY_r 
    \end{equation}
    on the time interval $[0,T]$ is a map $u: [0, T] \times \mathbb{R}^{d_X}  \to \mathbb{R}$ satisfying the following properties: 
  \begin{enumerate}[(i)]
    \item
    for any $x \in \mathbb{R}^{d_X}$, the function $t \mapsto A_t^{\BY} u_t(x) - \sum_{\kappa=1}^{d_Y} (\Gamma_t^\BY)'_{\lambda \kappa} u_t(x) \dot{[\BY]}_t^{\kappa\lambda}$ 
    belongs to $L_1([0,T],dt)$;
    \item
    there is a constant $C > 0$ such that, for any $0 \le s \le t \le T$ and for any $\kappa,\lambda=1,\dots,d_Y$,
    \begin{align*}
      \sup_{x \in \mathbb{R}^d} \left| (\Gamma^{\BY}_s)_\kappa
      (\Gamma^{\BY}_s)_\lambda u_s (x) - (\Gamma^{\BY}_s )'_{\kappa \lambda} u_s (x) 
      - \big((\Gamma^{\BY}_t)_\kappa (\Gamma^{\BY}_t)_\lambda u_t(x) -
      (\Gamma^{\BY}_t )_{\kappa \lambda}' u_t (x) \big) \right| & \le C | t
      - s |^{\alpha}\\
      \sup_{x \in \mathbb{R}^d} | (\Gamma^{\BY}_s)_\kappa u_s (x) -
      (\Gamma^{\BY}_t)_\kappa u_t (x) -  \sum_{\eta=1}^{d_Y} \big( (\Gamma^{\BY}_t)_\kappa 
      (\Gamma^{\BY}_t)_\eta u_t(x) - (\Gamma^{\BY}_t )'_{\kappa \eta} u_t (x) \big) \delta Y_{s,t}^\eta     | & \le C | t - s |^{2 \alpha}
    \end{align*}
    \item
    for any $0 \le s \le t \le T$ and for any $x \in \mathbb{R}^{d_X}$, the following Davie-type expansion holds \begin{multline*} \label{eq:Davieexpansion_roughKolmogorov}
          u_s(x) - u_t(x) = \int_s^t \Big( A^{\BY}_r u_r(x) - \sum_{\kappa,\lambda = 1}^{d_Y} (\Gamma^\BY_r)'_{\lambda\kappa} u_r(x) \dot{[\BY]}_r^{\kappa\lambda} 
          \Big) \, d r \\ + \sum_{\kappa=1}^{d_Y} (\Gamma^{\BY}_t)_\kappa u_t(x) \delta Y^\kappa_{s,t}  + \sum_{\kappa, \lambda =1}^{d_Y}  \big( (\Gamma^{\BY}_t)_\kappa (\Gamma^{\BY}_t)_\lambda u_t (x) - (\Gamma^{\BY}_t )'_{\kappa \lambda} u_t (x) \big)    \mathbb{Y}_{s,t}^{\kappa \lambda}+ u^\natural_{s,t}(x)
      \end{multline*}
      with $\sup_{x \in \mathbb{R}^{d_X}} \frac{|u^\natural_{s,t}(x)|}{|t-s|^{3\alpha}} < +\infty$.
  \end{enumerate}
  When the terminal condition $u_T(\cdot) = g: \mathbb{R}^{d_X} \to \mathbb{R}$ is specified, we say that $u$ is a solution with terminal datum $g$. 
\end{definition}
This kind of backward rough PDEs has been studied in \cite[Section 4]{BFS24}, where the unique solution is shown to admit a Feynman--Kac type representation under the following  \cref{assumptionU} (which also implies \cref{assum:E}).

\begin{assumption}{U} \label{assumptionU}
Let $\theta \in [4,5]$ such that $2\alpha + (\theta-4)\alpha >1$. 
Let the following hold: \begin{itemize}
    \item $ [t \mapsto b (t, \cdot \, ; \BY)] \in \mathcal{B}_b([0,T];\mathcal{C}^3_b(\R^{d_X};\R^{d_X}))$ and $$t \mapsto |b(t,\cdot \, ; \BY)|_\infty + |D_x b(t,\cdot \, ; \BY)|_\infty + |D^2_{xx} b(t,\cdot \, ; \BY)|_\infty \quad \text{is continuous};$$ 
    \item $[ t \mapsto \sigma(t,\cdot \, ; \BY)]  \in \mathcal{B}_b([0,T];\mathcal{C}^3_b(\R^{d_X};\mathrm{Lin}(\mathbb{R}^{d_B}, \mathbb{R}^{d_X}))) $ and $$ t \mapsto |\sigma(t,\cdot \, ; \BY)|_\infty + |D_x \sigma(t,\cdot \, ; \BY)|_\infty + |D^2_{xx} \sigma
    (t,\cdot \, ; \BY)|_\infty \quad \text{is continuous}; $$
    \item 
    $[t \mapsto (f (t, \cdot \, ;\BY),  f' (t, \cdot \, ; \BY))] \in \mathscr{D}_{Y}^{2 \alpha} \mathcal{C}_b^\theta$  
    in the sense of \cref{def:stochasticcontrolledvectorfields_intro} ; 
    \item 
    $[t \mapsto (h(t,\cdot \, ;\BY), h' (t,\cdot \, ; \BY))] \in \mathscr{D}_{Y}^{2 \alpha} \mathcal{C}_b^\theta$ .
\end{itemize}
\end{assumption}

The following result follows from \cite[Proposition 4.8 and Theorem 4.9]{BFS24}. For a more detailed discussion see \cref{appendix}. 

\begin{theorem} 
    For any final datum $ g \in \mathcal{C}^3_b(\R^{d_X};\R)$, there is a unique function-valued solution $u$ to equation \eqref{eq:dualroughequation} such that \footnote{
    We denote by $C^{0, 2} ([0, T] \times \mathbb{R}^{d_X} ; \mathbb{R})$ the space of functions $u=u(t,x)$ which are twice continuously differentiable in $x$ and such that $u,D_x u, D^2_{xx}$ are jointly continuous in $(t,x)$. 
    } 
    \begin{equation*}
    u \in C^{0, 2} ([0, T] \times \mathbb{R}^{d_X} ; \mathbb{R}) \cap \mathcal{B}_b([0,T]; \mathcal{C}^3_b (\mathbb{R}^{d_X};\mathbb{R}))
\end{equation*}
and the map $t \mapsto (|u(t,\cdot)|_\infty + |D_xu(t,\cdot)|_\infty + |D^2_{xx}u(t,\cdot)|_\infty)$ is $\alpha$-H\"older continuous. 
\end{theorem}

\begin{theorem}[Uniqueness] \label{thm:uniqueness}
  Under \cref{assumptionU}  and for any $\nu \in \mathcal{M}_F (\mathbb{R}^{d_X})$,  there is at most one solution $\mu$ to \eqref{eq:roughZakai_CITE} starting from $\mu_0=\nu$ such that 
  the following holds, for any $R>0$:
  \begin{equation} \label{eq:condition_uniqueness}
  \sup_{\substack{\phi \in \mathcal{C}^1_b(\R^{d_X};\R) : \\ |\phi|_{\mathcal{C}^1_b} \le R}} \ \sup_{0 \le s < t \le T} \frac{|\mu_t (\phi) - \mu_s (\phi)|}{|t-s|^{\alpha}} < + \infty .
  \end{equation}
\end{theorem}

\begin{proof}
Let $\mu^1, \mu^2$ be two solutions of \eqref{eq:roughZakai_CITE}, both satisfying \eqref{eq:condition_uniqueness}.
  Let $\tau \in [0, T]$ and $\varphi \in
  C_c^{\infty} (\mathbb{R}^{d_X} ; \mathbb{R})$ be fixed but arbitrary; it is sufficient to show that $\mu_{\tau}^1 (\varphi ) = \mu_{\tau}^2 (\varphi)$.
  We consider the solution $u = u_t (x) : [0, \tau] \times \mathbb{R}^{d_X} \rightarrow \mathbb{R}$ to the backward equation \eqref{eq:dualroughequation} on $[0,\tau]$ with terminal datum 
  $\varphi$, in the sense of \cref{def:solutionbackwardroughPDE}.
  The conclusion follows if we show that the map $t \mapsto \mu^i_t (u_t)$ is
  constant for $i = 1, 2$.
  Indeed, we would have
  \[ \mu^1_{\tau} (\varphi) = \mu^1_{\tau} (u_{\tau}) = \mu_0^1 (u_0) = \nu(u_0) =
     \mu_0^2 (u_0) = \mu^2_{\tau} (u_{\tau}) = \mu^2_{\tau} (\varphi) . \]
  Let us show that $[0,\tau] \ni t \mapsto \mu^i_t (u_t)$ is $(1+\varepsilon)$-H\"older continuous for $i=1,2$ and for some $\varepsilon >0$.
  For $\mu \in \{\mu^1, \mu^2\}$ and for any $0 \le s \le t \le T$ we can write \begin{equation} \label{eq:decomposition_uniqueness}
       \mu_t (u_t) - \mu_s (u_s) =  - \mu_s (u_s - u_t) + (\mu_t - \mu_s) (u_s) + (\mu_t - \mu_s) (u_t - u_s) .
  \end{equation}
  From the fact that $u$ solves \eqref{eq:dualroughequation} we have
  \begin{multline*}
    - \mu_s(u_s - u_t) = - \mu_s \left( \int_s^t (A^{\BY}_r u_r - \sum_{\kappa,\lambda=1}^{d_Y} (\Gamma_r^\BY)_{\kappa \lambda}' u_r \dot{[\BY]}_r^{\kappa \lambda} 
    )d r 
    \right) \\ - \mu_s \left( \sum_{\kappa=1}^{d_Y} (\Gamma^{\BY}_t)_\kappa u_t \delta Y^\kappa_{s,t} \right) 
     - \mu_s \left( \sum_{\kappa, \lambda=1}^{d_Y} \Big((\Gamma^{\BY}_t)_\kappa (\Gamma^{\BY}_t)_\lambda u_t -  (\Gamma^{\BY}_t)'_{\kappa \lambda} u_t \Big) \mathbb{Y}_{s,t}^{\kappa \lambda} \right) + \mu_s (u^{\natural}_{s, t}) 
  \end{multline*}
  with $| \mu_s (u^{\natural}_{s, t}) | \leqslant \sup_{r \in [0, \tau]} |  \mu_r (1)| \sup_{x \in \R^{d_X}} | u^{\natural}_{s, t}(x)| \lesssim | t - s |^{3 \alpha}$ uniformly in $s,t$. 
  From the fact that $\mu$ solves \eqref{eq:roughZakai_CITE} and that $\sup_{t \in [0,\tau]} |u_t|_{\mathcal{C}^3_b} < +\infty$ it follows that 
  \begin{multline*}
    (\mu_t - \mu_s) (u_s) = \mu_t (u_s) - \mu_s (u_s)
    = \int_s^t \mu_r (A^{\BY}_r u_s)
     \,  d r  \\
    + \sum_{\kappa=1}^{d_Y} \mu_s
    ((\Gamma^{\BY}_s)_\kappa u_s) \delta Y^\kappa_{s, t} + \sum_{\kappa, \lambda =1}^{d_Y}  \mu_s
    ((\Gamma^{\BY}_s)_\kappa (\Gamma^{\BY}_s)_\lambda u_s +
    (\Gamma^{\BY}_s )'_{\lambda \kappa} u_s) \mathbb{Y}_{s,t}^{\kappa \lambda} + \mu_{s, t}^{u_s,\natural}
  \end{multline*}
  where $| \mu_{s, t}^{ u_s, \natural} | \lesssim |t - s|^{3 \alpha}$. 
  Recall that, being $u$ solution to \eqref{eq:dualroughequation}, \begin{equation*}
      \mu_s((\Gamma^{\BY}_s)_\kappa (\Gamma^{\BY}_s)_\lambda u_s +
    (\Gamma^{\BY}_s )'_{\lambda \kappa} u_s) = \mu_s
    ((\Gamma^{\BY}_t)_\kappa (\Gamma^{\BY}_t)_\lambda u_t +
    (\Gamma^{\BY}_t )'_{\lambda \kappa} u_t) + P_{s,t}
  \end{equation*}
  with $|P_{s,t}|\lesssim |t-s|^{\alpha}$. 
  For the last term on the right-hand side of \eqref{eq:decomposition_uniqueness}, recalling that 
  $(|u_t|_\infty + |D_x u_t|_\infty + |D^2_{xx} u_t|_\infty) \in \mathcal{C}^\alpha([0,T];\R)$
  we write \begin{equation*}
      \begin{aligned}
        (\mu_t - \mu_s) (u_t - u_s) 
      & = \sum_{\kappa=1}^{d_Y} \mu_s ((\Gamma^{\BY}_s)_\kappa (u_t-u_s)) \delta Y^\kappa_{s,t} + Q_{s,t} \\
      &= \sum_{\kappa=1}^{d_Y} \mu_s ((\Gamma^{\BY}_s)_\kappa u_t - (\Gamma^{\BY}_t)_\kappa u_t) \delta Y^\kappa_{s,t} + \sum_{\kappa=1}^{d_Y} \mu_s ((\Gamma^{\BY}_t)_\kappa u_t - (\Gamma^{\BY}_s)_\kappa u_s ) \delta Y^\kappa_{s,t} + Q_{s,t}
      \end{aligned}
  \end{equation*}
  where, by construction, $|Q_{s,t}| \lesssim |t-s|^{3\alpha}$ uniformly in $s,t$. 
  Recalling the definition of the operator $\Gamma^{\BY}$, it is clear that under \cref{assumptionU} \begin{equation*}
      (\Gamma^{\BY}_s)_\kappa u_t(x) - (\Gamma^{\BY}_t)_\kappa u_t (x) = - \sum_{\lambda=1}^{d_Y} (\Gamma^{\BY}_t)'_{\kappa \lambda} u_t(x) \delta Y_{s,t}^{\lambda} + R_{s,t}(x)
  \end{equation*}
  with $\sup_{x \in \mathbb{R}^{d_X}} |R_{s,t}(x)| \lesssim |t-s|^{2\alpha}$ uniformly in $s,t$.
  Hence, we can write 
  \begin{align*}
      &\mu_s (u_t - u_s) + (\mu_t - \mu_s) (u_s) + (\mu_t - \mu_s) (u_t - u_s) \\
      &= \int_s^t  \mu_r (A^{\BY}_r u_s) 
      \,  d r 
      - \mu_s \left( \int_s^t A^{\BY}_r u_r 
      - \sum_{\kappa,\lambda=1}^{d_Y} (\Gamma_r^\BY)_{ \kappa \lambda}' u_r \dot{[\BY]}_r^{\kappa \lambda} 
      \, d r \right) \\
      &\quad + \sum_{\kappa,\lambda=1}^{d_X} \mu_s((\Gamma_t^{\BY})'_{\kappa\lambda} u_t)  ( \mathbb{Y}_{s,t}^{\kappa \lambda} + \mathbb{Y}_{s,t}^{\lambda \kappa} - \delta Y^\lambda_{s,t} \delta Y^\kappa_{s,t} ) + S_{s,t}
  \end{align*}
  with $|S_{s,t}|\lesssim |t-s|^{3\alpha}$. 
  The conclusion follows recalling that $$\mathbb{Y}_{s,t}^{\kappa \lambda} + \mathbb{Y}_{s,t}^{\lambda \kappa} - \delta Y_{s,t}^\kappa \delta Y^\lambda_{s,t} = - \delta [\BY]_{s,t}^{\kappa\lambda} = - \int_s^t \dot{[\BY]}^{\kappa\lambda}_r \, dr$$ for any $\kappa, \lambda =1,\dots,d_Y$, from the fact that $x \mapsto A_t^{\BY} u_t(x)$ is 
  Lipschitz continuous and bounded, uniformly in $t$, and noting that under \cref{assumptionU} $$\sup_{x \in \R^{d_X}} |(\Gamma_t^\BY)'_{\kappa\lambda} u_t(x) - (\Gamma_s^\BY)'_{\kappa\lambda} u_s(x)| \lesssim |t-s|^\alpha$$ uniformly in $s,t$.  
  Indeed, we get that \begin{equation*}
      \left|\mu_s \left( \int_s^t \sum_{\kappa,\lambda=1}^{d_Y} (\Gamma_r^\BY)_{ \kappa \lambda}' u_r \dot{[\BY]}_r^{\kappa \lambda}  \, d r \right) - \sum_{\kappa,\lambda=1}^{d_X} \mu_s((\Gamma_t^{\BY})'_{\kappa\lambda} u_t) \int_s^t \dot{[\BY]}_r^{\kappa\lambda} \, dr \right| \lesssim |t-s|^{1+\alpha}
  \end{equation*} and from \eqref{eq:condition_uniqueness} we also obtain \begin{align*}
      \left| \int_s^t \mu_r(A^{\BY}_r u_s) \, dr - \mu_s\left( \int_s^t A^{\BY}_r u_r \, dr \right) \right| &\le \left| \int_s^t \mu_r(A^{\BY}_r (u_s - u_r)) \, dr  \right| + \left| \int_s^t (\mu_r-\mu_s)(A^{\BY}_r u_r) \, dr  \right| \\
      &\lesssim |u|_{\mathcal{C}^{\alpha}} \sup_{r \in [0,\tau]} |\mu_r(1)| |t-s|^{1+\alpha}  + |t-s|^{1+\alpha} . 
  \end{align*}
\end{proof}

\begin{remark} The prototypical examples of $\BY$-dependence in $b,\sigma,f,h$ is ``Markovian'', with $\BY$ replaced by $Y_t = \pi_1 \BY_t$, in which case checking \cref{assumptionU} amounts to calculus. A simple sufficient condition reads $b=b(t,x,y),\sigma \in \mathcal{C}^{0,2,0}_{b}$ and 
$ f,h \in \mathcal{C}_b^{1,5,5}$.
\end{remark}

\subsection{Rough Kushner--Stratonovich equation} \label{sec:RKS}

In view of the Kallianpur--Striebel formula (i.e.\ \cref{thm:KS}), we define the normalized rough filter
$$
\varsigma^{\BY}_t(\varphi):= \frac{\mu^{\BY}_t(\varphi)}{\mu^{\BY}_t(1)}, \qquad t \in [0,T], \, \varphi \in \mathcal{B}_b(\R^{d_X}).
$$
We show that $\varsigma ^{\BY}$ solves the {\em rough
Kushner--Stratonovich equation} given by 
\begin{multline*}
    \varsigma_t^{\BY} (\varphi) = \varsigma_0^{\BY} (\varphi) + \int_0^t \varsigma_r^\BY  (A^{\BY}_r \varphi) \, d r  \\
    + \int_0^t (\varsigma_r^{\BY}  (\Gamma^{\BY}_r \varphi) - \varsigma_r^{\BY} (\varphi)  \varsigma_r^{\BY} (h(r,\cdot \, ;\BY))) (d  \BY_r - \dot{[\BY]}_r  \varsigma_r ({h}(r, \cdot \, ; \BY)^\top) d r) 
\end{multline*}
where $A^{\BY}, \Gamma^{\BY}$ are the time-dependent differential operators defined in \eqref{eq:operatorA^Y}, \eqref{eq:operatorGamma^Y}.  
For sake of a more compact notation, for any $\varsigma:[0,T] \to \mathcal{M}_F(\mathbb{R}^{d_X})$ and for any sufficiently regular test function $\varphi:\mathbb{R}^{d_X} \to \mathbb{R}$, let us define $\Phi^{\varsigma,\varphi}_t = ((\Phi^{\varsigma,\varphi}_t)_1,\dots,(\Phi^{\varsigma,\varphi}_t)_{d_Y})$ as \begin{equation*}
    (\Phi^{\varsigma,\varphi}_t)_\kappa := \varsigma_t((\Gamma^{\BY}_t)_\kappa \varphi) - \varsigma_t(\varphi) \varsigma_t(h_\kappa(t,\cdot \, ; \BY))  \qquad \kappa=1,\dots,d_Y. 
\end{equation*}
The following formal identities \begin{equation} \label{eq:formalidentities_roughKS} \begin{aligned}
    \varsigma_t((\Gamma^{\BY}_t)_\kappa \varphi) - \varsigma_s((\Gamma^{\BY}_s)_\kappa \varphi) &\approx \sum_{\eta=1}^{d_Y} [\varsigma_s((\Gamma_s^{\BY})_\eta (\Gamma_s^{\BY})_\kappa \varphi + (\Gamma_s^{\BY})'_{\kappa \eta} \varphi )  - \varsigma_s((\Gamma_s^{\BY})_\kappa \varphi) \varsigma_s(h_\eta(s,\cdot \, ; \BY))] \delta Y^\eta_{s,t} \\
    & =: \sum_{\eta=1}^{d_Y} \varsigma_s((\Gamma^{\BY}_s)_\kappa \varphi)'_\eta \delta Y^\eta_{s,t} \\
    \varsigma_t (\varphi) - \varsigma_s(\varphi) &\approx \sum_{\eta=1}^{d_Y} [ \varsigma_s((\Gamma^{\BY}_s)_\eta \varphi) - \varsigma_s(\varphi) \varsigma_s(h_\eta(s,\cdot \, ;\BY))  ] \delta Y^\eta_{s,t} \\
    & =: \sum_{\eta=1}^{d_Y} \varsigma_s(\varphi)_\eta' \delta Y^\eta_{s,t}
\end{aligned}
\end{equation}
and the product rule for Gubinelli derivatives justify the definition of \[
(\Phi^{\varsigma,\varphi}_t)' \assign 
\left(
  \begin{array}{cccc}
    (\Phi^{\varsigma,\varphi}_t)_{11}' & \cdots  & (\Phi^{\varsigma,\varphi}_t)'_{1 d_Y} \\
    \vdots & \ddots & \vdots       \\
    (\Phi^{\varsigma,\varphi}_t)_{d_Y 1}' & \cdots & (\Phi^{\varsigma,\varphi}_t)'_{d_Y d_Y}
  \end{array}
\right)
\]
given by 
   \begin{equation*}
    (\Phi^{\varsigma,\varphi}_t)'_{\kappa \lambda} := \varsigma_t((\Gamma^{\BY}_t)_\kappa \varphi)'_\lambda - \varsigma_t(\varphi)_\lambda' \varsigma_t(h_\kappa(t,\cdot \, ; \BY)) - \varsigma_t(\varphi) \varsigma_t((\Gamma^{\BY}_t)_\kappa 1)'_\lambda, \qquad \kappa,\lambda=1,\dots,d_Y .
\end{equation*}
We used the fact that $h_\kappa(t,x;\BY) = (\Gamma^\BY_t)_\kappa 1(x)$.

\begin{definition}[Rough
Kushner--Stratonovich] \label{def:roughKushnerStratonovich}
  A measure-valued solution to
  \begin{multline} \label{eq:roughKS_def}
      \varsigma_t (\varphi) = \varsigma_0 (\varphi) + \int_0^t \varsigma_r     (A^{\mathbf{Y}}_r \varphi) d r  \\+ \int_0^t (\varsigma_r     (\Gamma^{\BY}_r \varphi) - \varsigma_r (\varphi) \varsigma_r
     (h(r,\cdot \, ; \BY))) (d \BY_r - \dot{[\BY]}_r \varsigma_r(h(r,\cdot;\BY)^\top) \, dr)
  \end{multline} 
  is a weakly continuous curve $\varsigma : [0, T] \rightarrow \mathcal{M}_F
  (\mathbb{R}^{d_X})$ such that, for any test function $\varphi \in
  \mathcal{C}^3_b (\mathbb{R}^{d_X} ; \mathbb{R})$, the following are satisfied:
  \begin{enumerate}
    \item the maps $t \mapsto \varsigma_t  (A^{\BY}_t \varphi)$ and $t \mapsto  \sum_{\kappa,\lambda=1}^{d_Y} (\varsigma_t     ((\Gamma^{\BY}_t)_\kappa \varphi) - \varsigma_t (\varphi) \varsigma_t  (h_\kappa(t,\cdot \, ; \BY))) \dot{[\BY]}_t^{\kappa\lambda} \varsigma_t( h_\lambda(t,\cdot \, ;\BY))$ belong to
    $L^1  ([0, T], dt)$; 
    
    \item there is $C = C_{\varphi} > 0$ uniform over bounded sets of
    $\varphi$ in $\mathcal{C}^3_b(\R^{d_X};\R)$ such that, for any $0 \le s \le t
    \le T$ and for any $\kappa, \lambda = 1, \ldots, d_Y$, 

    \begin{align*}
        | (\Phi^{\varsigma,\varphi}_t)'_{\kappa \lambda} - (\Phi^{\varsigma,\varphi}_s)'_{\kappa \lambda} | &\le C_\varphi |t-s|^\alpha \\
        |  (\Phi^{\varsigma,\varphi}_t)_\kappa - (\Phi^{\varsigma,\varphi}_s)_\kappa - \sum_{\eta=1}^{d_Y} (\Phi^{\varsigma,\varphi}_s)'_{\kappa \lambda} \delta Y^\eta_{s,t} | &\le C_\varphi |t-s|^{2\alpha} ;
    \end{align*}
    
    \item the following Davie-type expansion holds for any $0 \le s \le t \le
    T$,
    \begin{multline*}
      \varsigma_t (\varphi) - \varsigma_s (\varphi)  = \int_s^t \varsigma_r (A^{\BY}_r \varphi) - \sum_{\kappa,\lambda=1}^{d_Y} (\varsigma_t     ((\Gamma^{\BY}_t)_\kappa \varphi) - \varsigma_t (\varphi) \varsigma_t  (h_\kappa(t,\cdot \, ; \BY))) \dot{[\BY]}_t^{\kappa\lambda} \varsigma_t( h_\lambda(t,\cdot \, ;\BY))\, dr  \\
       + \sum_{\kappa = 1}^{d_Y} \underbrace{(\varsigma_s ((\Gamma^{\BY}_s)_{\kappa} \varphi) - \varsigma_s (\varphi) \varsigma_s (h_\kappa(s,\cdot \, ; \BY)))}_{= (\Phi^{\varsigma,\varphi}_s)_\kappa} \delta Y^{\kappa}_{s, t}   + \sum_{\kappa, \lambda = 1}^{d_Y} (\Phi^{\varsigma,\varphi}_t)'_{\lambda \kappa } \mathbb{Y}_{s,t}^{\kappa \lambda} + \mu_{s, t}^{\varphi, \natural}
    \end{multline*}

    where $\mu_{s, t}^{\varphi, \natural} = o (| t - s |)$ as $| t - s |
    \rightarrow 0$.
  \end{enumerate}
  If the initial value $\varsigma_0 = \nu \in \mathcal{M}_F (\mathbb{R}^{d_X})$ is
  specified, we say that $\varsigma$ is a solution starting from $\nu$.
\end{definition}

\begin{proposition}
  Let $\varsigma : [0, T] \to \mathcal{M}_F (\R^{d_X})$ satisfying condition (ii) of
  the previous definition. Then the following limit exists for any $t \in [0,
  T]$ along any sequence of partitions $\pi$ of $[0, t]$ whose mesh tends to
  zero:
  \begin{equation*}
    \int_0^t \varsigma_r (\Gamma_r^{\BY} \varphi) - \varsigma_r (\varphi) \varsigma_r (h(r,\cdot \, ;\BY)) d \mathbf{Y}_r \assign  \lim_{| \pi | \to 0}  \sum_{[s, u] \in \pi} \left( \sum_{\kappa=1}^{d_Y} (\Phi^{\varsigma,\varphi}_s)_\kappa \delta Y^{\kappa}_{s, u} + \sum_{\kappa,\lambda=1}^{d_Y} (\Phi^{\varsigma,\varphi}_s)'_{\lambda \kappa} \mathbb{Y}_{s,
    u}^{\kappa \lambda} \right) .
  \end{equation*}
  In particular, condition 3. of the previous definition can be equivalently
  replaced by
  \begin{itemize}
    \item[3'.] for any $t \in [0, T]$ \begin{multline*}
        \varsigma_t (\varphi) = \varsigma_0 (\varphi) + \int_0^t \varsigma_r (A^{\BY}_r \varphi) - \sum_{\kappa,\lambda=1}^{d_Y} (\varsigma_t     ((\Gamma^{\BY}_t)_\kappa \varphi) - \varsigma_t (\varphi) \varsigma_t  (h_\kappa(t,\cdot \, ; \BY))) \dot{[\BY]}_t^{\kappa\lambda} \varsigma_t( h_\lambda(t,\cdot \, ;\BY)) \, dr \\ + \int_0^t \varsigma_r
       (\Gamma^{\BY}_r \varphi) - \varsigma_r (\varphi) \varsigma_r (h(r,\cdot \, ;\BY)) d \mathbf{Y}_r,
    \end{multline*}
  \end{itemize}
  and the quantity $\mu_{s, t}^{\natural, \varphi}$ in fact satisfies $|
  \mu_{s, t}^{\natural, \varphi} | \le K_{\varphi}  |t - s|^{3 \alpha}$ for
  any $s, t \in [0, T]$, where $K_{\varphi} > 0$ is uniformly over bounded
  sets of $\varphi$ in $\mathcal{C}^3_b (\R^{d_X} ; \R)$.
\end{proposition}

\begin{proof}
  Follows in a similar fashion to the proof of \cref{prop:roughZakai_integralform}, taking
  $$\Xi^{\varphi}_{s, u} = \sum_{\kappa=1}^{d_Y} (\Phi^{\varsigma,\varphi}_s)_\kappa \delta Y^{\kappa}_{s, u} + \sum_{\kappa, \lambda =1}^{d_Y} (\Phi^{\varsigma,\varphi}_s)'_{\lambda \kappa} \mathbb{Y}_{s,u}^{\kappa \lambda} . $$
\end{proof}

Existence for rough Kushner--Stratonovich follows from existence for rough Zakai, while uniqueness of the former is equivalent to uniqueness of the latter. 

\begin{theorem} \label{thm:existenceroughKS}
Let $\mu:[0,T] \to \mathcal{M}(\mathbb{R}^{d_X})$ be a measure-valued solution to the rough Zakai equation \eqref{eq:roughZakai_CITE} in the sense of \cref{def:solutionroughZakai} such that $\mu_t(1) >0$ \footnote{Notice that $\mu^\BY$ defined as in \eqref{eq:roughunnormfilter} satisfies $\mu^\BY_t(1)>0$ for any $t \in [0,T]$. However we consider here just a solution in the sense of \cref{def:solutionroughZakai}.} for any $t \in [0,T]$.
Define the measure-valued path $\varsigma: [0,T] \to \mathcal{M}(\mathbb{R}^{d_X})$ via \begin{equation} \label{eq:definitionofvarsigma_existenceroughKS}
    \varsigma_t(\varphi) := \frac{\mu_t(\varphi)}{\mu_t(1)} 
\end{equation} for any $\varphi:\mathbb{R}^{d_X} \to \mathbb{R}$ measurable and bounded. 
Then $\varsigma$ solves the following rough PDE, in the sense of \cref{def:roughKushnerStratonovich}: 
\begin{equation} \label{eq:roughKushnerStratonovich_CITE}
    \varsigma_t (\varphi) = \varsigma_0 (\varphi) + \int_0^t \varsigma_r     (A^{\BY}_r \varphi) d r + \int_0^t (\varsigma_r     (\Gamma^{\BY}_r \varphi) - \varsigma_r (\varphi) \varsigma_r     (h(r,\cdot;\BY))) (d \BY_r - \dot{[\BY]}_r \varsigma_r(h(r,\cdot;\BY)^\top) \, dr). 
\end{equation}
\end{theorem}

\begin{proof} 
Let $\varsigma:[0,T] \to \mathcal{M}(\R^{d_X})$ be defined as in \eqref{eq:definitionofvarsigma_existenceroughKS}. By construction, it is weakly continuous and $\varsigma_t(1)=1$ for any $t$. 
Note that $h_\kappa(t,x;\BY) = (\Gamma^\BY_t)_\kappa 1(x)$. From \cref{def:solutionroughZakai} we have, for any $0 \le s \le t \le T$,  \begin{equation*}
    \mu_t (1) - \mu_s (1) = 
     \mu_s ((\Gamma_s^{\BY})_{\kappa} 1) \delta
    Y^{\kappa}_{s, t} + \mu_s
    ((\Gamma_s^{\BY})_{\kappa} (\Gamma_s^{\BY})_{\lambda} 1
    + (\Gamma_s^{\BY})'_{\lambda \kappa} 1) \mathbb{Y}^{\kappa
    \lambda}_{s, t} + \mu_{s, t}^{1, \natural} .
\end{equation*}
Straightforwardly from \cite[Theorem 7.7]{FH20} with $F(x)=\frac{1}{x}$, together with a trivial localization argument for $x \mapsto 1/x$ away from zero, 
\begin{equation} \label{eq:expansion1/mu}
\begin{aligned}
    \frac{1}{\mu_t (1)} - \frac{1}{\mu_s (1)} 
    & =  \int_s^t 
    \varsigma_s
    ((\Gamma_s^{\BY})_{\kappa} 1) \varsigma_s
    ((\Gamma_s^{\BY})_{\lambda} 1) 
    (\dot{[\BY]}_r)^{\kappa \lambda} \frac{1}{\mu_r(1)} \, d r\\
    &   - \frac{\varsigma_s ((\Gamma_s^\BY)_{\kappa}
    1)}{\mu_s (1)} \delta Y^{\kappa}_{s, t} \\
    &   + \left( - \varsigma_s
    ((\Gamma_s^{\BY})_{\kappa} (\Gamma_s^{\BY})_{\lambda} 1
    + (\Gamma_s^{\BY})'_{\lambda \kappa} 1) + 2 \varsigma_s
    ((\Gamma_s^{\BY})_{\kappa} 1) \varsigma_s
    ((\Gamma_s^{\BY})_{\lambda} 1) \right) \frac{1}{\mu_s(1)}
    \mathbb{Y}^{\kappa \lambda}_{s, t} \\
    &   + P_{s,t} ,
\end{aligned}
\end{equation}
with $|P_{s,t}| \lesssim |t-s|^{3\alpha}$.
Another application of \cite[Theorem 7.7]{FH20} with $F(x,y)=x \cdot\frac{1}{y}$ shows that the measure-valued path $\varsigma$ satisfies requirement \textit{3.}\ in \cref{def:roughKushnerStratonovich}. 
To prove that condition \textit{2.}\ is satisfied, it is sufficient to justify the two formal identities in \eqref{eq:formalidentities_roughKS}.
Let
$\varphi \in \mathcal{C}^3_b(\R^{d_X};\R)$ be fixed. 
 From \eqref{eq:expansion1/mu} we get that, for any $0 \le s \le t \le T$ and for any $\kappa=1,\dots,d_Y$,
 \begin{align*}
    &\varsigma_t((\Gamma^\BY_t)_\kappa\varphi) - \varsigma_s((\Gamma^\BY_s)_\kappa\varphi) = \frac{1}{\mu_s(1)} (\mu_t((\Gamma^\BY_t)_\kappa\varphi) - \mu_s((\Gamma^\BY_s)_\kappa\varphi))   \\ & \quad +\mu_s((\Gamma^\BY_s)_\kappa\varphi) \left(\frac{1}{\mu_t(1)} - \frac{1}{\mu_s(1)} \right) + \left(\frac{1}{\mu_t(1)} - \frac{1}{\mu_s(1)} \right)(\mu_t((\Gamma^\BY_t)_\kappa\varphi) - \mu_s((\Gamma^\BY_s)_\kappa\varphi))  \\
    &= [\varsigma_s( (\Gamma_s^\BY)_\eta (\Gamma_s^\BY)_\kappa \varphi + (\Gamma^\BY_s)'_{\kappa \eta} \varphi ) - \varsigma_s((\Gamma^\BY_s)_\kappa \varphi) \varsigma_s((\Gamma^\BY_s)1) ] \delta Y^\eta_{s,t} + Q_{s,t}
\end{align*}
where, by construction, $|Q_{s,t}| \lesssim |t-s|^{2\alpha}$. Similarly \begin{multline*}
    \varsigma_t(\varphi) - \varsigma_s(\varphi) = \frac{1}{\mu_s(1)} (\mu_t(\varphi)-\mu_s(\varphi)) + \left( \frac{1}{\mu_t(1)} - \frac{1}{\mu_s(1)} \right) \mu_s(\varphi)  \\
     + \left( \frac{1}{\mu_t(1)} - \frac{1}{\mu_s(1)} \right) (\mu_t(\varphi)-\mu_s(\varphi)) = [ \varsigma_s((\Gamma_s^\BY)_\eta) - \varsigma_s((\Gamma^\BY_s)_\eta 1) \varsigma_s(\varphi) ] \delta Y^\eta_{s,t} + R_{s,t}
\end{multline*}
with $|R_{s,t}|\lesssim |t-s|^{2\alpha}$. Now we show that condition \textit{1.} of \cref{def:roughKushnerStratonovich} is also satisfied. The map $[0,T] \ni t \mapsto \varsigma_t(A^\BY_t \varphi) \in \R$ is measurable because $t \mapsto \mu_t(A^\BY_t \varphi)$ is measurable and $t \mapsto \mu_t(1)$ is continuous and bounded away from 0. Moreover, \begin{equation*}
    \int_0^T |\varsigma_t(A^\BY_t \varphi)| \, dt \le \frac{1}{\min_{0\le t \le T} \mu_t(1)} \int_0^T |\mu_t(A^\BY_t \varphi)| \, dt < +\infty. 
\end{equation*}
Note that $t \mapsto  \sum_{\kappa,\lambda=1}^{d_Y} (\varsigma_t((\Gamma^\BY_t)_\kappa\varphi) -\varsigma_t(\varphi) \varsigma_t((\Gamma^\BY_t)_\kappa 1)) \dot{[\BY]}^{\kappa\lambda}_t \varsigma_t((\Gamma^\BY_t)_\lambda 1)$ is continuous, which concludes the proof. 
\end{proof}

\begin{theorem} \label{thm:uniqueness_roughKS}
  Let $\mu_0 = \varsigma_0$ be probability measures on $\R^{d_X}$. Then the following are equivalent:
  \begin{enumerate}
      \item there exists at most one measure-valued solution $\mu$ to the rough Zakai equation \eqref{eq:roughZakai_CITE} starting from $\mu_0$ and such that $\mu_t(1) >0$ for any $t \in [0,T]$;
      \item there exists at most one measure-valued solution $\varsigma$ of the rough Kushner--Stratonovich equation \eqref{eq:roughKushnerStratonovich_CITE} starting from $\varsigma_0$ and such that
      \begin{equation} \label{eq:condition2b}
          (\Psi^\varsigma, (\Psi^\varsigma)') \in \mathscr{D}_Y^{2\alpha}([0,T], \R^{1\times d_Y})
      \end{equation}
         in the sense of \cref{def:stochasticcontrolledroughpaths} \footnote{For definiteness, the condition requires that the following are satisfied, for any $\kappa,\lambda=1,\dots,d_Y$ and for any $0 \le s \le t \le T$: 
          $$|(\Psi_t^\varsigma)'_{\kappa\lambda} - (\Psi_s^\varsigma)'_{\kappa\lambda}| \lesssim |t-s|^{\alpha} \quad \text{and} \quad  |(\Psi_t^\varsigma)_\kappa - (\Psi^\varsigma_s)_\kappa - \sum_{\eta=1}^{d_Y} (\Psi_s^\varsigma)'_{\kappa\lambda} \delta Y^\eta_{s,t}| \lesssim |t-s|^{2\alpha}.$$
          }, where \begin{align*}
        (\Psi^\varsigma_t)_\kappa &:= \varsigma_t(h_\kappa(t,\cdot \, ; \BY)) \\
        (\Psi^\varsigma_t)'_{\kappa\lambda} &:= \varsigma_t((\Gamma^\BY_t)_\lambda h_\kappa(t,\cdot\, ; \BY)+h'_{\kappa\lambda} (t,\cdot \, ; \BY)) - \varsigma_t(h_\kappa(t,\cdot \, ; \BY )) \varsigma_t(h_\lambda(t,\cdot \, ; \BY)).
    \end{align*} 
          for $ t \in [0,T]$ and $\kappa,\lambda=1,\dots,d_Y$.   
    \end{enumerate}
\end{theorem}

\begin{remark}
    It is easy to see that condition \eqref{eq:condition2b} holds for any solution $\varsigma$ to \eqref{eq:roughKushnerStratonovich_CITE} in the sense of \cref{def:roughKushnerStratonovich}, provided that $[ t \mapsto (h(t,\cdot \, ; \BY), h'(t,\cdot \, ; \BY) )] \in \mathscr{D}_Y^{2\alpha} \mathcal{C}^\beta_b $ in the sense of \cref{def:stochasticcontrolledvectorfields_intro} and for some $\beta \in \left(\frac{1}{\alpha},3\right]$. 
    In particular, condition \eqref{eq:condition2b} holds under \cref{assum:E}.
\end{remark}

\begin{proof}
    Let us first assume that \textit{2.}\ holds.  
    Let $\mu^1$ and $\mu^2$ be two measure-valued solutions to \eqref{eq:roughZakai_CITE} such that $\mu_t^i(1) >0$ for any $t \in [0,T]$ and $i=1,2$. 
   Notice that the total mass processes can be expressed as solutions to the following \begin{equation*}
        \mu^i_t(1) = 1 + \int_0^t \frac{\mu^i_r(\Gamma^\BY_r 1)}{\mu_r^i(1)} \mu_r^i(1) \, d \BY_r \qquad t \in [0,T], \ i=1,2  .
    \end{equation*}
    Define now $\varsigma^i := \frac{\mu^i}{\mu^i(1)}$, $i=1,2$.  With the same argument as in the proof of \cref{thm:existenceroughKS} we have that $\varsigma^1$ and $\varsigma^2$ both solve \eqref{eq:roughKushnerStratonovich_CITE} and satisfy condition \eqref{eq:condition2b}.
    Hence, by uniqueness, $\varsigma^1=\varsigma^2=\varsigma$. Moreover, $Z^i_t := \mu^i_t(1)$ solves the linear equation \eqref{eq:auxiliarylinearRDE} of \cref{lemma:linearauxilaryequation} with $Z^i_0=1$, for $i=1,2$. Hence, $Z^1=Z^2$. We conclude that  \begin{equation*}
        \mu^1_t = \mu^1_t(1) \varsigma^1_t = \mu^2_t(1) \varsigma^2_t = \mu^2_t \qquad \text{for any $t \in [0,T]$} .
    \end{equation*}

    Let us now assume that \textit{1.}\ is satisfied. 
    Let $\varsigma^1$ and $\varsigma^2$ be two measure-valued solutions to \eqref{eq:roughKushnerStratonovich_CITE} starting from $\varsigma_0$ and satisfying condition \eqref{eq:condition2b}.
    By \cref{lemma:linearauxilaryequation}, we can consider $Z^{\varsigma^1}$ and $Z^{\varsigma^2}$ as the solutions to the linear RDE \eqref{eq:auxiliarylinearRDE} starting from $Z^{\varsigma^1}_0=Z^{\varsigma^2}_0=1$. 
    Moreover, both $Z^{\varsigma^1}\varsigma^1$ and $Z^{\varsigma^2}\varsigma^2$ are solutions to \eqref{eq:roughZakai_CITE} starting from $\varsigma_0$.
    Recall that by assumption $\varsigma_0(1)=1$. Hence, by \cref{lemma:linearauxilaryequation}, $\varsigma^1_t$ and $\varsigma^2_t$ are probability measures for any $t \in [0,T]$. We deduce that $Z^{\varsigma^i}_t \varsigma_t^i(1) = Z^{\varsigma^i}_t >0$ for any $ t \in [0,T]$.     
    By uniqueness, $Z^{\varsigma^1}\varsigma^1=Z^{\varsigma^2}\varsigma^2 = \mu$.  
    In particular \begin{equation*}
        Z_t^{\varsigma^1} = Z_t^{\varsigma^1} \varsigma_t^1(1) = \mu_t(1) = Z_t^{\varsigma^2} \varsigma^2_t(1) = Z_t^{\varsigma^2}  \qquad \text{for any $t \in [0,T]$}.
    \end{equation*}
    Recalling that $\mu_t(1) >0$ for any $t \in [0,T]$, we conclude that $\mu_t(1) \varsigma^1_t = \mu_t(1) \varsigma^2_t$, for any $t \in [0,T]$, which proves $\varsigma^1=\varsigma^2$ being $\mu_t(1) >0$.
\end{proof}

The following result is needed in the proof of \cref{thm:uniqueness_roughKS}. 

\begin{lem} \label{lemma:linearauxilaryequation}
    Let $\varsigma$ be a measure-valued solution to \eqref{eq:roughKushnerStratonovich_CITE} satisfying condition \eqref{eq:condition2b} of \cref{thm:uniqueness_roughKS}. 
    Then the following linear RDE is well-posed: \begin{equation} \label{eq:auxiliarylinearRDE}
        dZ^\varsigma_t = Z_t^\varsigma \varsigma_t(h(t,\cdot \, ; \BY))  \, d\BY_t.      \end{equation} 
    In particular,  its unique solution $Z^\varsigma$ is strictly positive if $Z^\varsigma_0 >0$. 
    Moreover, if $\varsigma_0(1)=1$ the following hold true: \begin{enumerate} 
        \item[i.]  every $\varsigma_t$ is a probability measure on $\mathbb{R}^{d_Y}$;
        \item[ii.] $Z^\varsigma \varsigma$ is a solution to \eqref{eq:roughZakai_CITE} starting from $Z^\varsigma_0 \varsigma_0$. 
    \end{enumerate}
\end{lem}

\begin{proof}
    Well-posedness of \eqref{eq:auxiliarylinearRDE} follows by arguing as in \cref{thm:wellposedness_roughfiltering}, with the only difference that now $Z^\varsigma$ is deterministic.
    Its (unique) solution is necessarily given by \begin{equation*}
        Z^\varsigma_t = Z^\varsigma_0 \exp \left( \int_0^t \varsigma_t(h(t,\cdot \, ; \BY)) \, d \BY_r \right) \qquad t \in [0,T]
    \end{equation*}
    and it is strictly positive for any $t$ provided that $Z^\varsigma_0 >0$.
    Recall that a solution to \eqref{eq:auxiliarylinearRDE} is a continuous map $Z^\varsigma:[0,T] \to \R$ such that for any $0 \le s \le t \le T$ \begin{equation} \label{eq:DavieexpansionforZ}
        Z^\varsigma_t - Z^\varsigma_s = 
        \sum_{\kappa=1}^{d_Y} (\Psi_s^\varsigma)_\kappa Z^\varsigma_s \delta Y^\kappa_{s,t}
        + \sum_{\kappa,\lambda=1}^{d_Y} ( (\Psi_s^\varsigma)_\lambda (\Psi_s^\varsigma)_\kappa Z^\varsigma_s + (\Psi_s^\varsigma)'_{\lambda \kappa} Z^\varsigma_s) \mathbb{Y}^{\kappa \lambda}_{s,t} + Z^{\varsigma,\natural}_{s,t}
    \end{equation}
    with $Z^{\varsigma,\natural}_{s,t} = o(|t-s|)$ as $|t-s|\to 0$. 
      \\ 

    \textit{i.} \ Assume $\varsigma_0(1)=1$ and define $a_t := \varsigma_t(1)-1, \ t \in [0,T]$. It is straightforward to see that \begin{align*}
        a_t &= \varsigma_t(1) - \varsigma_0(1) = \int_0^t \varsigma_r(h(r,\cdot)) \underbrace{(1-\varsigma_r(1))}_{=-a_r}  \, (d\BY_r - \dot{[\BY]}_r \varsigma_r(h(r,\cdot \, ; \BY)^\top) \, dr) \quad t \in [0,T]. 
    \end{align*}
    Hence, $a$ is a solution to a linear RDE (with drift) starting from $a_0=0$. The well-posedness of such an equation follows with a similar argument as for \eqref{eq:auxiliarylinearRDE}, additionally noting that $[t \mapsto - \sum_{\kappa,\lambda=1}^{d_Y} \varsigma_t(h_\kappa(t,\cdot \, ; \BY)) \dot{[\BY]} \varsigma_t(h_\lambda(t,\cdot\, ; \BY))]$ is measurable and bounded 
    by condition \eqref{eq:condition2b}. 
    It follows that $a_t=0$ for any $t \in (0,T]$.  \\

    \textit{ii.} \ For sake of compact notation we write $h(t,x)$ 
    instead of $h(t,x;\BY)$ and $Z$ instead of $Z^\varsigma$. 
    Define $\mu_t(\varphi) := Z_t \varsigma_t(\varphi)$ for $ \ t \in [0,T]$. 
    Note that $[t \mapsto \mu_t (A^\BY_t \varphi)
    ]$ belongs to $ L^1([0,T])$ since $Z$ is continuous and $[t \mapsto \varsigma_t (A^\BY_t \varphi) 
    ]$ is Lebesgue integrable by assumption.  
    Recall the Davie-type expansion for $\varsigma(\varphi)$ with the notation introduced in \cref{def:roughKushnerStratonovich} and also consider the one for $Z^\varsigma$ in \eqref{eq:DavieexpansionforZ}.
    One can therefore get that
    \begin{align*}
        & Z_t (\Phi_t^{\varsigma,\varphi})'_{\lambda \kappa} + \varsigma_t(\varphi) ((\Psi^\varsigma_t)_\lambda (\Psi^\varsigma_t)_\kappa Z_t + (\Psi^\varsigma_t)'_{\lambda \kappa}Z_t) +(\Psi^\varsigma_t)_\kappa Z_t (\Phi^{\varsigma,\varphi}_t)_\lambda +(\Psi^\varsigma_t)_\lambda Z_t (\Phi^{\varsigma,\varphi}_t)_\kappa =  \\
        &= Z_s (\Phi_s^{\varsigma,\varphi})'_{\lambda \kappa} + \varsigma_s(\varphi) ((\Psi^\varsigma_s)_\lambda (\Psi^\varsigma_s)_\kappa Z_s + (\Psi^\varsigma_s)'_{\lambda \kappa}Z_s) +(\Psi^\varsigma_s)_\kappa Z_s (\Phi^{\varsigma,\varphi}_s)_\lambda  +(\Psi^\varsigma_s)_\lambda Z_s (\Phi^{\varsigma,\varphi}_s)_\kappa  + P_{s,t} 
        \end{align*} and
        \begin{align*}
            & \varsigma_t(\varphi) (\Psi^\varsigma_t)_\kappa Z_t + Z_t (\Phi^{\varsigma,\varphi}_t)_\kappa - \varsigma_s(\varphi) (\Psi^\varsigma_s)_\kappa Z_s + Z_s (\Phi^{\varsigma,\varphi}_s)_\kappa  \\
        &= \sum_{\eta=1}^{d_Y} \Big(Z_s (\Phi_s^{\varsigma,\varphi})'_{\eta \kappa} + \varsigma_s(\varphi) ((\Psi^\varsigma_s)_\eta (\Psi^\varsigma_s)_\kappa Z_s + (\Psi^\varsigma_s)'_{\eta \kappa}Z_s) +(\Psi^\varsigma_s)_\kappa Z_s (\Phi^{\varsigma,\varphi}_s)_\eta  +(\Psi^\varsigma_s)_\eta Z_s (\Phi^{\varsigma,\varphi}_s)_\kappa\Big) \delta Y^{\eta}_{s,t} + Q_{s,t}
        \end{align*}
    where $|P_{s,t}|\lesssim |t-s|^{\alpha}$, $|Q_{s,t}|\lesssim |t-s|^{2\alpha}$.
    From \cite[Theorem 7.7]{FH20} with $F(x,y)=xy$ and for any $0 \le s \le t \le T$ we get that 
    \begin{align*}
        \mu_t(\varphi) - \mu_s(\varphi) &= \int_s^t \mathcal{A}^{\varsigma,\varphi;\BY}_r \, dr + \sum_{\kappa=1}^{d_Y} (\varsigma_s(\varphi) (\Psi^\varsigma_s)_\kappa Z_s + Z_s (\Phi^{\varsigma,\varphi}_s)_\kappa) \delta Y^\kappa_{s,t}  \\
        &\quad + \Big( Z_s (\Phi_s^{\varsigma,\varphi})'_{\lambda \kappa} + \varsigma_s(\varphi) ((\Psi^\varsigma_s)_\lambda (\Psi^\varsigma_s)_\kappa Z_s + (\Psi^\varsigma_s)'_{\lambda \kappa}Z_s)  \\
        & \quad +(\Psi^\varsigma_s)_\kappa Z_s (\Phi^{\varsigma,\varphi}_s)_\lambda + (\Psi^\varsigma_s)_\lambda Z_s (\Phi^{\varsigma,\varphi}_s)_\kappa \Big) \mathbb{Y}^{\kappa\lambda}_{s,t} + R_{s,t}
    \end{align*}
    where
     $|R_{s,t}|\lesssim |t-s|^{3\alpha}$ and  \begin{multline*}
        \mathcal{A}^{\varsigma,\varphi;\BY}_t = Z^\varsigma_t \varsigma_t(A^\BY_t \varphi) - \sum_{\kappa,\lambda=1}^{d_Y} \Big( Z^\varsigma_t \varsigma_t((\Gamma_t^\BY)_\kappa \varphi) \dot{[\BY]}_t^{\kappa\lambda} \varsigma_t( h_\lambda(r,\cdot))  \\
        + Z_t\varsigma_t(\varphi) \varsigma_t(h_\kappa(t,\cdot) \dot{[\BY]}_t^{\kappa\lambda} \varsigma_t( h_\lambda(r,\cdot)) + \sum_{\lambda=1}^{d_Y} (\Psi^{\varsigma}_t)_\lambda Z_t (\Phi_t^{\varsigma,\varphi})_\kappa \dot{[\BY]}_t^{\kappa\lambda} \Big).
    \end{multline*}
    By construction, $\varsigma_s(\varphi) (\Psi^\varsigma_s)_\kappa Z_s + Z_s (\Phi^{\varsigma,\varphi}_s)_\kappa = \mu_s((\Gamma^\BY_s)_\kappa \varphi)$ and \begin{multline*}
        Z_s (\Phi_s^{\varsigma,\varphi})'_{\lambda \kappa} + \varsigma_s(\varphi) ((\Psi^\varsigma_s)_\lambda (\Psi^\varsigma_s)_\kappa Z_s + (\Psi^\varsigma_s)'_{\lambda \kappa}Z_s) +(\Psi^\varsigma_s)_\kappa Z_s (\Phi^{\varsigma,\varphi}_s)_\lambda \\ +(\Psi^\varsigma_s)_\lambda Z_s (\Phi^{\varsigma,\varphi}_s)_\kappa = \mu_s((\Gamma^\BY_s)_\kappa (\Gamma_s^\BY)_\lambda \varphi + (\Gamma^\BY_s)'_{\lambda \kappa}\varphi) .
    \end{multline*}
    Hence, it is easy to see that \begin{equation*}
        \mathcal{A}^{\varsigma,\varphi;\BY}_t = \mu_t (A^\BY_t \varphi) 
        .
    \end{equation*}
\end{proof}

\subsection{Extension to degenerate observation noise }\label{subsec:degen-case}
In this section we discuss how to extend the rough filtering framework to the case of a degenerate observation noise (i.e.\ when $k$ is not invertible). Taking inspiration from the stochastic setting (cf. \cref{sec:review}), rewrite the signal dynamics \eqref{eq:stoch-sig-1} as
\begin{align*}
    dX_t 
    &= (\bar b - \bar f h_2)(t,X_t,Y_t) \, dt + \sigma(t,X_t,Y_t) \, dB_t + \bar f(t,X_t,Y_t) k^+(t,Y_t) \, (dY_t - h_1(t,Y_t) \, dt )  \\ & \quad +  \bar f(t,X_t,Y_t) (I - k^+(t,Y_t) k(t,Y_t)) \, dW_t  \\
    &=: (\bar b - \bar f h_2)(t,X_t,Y_t) \, dt + \sigma(t,X_t,Y_t) \, dB_t + \bar{\mathfrak{f}}(t,X_t,Y_t) \,  d \hat{Y}_t    
\end{align*}
where $k^+$ denotes the Moore-Penrose inverse of $k$ (cf.\ \cite[Appendix]{CP24}), with $(d_Y + 2 d_W)$-dimensional process $\hat Y$ defined as  \begin{equation*}
   \hat Y_t = \begin{pmatrix}
       \hat Y^0_t \\ \hat Y ^1_t \\ \hat Y^2_t
   \end{pmatrix} := 
   \begin{pmatrix}
        Y_t \\ \int_0^t k^+(r,Y_r) (dY_t - h_1(r,Y_r) \, dr) \\ \int_0^t (I- k^+(r,Y_r) k(r,Y_r)) \, dW_r    \end{pmatrix}
\end{equation*}
 and $\bar{\mathfrak{f}}:[0,T] \times \R^{d_X} \times \R^{d_Y} \to \mathrm{Lin}(\R^{d_Y + 2d_W},\R^{d_X})$ defined by \begin{equation*}
     \bar{\mathfrak{f}}(t,x,y) \hat y = \bar f(t,x,y) \hat y^1 + \bar f(t,x,y) \hat y^2 \qquad \text{for any $\hat y = (\hat y^0, \hat y^1, \hat y^2)^\top \in \R^{d_Y + 2d_W}$}.
 \end{equation*} 
 Similarly, we rewrite equation \eqref{eq:stoch-z-1} as \begin{equation*}
     dZ_t = Z_t \mathfrak{h}(t,X_t,Y_t) \, d \hat Y_t
 \end{equation*}
 where $\mathfrak{h}:[0,T] \times \R^{d_X} \times \R^{d_Y} \to \mathrm{Lin}(\R^{d_Y + 2d_W},\R)$ is defined by $\bar{\mathfrak{h}}(t,x,y) \hat y = h_2^\top (t,x,y) \hat y^1 + h_2^\top (t,x,y) \hat y^2$.

 \begin{proposition}
     Let \cref{assumptionCPE} be in force. For any $s,t \in [0,T]$ define \begin{equation*}
         \hat{\mathbb{Y}}^{\text{It\^o}}_{s,t}(\omega) := \left(\int_s^t (\hat Y_r - \hat Y_s) \, d \hat Y_r\right) (\omega). 
     \end{equation*}
     Then $\hat{\BY}^{\text{It\^o}}(\omega) := (\hat Y(\omega), \hat{\mathbb{Y}}^{\text{It\^o}}(\omega)) \in \mathscr{C}^\alpha([0,Y], \R^{d_Y + 2d_W})$, for any $\alpha \in (\frac{1}{3}, \frac{1}{2})$ and for $\mathbb{P}$-almost every $\omega \in \Omega$.  
 \end{proposition}
 \begin{proof}
     Note that $\hat Y$ is a continuous semimartingale. By the properties of the Moore--Penrose inverse (see \cite[Appendix]{CP24}), $[t \mapsto k^+(t,Y_t)]$ is progressively $\{\mathcal{F}^Y_t\}_t$-measurable and $[t \mapsto k^+(t,Y_t) k(t,Y_t)]$ is uniformly bounded on $[0,T]$. 
     The assertion follows as a standard rough path theory result (see, for instance, \cite[Proposition 2.11]{BBFP25}). 
 \end{proof}

Let $\alpha \in (\frac{1}{3}, \frac{1}{2}]$. 
Let $\BY = (\pi_1 \BY, \pi_2 \BY) \in \mathscr{C}^{0,\alpha,1}([0,T], \R^{d_Y})$ and let $\hat \BY = (\pi_1 \hat \BY, \pi_2 \hat \BY) \in \mathscr{C}^\alpha([0,T], \R^{d_Y + 2d_W})$ such that $(\pi_1 \hat \BY)_t^i = (\pi_1 \BY)_t^i$ for any $t \in [0,T]$ and for any $i=1,\dots, d_Y$. Assume \begin{equation*}
    \delta [\hat \BY]_{0,t} = \int_0^t [\dot{\hat{\BY}}]_r \, dr
\end{equation*}
for some $[t \mapsto [\dot{\hat \BY}]_t] \in L_1([0,T];\R^{d_Y + 2d_W} \otimes \R^{d_Y + 2d_W})$ not necessarily continuous. 
Define $\bar{\mathfrak{f}}(\cdot, \cdot \, ; \BY) : [0,T] \times \R^{d_X} \to \mathrm{Lin}(\R^{d_Y + 2d_W}; \R^{d_X})$ via \begin{equation*}
    \bar{\mathfrak{f}}(\cdot, \cdot \, ; \BY) \hat y := \bar{f}(\cdot, \cdot \, ; \BY) \hat y^1 + \bar{f}(\cdot, \cdot \, ; \BY) \hat y^2, \qquad \text{for any $\hat y= (\hat y^0, \hat y^1, \hat y^2)^\top \in \R^{d_Y + 2d_W}$},
\end{equation*} and define $\mathfrak{h}(\cdot, \cdot \, ; \BY) :[0,T] \times \R^{d_X} \to \mathrm{Lin}(\R^{d_Y + 2d_W}, \R)$ as $\mathfrak{h} (\cdot, \cdot \, ; \BY) \hat y = h_2^\top(\cdot, \cdot \, ; \BY) \hat y^1 + h_2^\top(\cdot, \cdot \, ; \BY) \hat y^2$. 
Consider the following system of RSDEs: 
\begin{align}
    dX_t^{\hat \BY} &=  \bar b(t,X^{\hat \BY}_t; \BY) \, dt + \sigma(t,X^{\hat \BY}_t;\BY) \, dB_t + \bar{\mathfrak{f}}(t,X^{\hat \BY}_t; \BY) \,  d \hat{\BY}_t
    \label{eq:roughsignal_degenerate} \\
    dZ^{\hat \BY}_t &= Z^{\hat \BY}_t \mathfrak{h}(t,X^{\hat \BY}_t;\BY) \, d \hat \BY_t, \quad Z^{\hat \BY}_0 = 1. \label{eq:roughZ_degenerate} 
\end{align}
Notice that it formally coincides with the one we considered in \eqref{eq:roughSDE_filtering}, \eqref{eq:Girsanovexponential_filtering}.
It is therefore not difficult to deduce an existence-and-uniqueness result for the corresponding rough Zakai (and Kushner--Stratonovich) equation. We briefly summarize here the main ideas. 

\begin{assumption}{ED} \label{assum:E_degenerate}
    Assume that \cref{assum:E} are satisfied with $b(\cdot, \cdot \, ; \BY)$ replaced by $\bar b(t, \cdot \, ; \BY)$, $f(\cdot, \cdot \, ; \BY)$ replaced by $\bar f(t, \cdot \, ; \BY)$ and $h(\cdot, \cdot \, ; \BY)$ replaced by $h_2(t, \cdot \, ; \BY)$. 
\end{assumption}

\begin{theorem}
    Under \cref{assum:E_degenerate}, equations \eqref{eq:roughsignal_degenerate} and \eqref{eq:roughZ_degenerate} are well-posed. 
\end{theorem}

\begin{proof}
    It is sufficient to note that $[t \mapsto b(t,\cdot\, ; \BY)]$ is a (deterministic) bounded Lipschitz vector field in the sense of \cite[Definition 4.1]{FHL21} and that both $\bar{\mathfrak{f}}(\cdot,\cdot\, ; \BY)$ and $\mathfrak{h}(\cdot,\cdot\, ; \BY)$ can be suitably controlled by $\hat Y = \pi_1 \hat \BY$.
    Indeed, one can define
    \begin{equation*}
        \bar{\mathfrak{f}}'_{\kappa \lambda}(t,\cdot \, ; \BY):=
        \begin{cases}
            \bar f'_{\kappa \lambda} (t,\cdot \, ; \BY) &  d_Y +1 \le \kappa \le d_Y + d_W , \ 1 \le \lambda \le d_Y \\
            \bar f'_{(\kappa - d_W) \lambda} (t,\cdot \, ; \BY) &  d_Y + d_W + 1 \le \kappa \le d_Y + 2d_W , \ 1 \le \lambda \le d_Y \\
            0 &  \text{otherwise}
        \end{cases}
    \end{equation*}
    and notice that $[t \mapsto (\bar{\mathfrak{f}}(t,\cdot\, ; \BY), \bar{\mathfrak{f}}(t,\cdot\, ; \BY))] \in \mathscr{D}_{\hat Y}^{2\alpha}\mathcal{C}^\beta_b$ in the sense of \cref{def:stochasticcontrolledvectorfields_intro}. A similar argument applies, mutatis mutandis, to $\mathfrak{h}$. 
\end{proof}

\begin{assumption}{UD} \label{assumptionU_degenerate} 
Let \cref{assumptionU} hold with $b(\cdot, \cdot \, ; \BY)$ replaced by $\bar b(t, \cdot \, ; \BY)$, $f(\cdot, \cdot \, ; \BY)$ replaced by $\bar f(t, \cdot \, ; \BY)$ and $h(\cdot, \cdot \, ; \BY)$ replaced by $h_2(t, \cdot \, ; \BY)$.
\end{assumption}

\begin{theorem}
    For any $t \in [0,T]$ and for any $\varphi:\R^{d_X} \to \R$ measurable and bounded, define 
    \begin{equation*} \label{eq:roughZakai_degenerate}
    \mu_t^{\hat \BY}(\varphi) := \E(\varphi(X_t^{\hat \BY}) Z_t^{\hat \BY}). 
    \end{equation*}
    Then, under \cref{assum:E_degenerate} and \cref{assumptionU_degenerate}, $\mu^{\hat \BY}$ is the unique measure-valued solution to the following equation starting from $\mu_0^{\hat \BY} = \mathrm{Law}(\xi^{\hat \BY})$, in the sense of \cref{def:solutionroughZakai}: \begin{equation} \label{eq:roughZakaiequation_degenerate}
        \mu_t(\varphi) = \mu_0(\varphi) + \int_0^t \mu_r(\mathfrak{A}^{\hat \BY}_r \varphi) \, dr + \int_0^t \mu_r(\mathfrak{G}^{\hat \BY}_r \varphi) \, d \hat \BY_r, 
    \end{equation}
    where \begin{align*}
        &\mathfrak{A}^{\hat \BY}_t \varphi(x) = \sum_{i,j=1}^{d_X} \Big( \sum_{\theta=1}^{d_B} \sigma_\theta^i(t,x;\BY) \sigma_\theta^j(t,x;\BY) + \sum_{\kappa,\lambda=1}^{d_Y + 2d_W} \bar{\mathfrak{f}}_\kappa^i (t,x;\BY) [\dot{\hat \BY}]^{\kappa \lambda}_t \bar{\mathfrak{f}}_\lambda^j (t,x;\BY)  \Big) \partial^2_{ij} \varphi(x)  \\
        &  \quad + \sum_{i=1}^{d_X} \Big(\bar b^i(t,x;\BY)- \sum_{\kappa=1}^{d_Y} \bar f_\kappa^i(t,x;\BY) h_2^\kappa(t,x;\BY) + \sum_{\kappa,\lambda=1}^{d_Y + 2d_W} \bar{\mathfrak{f}}_\kappa^i(t,x;\BY) [\dot{\hat \BY}]_t^{\kappa\lambda} \mathfrak{h}_\lambda(t,x;\BY) \Big) \partial_i \varphi (x) \\
        & (\mathfrak{G}^{\hat \BY}_t)_\kappa \varphi(x) = \sum_{i=1}^{d_X} \bar{\mathfrak{f}}_\kappa^i(t,x;\BY) \partial_i \varphi(x) + \mathfrak{h}_\kappa (t,x;\BY) \varphi(x) \qquad \kappa=1,\dots,d_Y+2d_W \\
         & (\mathfrak{G}^{\hat \BY}_t)'_{\kappa \lambda} \varphi(x) = \sum_{i=1}^{d_X} (\bar{\mathfrak{f}}'_{\kappa \lambda})^i(t,x;\BY) \partial_i \varphi(x) + \mathfrak{h}'_{\kappa \lambda} (t,x;\BY) \varphi(x) \qquad \kappa,\lambda = 1,\dots,d_Y+2d_W . 
    \end{align*}
\end{theorem}

\begin{proof}
    Straighforward modification of the proofs of \cref{thm:existence} and \cref{thm:uniqueness}.  
\end{proof}

\textit{Consistency with stochastic theory.} 
    Upon randomisation $\hat \BY \rightsquigarrow \hat \BY^{\text{It\^o}} (\omega)$, all the terms in \eqref{eq:roughZakaiequation_degenerate} turn into the correct stochastic terms appearing in the stochastic Zakai equation \eqref{equ:ZakDeg}.
    Randomisation is discussed in detail in \cref{section:bridgingroughandstochasticfiltering} in the non-degenerate case; the present remark implies that it also works in the degenerate case. 
    Arguing as we do in \cref{thm:mu_as_process} below in the non-degenerate case, it is possible to obtain that, under suitable assumptions, 
    \begin{equation*}
        \hat \mu_t (\omega) (\varphi) := \mu_t^{\hat \BY}(\varphi) | _{\hat \BY = \hat \BY^{\text{It\^o}}(\omega)} = \E( \varphi(X_t) Z_t \mid \mathcal{F}^{\hat{Y}}_t) (\omega)
    \end{equation*}
    where $X,Z$ are defined as in \eqref{eq:roughsignal_degenerate}, \eqref{eq:roughZ_degenerate}. Moreover, $\hat \mu$ solves the following stochastic PDE: 
        \begin{equation} \label{eq:intermediateZakai_degenerate}
        d \hat \mu_t(\varphi) = \hat \mu_t (\mathfrak{A}_t\varphi) \, dt + \sum_{\kappa=1}^{d_Y + 2d_W} \hat \mu_t((\mathfrak{G}_t)_\kappa \varphi) \, d\hat Y^\kappa_t, 
    \end{equation}
    where the random time-dependent differential operators $\mathfrak{A}_t$ and $\mathfrak{G}_t=((\mathfrak{G}_t)_1,\dots,(\mathfrak{G}_t)_{d_Y + 2d_W})$ are defined via \begin{align*}
        \mathfrak{A}_t \varphi(x) &= \sum_{i,j=1}^{d_X} \Big( \sum_{\theta=1}^{d_B} \sigma_\theta^i(t,x,Y_t) \sigma_\theta^j(t,x,Y_t) + \sum_{\kappa=1}^{d_Y} \bar f_\kappa^i (t,x,Y_t) \bar f_\kappa^j (t,x,Y_t)  \Big) \partial^2_{ij} \varphi(x) \\
        & \quad + \sum_{i=1}^{d_X} \bar b^i(t,x,Y_t) \partial_i \varphi (x)  \\
        (\mathfrak{G}_t)_\kappa \varphi(x) &= \sum_{i=1}^{d_X} \bar{\mathfrak{f}}_\kappa^i(t,x,Y_t) \partial_i \varphi(x) + \mathfrak{h}_\kappa (t,x,Y_t) \varphi(x) .
    \end{align*}
    In the previous identities we used the fact that, by construction and $\mathbb{P}$-a.s., \begin{align*} 
        \sum_{\kappa,\lambda=1}^{d_Y + 2d_W} \bar{\mathfrak{f}}_\kappa^i (t,x,Y_t) \frac{d \langle \hat Y^\kappa, \hat Y^\lambda \rangle}{dt} \bar{\mathfrak{f}}_\lambda^j (t,x, Y_t) &= \sum_{\kappa=1}^{d_Y} \bar f_\kappa^i (t,x,Y_t) \bar f_\kappa^j(t,x,Y_t) \\ 
        \sum_{\kappa,\lambda=1}^{d_Y + 2d_W} \bar{\mathfrak{f}}_\kappa^i (t,x,Y_t) \frac{d \langle \hat Y^\kappa, \hat Y^\lambda \rangle}{dt} \mathfrak{h}_\lambda^j (t,x, Y_t) &= \sum_{\kappa=1}^{d_Y} \bar f_\kappa^i (t,x,Y_t) h_2^\kappa(t,x,Y_t). 
    \end{align*}
    Recall that we are interested in the unnormalised filter $\mu_t(\omega) (\varphi) := \E(\varphi(X_t) Z_t \mid \mathcal{F}^Y_t)(\omega)$ and that $\mathcal{F}^Y_t \subseteq \mathcal{F}^{\hat Y}_t$.
    Hence, by the tower property of conditional expectations, we have 
    \begin{equation*}
        \mu_t(\omega) (\varphi) = \E( \hat \mu_t (\cdot) (\varphi) \mid \mathcal{F}_t^Y)(\omega).
    \end{equation*}
    This deduces that $\mu_t(\omega)$ solves an equation of the form \eqref{equ:ZakDeg} follows from taking conditional expectation on both sides of \eqref{eq:intermediateZakai_degenerate} and noting that, thanks to \cite[Lemma 3.13]{CP24}, \begin{equation*}
        \E\Big( \int_0^t \sum_{\kappa=d_Y + d_W +1}^{d_Y + 2d_W}  \hat \mu_r((\mathfrak{G}_t)_\kappa \varphi)  \, d\hat Y^\kappa_r \mid \mathcal{F}^Y_t \Big) = 0. 
    \end{equation*}

\section{Rough Kalman--Bucy theory} \label{section:roughKB}
\subsection{Stochastic theory} 
Following \cite{HP88}, \cite[Sec 6.3]{pardoux2006filtrage}, \cite[Sec 8.1]{BC09}, we consider an ``extended Kalman--Bucy'' situation with signal/observation pair $(X,Y)$, with values in $\R^{d_X} \times \R^{d_Y}$, and It\^odynamics
\begin{align}
    dX_t &= (C(t;Y) X_t + c(t;Y)) \, dt + \Sigma(t;Y) \, dB_t + \sum_{\kappa=1}^{d_Y} (F_\kappa(t;Y) X_t + f_\kappa(t;Y)) \, dY^\kappa_t , \label{eq:signal_linear} \\ 
    dY_t &= (kk^\top)(t;Y) ( H (t;Y) X_t  + h (t;Y) ) \,  dt + k(t;Y) \, d B^{\perp}_t.
    \label{eq:obs_linear}  
\end{align}
The stochastic setup is as in Appendix \ref{sec:review}, with $\{(B_t, B^{\perp}_t)\}$ independent Brownian motions under $\mathbb{P}^o$. All coefficients are assumed bounded (locally in time) progressively measurable w.r.t. $\mathcal{F}_t^Y := \sigma(Y_s : s \le t), \ t \ge 0,$ of appropriate dimensions, with $C_t \in \mathrm{Lin}(\R^{d_X}, \R^{d_X}), c_t \in \R^{d_X}, \Sigma_t \in \mathrm{Lin}(\R^{d_B}, \R^{d_X}), 
        F_t \in \mathrm{Lin}(\R^{d_X}, \mathrm{Lin}(\R^{d_Y},\R^{d_X})), f_t \in  \mathrm{Lin}(\R^{d_Y},\R^{d_X}), H_t \in  \mathrm{Lin}(\R^{d_X}, \R^{d_Y}), 
        \bar h_t \in \R^{d_Y}.$
We consider here non-degenerate observation noise,
assuming invertible $k=k_t(y) \in \mathrm{Lin}(\R^{d_Y},\R^{d_Y})$, measurable in $t$, Lipschitz in $Y$. 
We also suppose that $X_0 \sim \mathcal{N} (m_0, v_0)$ is independent of $\{(B_t, B^{\perp}_t)\}$. 
Recall $\mathcal{F}_t^Y := \sigma(Y_s : s \le t)$,
we then have\footnote{$H^\kappa$ denotes the $\kappa$-th row of the matrix $H$.}

\begin{theorem}  
Let Assumptions (A1)-(A6) in \cite{HP88} be in force and assume $k$ invertible and Lipschitz in $Y$. 
Then it holds for all $t \ge 0$ 
and all $\varphi \in
  \mathcal{B}_b(\R^{d_X})$, with probability one,
$$ \mathbb{E}^o \left[ \varphi (X _t)  | \mathcal{F}_t^Y
\right] = \frac{\mathbb{E} \left[ \varphi (X _t) Z _t | \mathcal{F}_t^Y
\right]}{\mathbb{E} \left[ Z_t | \mathcal{F}_t^Y
\right]} =  \mathcal{N} (m _t (\omega), v _t (\omega)) ( \varphi)        
$$
where
\begin{align} 
Z_t 
&= \exp \bigg( \int_0^t [H(s;Y) X_s + h(s;Y)]^\top \, dY_s \notag  \\
& \qquad - \frac12 \int_0^t [H(s;Y) X_s + h(s;Y)]^\top  \frac{d \langle Y \rangle_s}{ds} [H(s;Y) X_s + h(s;Y)] \, ds   \bigg)
\label{eq:exponential_stochasticKalman} 
\end{align}
and $\{(m_t, v_t), \, t \geq 0\}$ is the unique strong solution started from $(m_0,v_0)$ to the following SDE system 
  \be\label{eq:stoch-recatti}
  \begin{split}
    d m_t 
    & = \Big(C (t; Y) m_t + c (t; Y) + \sum_{\kappa,\lambda=1}^{d_Y} F_\kappa (t; Y) v_t H^\lambda(t; Y) \frac{d \langle Y \rangle^{\kappa \lambda}}{dt} \Big) \, dt + \sum_{\kappa=1}^{d_Y} (F_\kappa (t; Y) m_t + f_\kappa (t; Y)) \, dY_t^\kappa \\
    & \quad + \sum_{\kappa=1}^{d_Y} v_t H^\kappa(t; Y)  \left(dY^\kappa_t - \sum_{\lambda=1}^{d_Y} \frac{d \langle Y \rangle ^{\kappa\lambda}}{dt} (H^\lambda (t; Y) m_t + h^\lambda (t; Y))  \, dt \right),\\
    dv_t 
    & = \Big( C (t; Y) v_t + v_t C(t; Y)^\top + \Sigma (t; Y) \Sigma(t; Y)^{\top} + \sum_{\kappa,\lambda=1}^{d_Y} F_\kappa (t; Y) v_t F_\lambda(t; Y)^{\top} \frac{d \langle Y \rangle ^{\kappa \lambda}}{dt} \\ 
    &\quad - v_t H(t; Y)^{\top} \frac{d \langle Y \rangle }{dt} H (t; Y)  v_t \Big) \, dt + \sum_{\kappa=1}^{d_Y} (F_\kappa (t; Y) v_t + v_t F_\kappa (t; Y)^{\top} ) \, dY_t^\kappa.
      \end{split}
  \ee
\end{theorem}

 \begin{proof}
     The theorem in the case $k(t;Y) = I$ (identity matrix) follows from \cite[Theorem 2.1, 2.2]{HP88}.
     As already suggested in \cite[Remark 2.1]{HP88}, since $k$ is non-singular one can recast the problem with $W_t := \int_0^t k(s,Y)^{-1}\, dY_s$, noting that $Y$ and $W$ generate the same  filtration. 
 \end{proof}

\subsection{Rough Riccati equation} \label{sec:RRE}

We start with a self-contained discussion of the (Riccati) rough differential equations that arises when replacing $Y$ in the previous theorem by a generic rough path. Due to quadratic drift terms,  the well-posedness (non-explosion!) of such equations hinges on the availability of a ``good'' sign, which explains the restriction on the class of rough paths for which ones has global well-posedness. Write $M \succeq 0$ to indicate that $M$ is a symmetric, positive semi-definite matrix.

\begin{theorem} \label{thm:RRDEwell} Let $\BY \in \mathscr{C}^{0,\alpha,1}$ and let $(Q(\cdot \, ; \BY),R(\cdot \, ; \BY)) :[0,T] \to (\R^{d_X} \otimes \R^{d_X}) \times  (\R^{d_Y} \otimes \R^{d_Y})$ locally bounded and measurable, with $Q(t;\BY), R(t;\BY) \succeq 0$ for any $t$. 
 Let $v_0^{\tmmathbf{Y}}$ be a symmetric and positive semi-definite matrix in $\R^{d_X}\otimes\R^{d_X}$. 
  Assume that $[t \mapsto (F(t;\BY), F'(t;\BY))] \in \mathscr{D}^{2 \alpha}_Y$, all other coefficients $(C(\cdot \, ;\BY),
  H(\cdot \, ;\BY))$ locally bounded and measurable in time.
   Then the (Riccati) rough differential equation, 
    \begin{align}
            & d  v_t^{\BY}   =    \Big( C(t;\BY) v^\BY_t + v_t^\BY C(t;\BY)^\top + Q(t;\BY)
            + \sum_{\kappa,\lambda=1}^{d_Y}  F_\kappa(t;\BY) v_t^{\BY} F_\lambda(t;\BY)^\top \dot{[\BY]}_t^{\kappa \lambda} \notag  \\
       & \quad - v_t^{\BY} H(t;\BY)^\top R(t;\BY)
       H(t;\BY) v_t^{\BY} \Big) dr  + 
       \sum_{\kappa=1}^{d_Y} (F_\kappa (t; \BY) v^{\BY}_t + v^{\BY}_t F_\kappa (t; \BY)^{\top} ) \, d\BY_t^\kappa,    \label{eq:roughRiccatiDE}   
        \end{align}
    has a global solution and 
    $v_t^{\tmmathbf{Y}} \succeq 0$ for all $t \ge 0$. Moreover, the solution is locally Lipschitz in its data, notably w.r.t. $\BY \in \mathscr{C}^{0,\alpha,1}$ and $(F, F') \in \mathscr{D}^{2 \alpha}_Y$.
\end{theorem}

\begin{proof}
    To simplify notation, we omit explicit dependence of the coefficients on $\BY$, writing, for example, $C_t$ instead of $C(t;\BY)$. 
  We first treat the (classical) case $F = 0$
  , i.e. 
  \[ \dot{v}_t = C_t v_t + v_t C_t^{\top} + Q_t - v_t H_t^{\top} R_t
     H_t v_t   \] 
  but with arguments chosen carefully with regard to later extension to the rough case, in particular avoiding Grownwall arguments (cf. \cite{hofmanova2016rough}).  
  
  \noindent {\em Step (0): local well-posedness.} It is clear from basic ODE theory
  (coefficients are smooth in the state-variable and measurable in $t$) that
  there there exists a unique maximal solution $v$. 
  
   \noindent{\em Step (1): symmetry and semi-definiteness.} 
   Let $\tilde{v}_t \assign
  v_t^{\top}$. Since $Q_t,R_t$ are symmetric, $\tilde{v}$ satisfies the same
  ODE with $\tilde{v}_0 = v_0^{\top} = v_0$. By uniqueness, $v_t =
  \tilde{v}_t$ on the maximal interval, hence $v_t = v_t^{\top}$. We now
  discuss positivity. 
  Set $K_t \assign \frac12 v_t H_t^{\top} R_t H_t$ 
  and
  rewrite the Riccati equation as
  \[ \dot{v}_t = (C_t - K_t) v_t + v_t  (C_t - K_t)^{\top} + Q_t 
     . \]
  Let $\Psi (t, s)$ solve $\partial_t \Psi = (C_t - K_t) \Psi$, $\Psi (s, s) =
  I \in \R^{d_X}\otimes \R^{d_X}$. Then $\Psi (t, 0) w_0 \Psi (t, 0)^{\top}$ solves the linear problem
  $\dot{w}_t = (C_t - K_t) w_t + w_t  (C_t - K_t)^{\top}$, and by Duhamel's
  principle (a.k.a.\ variation of constants) 
  \[ v_t = \Psi (t, 0) v_0 \Psi (t, 0)^{\top} + \int_0^t \Psi (t, s) Q_s 
     \Psi (t, s)^{\top} ds .\]
   {\noindent}We note that  $\Psi (t, 0) v_0 \Psi (t, 0)^{\top}, \Psi (t, s) Q_s \Psi(t,s)^\top\succeq 0$ because $v_0,Q_s\succeq 0$. 
  Thus the integrand is positive semidefinite and $v_t \succeq 0$
  for all $t$.     
  
  \noindent{\em Step (2): comparison with the linear Lyapunov-flow.}
  Let $S$ solve
  \[ \dot{S}_t = C_t S_t + S_t C_t^{\top} + Q_t, \quad
     S_0 = v_0 \]
  and consider the difference $D_t \assign S_t - v_t$. Subtracting the
  equations gives
  \[ \dot{D}_t = C_t D_t + D_t C_t^{\top} + M_t, \quad D_0 = 0, \quad \]
  with $M_t \assign v_t H_t^{\top} R_t H_t v_t \succeq 0$ by step (1).
  Let $\Phi (t, s)$ solve $\partial_t \Phi = C_t \Phi$, $\Phi (s, s) = I\in \R^{d_X}\otimes \R^{d_X}$. By
  Duhamel's principle,
  \[ D_t = \int_0^t \Phi (t, s) M_s \Phi (t, s)^{\top} ds \succeq 0, \]
  and hence $v_t \preceq S_t$ for all $t$. Since $\| v_t \|_{\tmop{op}}
  \leqslant \tmop{tr} (v_t) \leqslant \tmop{tr} (S_t)$ and the linear
  equation does not explode, this suffices to rule out explosion
  for $v_t$ on any finite intervals.  
  
  \noindent {\em Extension to rough case.} W.l.o.g.\ we can assume $\BY \in \mathscr{C}^{0,\alpha}_g$. Indeed, arguing as in \cref{prop:equivalentformofroughZakai}, one can define $\BY^\circ := \BY + \frac12 (0, \delta [\BY])$, $\circ d\BY:=d \BY^\circ$, and equivalently write \eqref{eq:roughRiccatiDE} as
  \begin{equation} \label{equ:dvt}
  d v_t = (C_t v_t + v_t C_t^{\top} + Q_t  
  - v_t H_t^{\top} R_t H_t v_t) \, d t + (F_t v_t +
     v_t F_t^{\top}) \circ d \BY_t. 
     \end{equation}
  We examine each step above. \tmtextit{Step (0)}: Local well-posedness comes
  from a familiar Picard argument, with integral formulation that requires
  $(F, F') \in \mathscr{D}^{2 \alpha}_Y$ for the rough integral and
  measurability in $t$ for the drift coefficients. Local Lipschitzness of the
  coefficients in the state-variable (at most quadratic in $v$) is obvious. By linearity of
  the rough integrand in the solution, the Picard map is contracting and ones obtains a maximal solution, e.g. \cite[Ex 8.4]{FH20}.
  
  \tmtextit{Step (1):} Symmetry follows from local well-posedness. For
  positivity, consider $\Psi (t, s)$ as solution to the linear RDE (as before the quadratic term with $R$ in \eqref{equ:dvt} is absorbed in $K$)
  \[ d_t \Psi = (C_t - K_t) \Psi \, d t + (F_t, F_t') \Psi \circ d \BY_t,
     \quad \Psi (s, s) = I, \]
  from which
  \[ v_t = \Psi (t, 0) v_0 \Psi (t, 0)^{\top} + \int_0^t \Psi (t, s)  
     Q_s   \Psi (t, s)^{\top} \, ds.
  \]

\noindent  \tmtextit{Step (2)} amounts to compare the affine linear RDE \eqref{equ:dvt} with
  \begin{align*}
      d S_t &= (C_t S_t + S_t C_t^{\top} + Q_t ) \, dt + (F_t S_t + v_t S_t^{\top}) \circ d \BY_t, \quad    S_0 = v_0
  \end{align*}
  and consider the difference $D_t \assign S_t - v_t$. Subtracting the
  equations gives
  \[ d D_t = (C_t D_t + D_t C_t^{\top})  \, d t  + (F_t D_t + v_t D_t^{\top}) \circ d \BY_t + M_t \,   d t, \quad D_0 = 0,
  \]
  with $M_t \assign v_t H_t^{\top} R_t H_t v_t \succeq 0$ by \textit{(1)}.
  Let now
  $\Phi (t, s)$ solve $d_t \Phi = C_t \Phi \, d t + (F_t, F_t') \Phi \circ d \BY_t$ with
  $\Phi (s, s) = I$. By Duhamel,
  \[ D_t = \int_0^t \Phi (t, s) 
     M_s  \Phi (t, s)^{\top} \, ds \succeq 0.  \]
  
 \noindent {\em Stability.} At last, the stated local Lipschitz property follows readily from standard results on RDEs, applicable after localization.
\end{proof}

\subsection{Rough Kalman--Bucy filtering} \label{sec:RKBF}
We rewrite the signal dynamics of \eqref{eq:signal_linear} in rough stochastic form,
\begin{align}
    dX^\BY_t &= (C(t;\BY) X_t^\BY + c(t;\BY)) \, dt + \Sigma(t;\BY) \, dB_t + (F(t;\BY) X^\BY_t + f(t;\BY)) \, d\BY_t , \label{eq:roughsignal_linear} \\  
    X^\BY_0 &= \xi^\BY\in\mathcal{F}_0, \notag
\end{align}
noting that under \cref{assumptions_linear} below this is a linear RSDE 
which admits a unique $L_2$-integrable solution, cf. \cref{def:integrablesolutionsRSDEs} and references thereafter. Here and below, $\BY \in 
    \mathscr{C}^{0,\alpha,1}$ so that $[\BY]$ is continuously differentiable. The rough likelihood process, rough counterpart of \eqref{eq:exponential_stochasticKalman}, is then given by
\begin{multline}
     Z^\BY_t = \exp \bigg( \int_0^t [H(s;\BY) X^\BY_s + h(s;\BY)]^\top \, d\BY_s \\ - \frac12 \int_0^t (H(s;\BY) X^\BY_s + h(s;\BY))^\top \dot{[\BY]}_s (H(s;\BY) X^\BY_s + h(s;\BY)) \, ds \bigg) , 
    \label{eq:roughZ_linear}
\end{multline}
which under \cref{assumptions_linear} is well-defined.\footnote{The stochastic rough integral in the exponent is well-defined, using that $(X^\BY,F(\cdot \, ; \BY) X^\BY +f(\cdot \, ; \BY)) \in \mathbf{D}_Y^{2\alpha} L_2$, as is its product with $(H(\cdot \, ; \BY),H'(\cdot \, ; \BY))\in \mathscr{D}_Y^{2\alpha}\mathrm{Lin} \subset \mathbf{D}_Y^{2\alpha} L_\infty\mathrm{Lin}$. In this case, the stochastic rough integral exists by \cite[Theorem 3.4]{FHL21}. Of course, the integral  $(h(\cdot \, ; \BY),h'(\cdot \, ; \BY))$ against $d\BY$ is a classical deterministic rough integral.}
We note that, in contrast to \cref{sec:RoughFiltering}, integrability of the $Z^\BY$ cannot be obtained by general results of John-Nirenberg type, as previously done in \cref{thm:wellposedness_roughfiltering}. We will overcome this problem below via tools from Wiener-It\^ochaos theory. 

%


\begin{assumption}{L} \label{assumptions_linear}
    For any $\BY \in  \mathscr{C}^{0,\alpha,1}$, the functions 
    \begin{align*}
    (C(\cdot \, ; \BY), c(\cdot \, ; \BY), \Sigma(\cdot \, ; \BY)) &: [0,T] \to \mathrm{Lin}(\R^{d_X}, \R^{d_X}) \times \R^{d_X} \times \mathrm{Lin}(\R^{d_B}, \R^{d_X}) \\
    (F(\cdot \, ; \BY), f(\cdot \, ; \BY)) &: [0,T] \to \mathrm{Lin}(\R^{d_X}, \mathrm{Lin}(\R^{d_Y},\R^{d_X})) \times \mathrm{Lin}(\R^{d_Y},\R^{d_X})  \\
    (H(\cdot \, ; \BY), h(\cdot \, ; \BY)) &: [0,T] \to  \mathrm{Lin}(\R^{d_X}, \R^{d_Y}) \times \R^{d_Y} 
     \end{align*}
    are measurable and bounded. 
    There exist some measurable and bounded functions \begin{align*}
        F'(\cdot \, ; \BY) &: [0,T] \to \mathrm{Lin}(\R^{d_X}, \mathrm{Lin}(\R^{d_Y},\mathrm{Lin}(\R^{d_Y}, \R^{d_X}))) \equiv \mathrm{Lin}(\R^{d_X}, \mathrm{Lin}(\R^{d_Y} \otimes \R^{d_Y}, \R^{d_X}))  \\
        H'(\cdot \, ; \BY) & : [0,T] \to  \mathrm{Lin}(\R^{d_X},\mathrm{Lin}(\R^{d_Y}, \R^{d_Y}))
    \end{align*}
    such that $[t \mapsto(F(t;\BY),F'(t;\BY)] , \, [t \mapsto (H(t;\BY),H'(t;\BY)] \in \mathscr{D}_Y^{2\alpha}\mathrm{Lin}$ in the sense of \cref{def:stochasticcontrolledlinearvectorfield}. 
    Moreover, there exist some measurable and bounded functions $$ (f'(\cdot \, ; \BY) , h' (\cdot \, ; \BY)) : [0,T] \to \mathrm{Lin}(\R^{d_Y},\mathrm{Lin}(\R^{d_Y}, \R^{d_X})) \times \mathrm{Lin}(\R^{d_Y},\R^{d_Y}) $$
    such that $[t \mapsto (f(t;\BY), f'(t;\BY))], \, [t \mapsto (h(t;\BY), h'(t;\BY))] \in \mathscr{D}_Y^{2\alpha}$ in the sense of \cref{def:stochasticcontrolledroughpaths}.
\end{assumption}

\begin{theorem}[Rough Kalman--Bucy filter] \label{thm:RKBf}
  Let \cref{assumptions_linear} hold and assume that $\xi^\BY$ is Gaussian.
  Let $\BY \in \mathscr{C}^{0,\alpha,1}$ with monotone bracket, $[\dot{\tmmathbf{Y}}] \succeq 0$.
  Then for all $t \in [0,T]$ and all $\varphi \in
  \mathcal{B}_b(\R^{d_X})$ we have 
  \be \varsigma^{\BY}_t (\varphi)  \equiv \frac{\mathbb{E} [\varphi (X_t^{\BY})
     Z^{\BY}_t]}{\mathbb{E} [Z^{\BY}_t]} = 
     \mathcal{N}(m^{\BY}_t , v^{\BY}_t ) 
     (\varphi) \label{e:rKBf}
     \ee
  where $\{(m^{\BY}_t, v^{\BY}_t), t \in
  [0,T]\}$ is the unique (global) solution to the RDE system
     \begin{align*}
            &d m^\BY_t =  \Big( C(t;\BY) m^\BY_t + c(t;\BY) + \sum_{\kappa,\lambda=1}^{d_Y} F_\kappa(t;\BY) v_t^\BY H^\lambda(t;\BY) \dot{[\BY]}_t^{\kappa \lambda} \\
        & \quad - v^{\BY}_t H(t;\BY)^\top \dot{[\BY]}_t (H(t;\BY) m_t^\BY + h(t;\BY))  \Big) \, dt +   \big( F(t;\BY) m_t^{\BY } + f(t;\BY) + v^\BY_t H(t;\BY)^\top \big) \, d\BY_t , \\
                   & d  v_t^{\BY}   =    \Big( C(t;\BY) v^\BY_t + v_t^\BY C(t;\BY)^\top + \Sigma(t;\BY) \Sigma(t;\BY)^\top + \sum_{\kappa,\lambda=1}^{d_Y}  F_\kappa(t;\BY) v_t^{\BY} F_\lambda(t;\BY)^\top \dot{[\BY]}_t^{\kappa \lambda}  \\
       & \quad - v_t^{\BY} H(t;\BY)^\top \dot{[\BY]}_t H(t;\BY) v_t^{\BY} \Big)\, dt  + \big(F(t;\BY) v_t^{\BY} + v_t^{\BY} F(t;\BY)^\top  \big) \, d \BY_t, \qquad t \in [0,T] .
        \end{align*}
\end{theorem}

We divide the proof in several steps, and first convince ourselves that the left-hand side of \eqref{e:rKBf} really defines a Gaussian law.

\begin{prop} \label{thm:gaussianity_linear}
Under assumptions of Theorem \ref{thm:RKBf},  $\varsigma^{\BY}_t$ is a Gaussian measure on $\R^{d_X}$, $t \in [0,T]$. 
\end{prop}

\begin{proof} 
We assume $d_X=d_B=1$ for clarity.  The arguments in the general case are similar. 
It is immediate that $X^\BY$ is a continuous Gaussian process, so that, for fixed $t>0$, there are 
deterministic (continuous and square-integrable, respectively) functions $\theta_0, \theta_1$
\[ e^{i \xi X^{\mathbf{Y}}_t} = e^{i \xi \left( \int_0^t \theta_1 (s;{ t}) d B_s
   + \theta_0 (t) \right)} . \]
Also, 
the rough stochastic exponential
$Z_t^{\mathbf{Y}}$,
which (thanks to $[\mathbf{Y}] \neq 0$) involves second chaos terms in the exponent, can be written as
\[ Z_t^{\mathbf{Y}} = \exp \left( \int_0^t { \int_0^t} \psi_2 (r, s;{t}) d B_r d B_s +
   \int_0^t \psi_1 (s{;t}) d B_s + \psi_0 (t) \right), \]
where $\psi_2(r,s{;t})=-\int_0^t (H_\tau \theta_1(\tau{;t}))^\top  \dot{[\BY]}_\tau H_\tau \theta_1 (\tau{;t})d\tau $,  and $\psi_1,\psi_0$ are some deterministic functions.
Let $\Psi_2$ the $L^2$-operator associated to $\psi_2$.
It is evident that $\Psi_2$  is negative,
so that $(I - 2 \Psi_2)$ is strictly positive,
which in turn gives, by Proposition \ref{prop:exp2chaos}, with $h = i \xi \theta_1 + \psi_1$,
\[ \mathbb{E} [e^{i \xi X^{\mathbf{Y}}_t} Z_t^{\mathbf{Y}}] = e^{i \xi
   \theta_0 (t) + \psi_0 (t)} \times {\det}_2 (I - 2 \Psi_2)^{- \frac{1}{2}} 
   \hspace{0.27em} \exp \hspace{-0.17em} \left( \frac{1}{2}  \hspace{0.17em}
   \left\langle (I - 2 \Psi_2)^{- 1} h, \hspace{0.17em} h \right\rangle
   \right) . \]
(See also \cite[Lemma 1]{HP88} for an alternative argument.) It follows immediately that, for fixed $t$, with some second order polynomial
$P_2 (\cdot)$,
\[ \mathbb{E} [e^{i \xi X^{\mathbf{Y}}_t} Z_t^{\mathbf{Y}}] = \exp (P_2
   (\xi)) \Rightarrow \frac{\mathbb{E} [e^{i \xi X^{\mathbf{Y}}_t}
   Z_t^{\mathbf{Y}}]}{\mathbb{E} [Z_t^{\mathbf{Y}}]} = \exp (P_2 (\xi) - P_2
   (0)), \]
which implies the desired Gaussianity.
\end{proof}

The obvious way to obtain dynamic descriptions of mean and covariance comes from the (rough) Kushner--Stratonovich equations, applied with linear and quadratic test functions. 
 This is possible but requires a number of considerations to do with the non-boundedness of the coefficients, also when using the complex exponential as test functions. We instead employ a (somewhat non-standard) approximation argument, assuming at first that $\BY \in \mathrm{s}\mathscr{C}^{0,\alpha;1}$ by which we mean \footnote{Here $C^1$ is the space of continuously differentiable paths. }
 $$
                 \BY = \left(Y, \int \delta Y  \otimes d Y  - \tfrac{1}{2} \delta [\mathbf{Y}]\right), \quad Y \in {C}^1,  [\mathbf{Y}] \in {C}^1 \  \text{with $[\BY]_t^\top = [\BY]_t$},
 $$
 the point being that rough integrals become classical ones: for every $(Z,Z') \in \mathscr{D}_Y^{2\alpha}$,
 $$
 \int (Z_t,Z'_t) \, d \BY_t = \int (Z_t \dot{Y}_t - \tfrac{1}{2} Z'_t \dot{[\BY]}_t) \, d t. 
 $$

 \begin{lem} \label{lem:nga} Let $\alpha \in (1/3,1/2)$. Every $\BY = (Y,\mathbb{Y}) \in \mathscr{C}^{0,\alpha;1}$, non-geometric in general, can be approximated (in $\alpha$-rough path metric) by $\BY^{\varepsilon} = (Y^{\varepsilon},\mathbb{Y}^{\varepsilon}) \in \mathrm{s}\mathscr{C}^{0,\alpha;1}$. Moreover, every $(Z,Z') \in \mathscr{D}_Y^{2\alpha}$ can be approximated (in $2\alpha'$-controlled rough path sense, any $\alpha' < \alpha$) by 
 $(Z_\varepsilon,Z_{\varepsilon}') \in \mathscr{D}_{Y^\varepsilon}^{2\alpha}$. Hence, for any $t>0$ fixed,\footnote{Stronger H\"older type convergence statements hold true but are irrelevant for our purposes.}
 $$ \int_0^t (Z,Z') \,d \BY  = \lim_{\varepsilon \to 0} \int_0^t (Z_\varepsilon,Z_{\varepsilon}') \,d \BY^{\varepsilon}.$$
  \end{lem}
 \begin{proof} A geometric rough path is defined by $\BY^\circ :=  \BY +  \frac{1}{2} (0, \delta [\mathbf{Y}]) \in \mathscr{C}^{0,\alpha}_g$ , necessarily the rough path limit of $\BY^{\circ,\varepsilon} := (Y^{\varepsilon}, \int \delta Y^{\varepsilon}  \otimes d Y^{\varepsilon} )$. Now set $\BY^{\varepsilon}:=  \BY^{\circ,\varepsilon} -  \frac{1}{2} (0, \delta [\mathbf{Y}])$. 
 Obviously $Y^\varepsilon \to Y$ in $\alpha$-H\"older sense, so that the second statement follows from \cite[Ex. 4.8]{FH20}. The conclusion comes from classical stability results for rough integrals.
 \end{proof} 
\begin{proof} [Proof of Theorem \ref{thm:RKBf}] Set $\iota(x) := x$.   By (the proof of) \cref{thm:gaussianity_linear}, we have
\be\label{eq:hmu-anz}
\hat{\mu}_t (\xi) :=  \mu^{\BY}_t ( e^{i
   \xi \iota } ) \equiv \mathbb{E} [e^{i \xi X^{\mathbf{Y}}_t} Z_t^{\mathbf{Y}}]   =  \theta_t e^{i \xi m_t - \frac12
   \xi^2 v_t} = \exp (- \frac12 \xi^2 v_t + i \xi
   m_t + \gamma_t) =: \exp (P^2_t (\xi)) 
\ee
where $P_t^2 (\xi)$ denotes a quadratic in $\xi$, and $\gamma_t = \log \mu_t^\BY(1)$ has some dynamics independent of $\xi$. 
Assuming at first that $\BY \in \mathrm{s}\mathscr{C}^{0,\alpha;1}$, it is easy to see that the maps $t \mapsto v_t,m_t,\gamma_t$ are $\mathcal{C}^1$ in time.
Thus, with classical calculus interpretation of all differentials, 

$$
d \hat{\mu}_t(\xi) = \hat{\mu}_t(\xi) (-\frac12 \xi^2 dv_t + i \xi dm_t + d\gamma_t).
$$
On the other hand, using 
\cref{prop:equivalentformofroughZakai} (with $\circ d\BY_t = d Y_t$ in calculus sense), together with classical Feynman-Kac theory,  we have 
\be\label{eq:strat-zakai}
\hat{\mu}_t(\xi) (-\frac12 \xi^2 dv_t + i \xi dm_t + d\gamma_t)=   \mu_t (\ML^\circ_t e^{i \xi \iota}) \, dt + \mu_t(\MG_t e^{i \xi \iota} ) \,  \circ d\BY_t,
\ee
where, 
{considering scalar objects for simplicity only,}
\begin{align*}
\ML^{\circ}_t : = & \frac{1}{2} \Sigma^2_t \partial_x^2 + \Big( C_t \iota  -  \frac{1}{2} (F^2_t \iota + F'_t \iota) [\dot{\mathbf{Y}}]_t \Big) \partial_x  - \frac{1}{2}  \Big( (H_t \iota )^2 + (H_t F_t \iota  + H'_t \iota ) \Big) [\dot{\mathbf{Y}}]_t,  \\
\MG_t := & F_t \iota \partial_x + H_t \iota .
\end{align*}
Note that
\begin{align*}
\mu_t(\ML^\circ_t e^{i \xi \iota })= & \Big\{- \frac{1}{2} \Sigma^2_t \xi^2 + \Big( C_t -
   \frac{1}{2} (F^2_t + F'_t) [\dot{\mathbf{Y}}]_t \Big) \xi  \partial_{\xi }  + \Big( \frac{1}{2} H^2_t \partial_{\xi}^2
+  i (H_t F_t + H'_t) \partial_{\xi}  \Big) [\dot{\mathbf{Y}}]_t \Big\}  \hat{\mu}_t (\xi)
\end{align*}
where we apply the fact $x \partial_x (e^{i\xi x})= \xi \partial_\xi (e^{i\xi x})$ and $x e^{i\xi x}=-i \partial_\xi (e^{i\xi x})$. In view of \eqref{eq:hmu-anz}, we have
$$
\partial_\xi \hat{\mu}_t(\xi)=  \hmu_t(\xi) (-\xi v_t+ im_t), \ \ \ \partial_\xi^2 \hat{\mu}_t(\xi)=   \hmu_t(\xi) \left( (-\xi v_t + i m_t)^2 -v_t \right).
$$
Then it follows that, 
\begin{multline} \label{eq:muML}
    \mu_t(\ML^\circ_t e^{i \xi \iota })= \Big\{- \frac12 \Sigma^2_t \xi^2 + \left( C_t -
   \frac{1}{2} (F^2_t + F'_t) [\dot{\mathbf{Y}}]_t \right) (-\xi v_t + im_t) \xi\\
   + \frac12 H^2_t \dot{[\BY]}_t \left[(-\xi v_t + i m_t)^2 -v_t \right] + i (H_tF_t+H'_t) [\dot{\mathbf{Y}}]_t (-\xi v_t + im_t) \Big\} \hmu_t(\xi).
\end{multline}
Similarly, we have \begin{align}
    \label{eq:muMG}
\mu_t (\MG_t e^{i \xi \iota}) &= (F_t \xi - i H_t)(-\xi v_t + i m_t) \hmu_t(\xi) . 
\end{align}
Substitute \eqref{eq:muML} and \eqref{eq:muMG} to \eqref{eq:strat-zakai}, and compare the coefficient of $\xi^2$ term. Then we obtain 
\begin{align*}
dv_t &=  \left( \Sigma_t^2 + (2C_t-(F^2+F') \dot{[\BY]}_t ) v_t - H^2 \dot{[\BY]}_t v_t^2 \right) dt + 2 F_t v_t \circ d\BY_t \\
   &=  \left( \Sigma_t^2 + (2C_t + F^2 \dot{[\BY]}_t ) v_t - H^2 \dot{[\BY]}_t v_t^2 \right) dt + 2 F_t v_t d\BY_t.
\end{align*}
We leave it to the reader to spell out the same computation in a multivariate setting, such as to obtain the exact RDE stated for $v^\BY$ in Theorem \ref{thm:RKBf}.
   \begin{align*}
            & d  v_t^{\BY}   =    \Big( C(t;\BY) v^\BY_t + v_t^\BY C(t;\BY)^\top + \Sigma(t;\BY) \Sigma(t;\BY)^\top + \sum_{\kappa,\lambda=1}^{d_Y}  F_\kappa(t;\BY) v_t^{\BY} F_\lambda(t;\BY)^\top \dot{[\BY]}_t^{\kappa \lambda}  \\
       & \quad - v_t^{\BY} H(t;\BY)^\top \dot{[\BY]}_t H(t;\BY) v_t^{\BY} \Big) dr  + \big(F(t;\BY) v_t^{\BY} + v_t^{\BY} F(t;\BY)^\top  \big) d \BY_t, \qquad t \in [0,T] .
        \end{align*}

Since this (Riccati) RDE is well-posed for all $\BY \in \mathrm{s}\mathscr{C}^{0,\alpha,1}, Q:= \Sigma \Sigma^\top, R:= \dot{[\BY]} $, by Theorem \ref{thm:RRDEwell}, and $\mathrm{s}\mathscr{C}^{0,\alpha,1}$ is dense in $\mathscr{C}^{0,\alpha,1}$, conclude with an approximation argument, using stability of Riccati RDE together with Lemma \ref{lem:nga}.
The RDE for the mean $m_t^{\BY}$ is derived by similar arguments, left to the reader. (This situation is easier since the RDE for the mean $m^\BY$ is affine linear, which comes with a priori well-posedness results.) 
\end{proof}
 


{\em Consistency with stochastic theory.} 
We conclude this section by noting that randomization of the driving rough path in our rough Riccati equations, with the It\^o rough path lift of the observation process $Y=Y(\omega)$ as given in \eqref{eq:obs_linear}  yields a non-geometric (random) rough path $\BY(\omega)$ with $\dot{[\BY]}_t =  k(t;Y) k(t;Y)^\top$. 
Recalling consistency of (It\^o randomized) rough integration with It\^o integration, \cite[Prop. 5.1]{FH20}, it is then clear that the randomized solutions to the system of (Riccati) RDEs of Theorem \eqref{thm:RKBf} are indistinguishable from the ODE (Riccati) solutions previously stated in \eqref{eq:stoch-recatti}. Such randomizations are more involved in the general, infinite-dimensional, setting which is the topic of the next section.

\section{Connecting rough and stochastic filtering} \label{section:bridgingroughandstochasticfiltering}

\footnote{See also the end of Sections \ref{subsec:degen-case} and \ref{sec:RKBF} for consistency discussions in the context of degenerate observation noise and Kalman--Bucy, respectively.} %
We focus here on the (ND, i.e. $k$ invertible) case, with $\mathbb{P}$-dynamics given by  \eqref{eq:stoch-sig-2}, \eqref{eq:stoch-obs-2} and \eqref{eq:stoch-z-2}. 
The corresponding rough dynamics are given by \eqref{eq:roughSDE_filtering} and \eqref{eq:Girsanovexponential_filtering}.

In this section, we show that properly randomized, all rough objects become the ``right'' stochastic objects. Though not strictly necessary,\footnote{See \cite[Theorem 7]{CDFO13} or \cite[Theorem 1]{BC09} } 
we shall assume here that

\be\label{eq:probab}
( {\Omega}, \MF,  (\MF_t)_t, \PP )  = (\Omega', \MF', (\MF'_t)_t, \PP') \otimes (\Omega'', \MF'', (\MF''_t)_t,   \PP'')
\ee
where $(\Omega', \MF', (\MF'_t)_t  ,\PP')$ is a complete probability space supporting a $d_B$-dimensional Brownian motion $B$, similar for $(\Omega'', \MF'', (\MF''_t)_t  ,\PP'')$ with a $d_W$-dimensional Brownian motion $W$. In the following we take $\mathscr{C}_T= \SC^{0,\alpha,1}$ with $\alpha \in (1/3,1/2)$ and equip it with the Borel sigma algebra $\mathfrak{C}_T$. We consider the randomization of $\BY \in \mathscr{C}_T$ on $(\Omega'', \MF'', \PP'')$.

\begin{assumption}{CME}\label{assum:CME} Assume \cref{assum:E} holds and\\
(i). $( b, \sigma, (f,f'), (h,h') )$ as functions of $(t,x,\mathbf{Y})$ are 
$\MB_{[0,t]} \otimes \MB({\mathbb{R}^{d_X} }) \otimes 
\mathfrak{C}_T /{\MB}(\mathbb{R}^{d'})$-measurable and 
uniformly bounded, with appropriate choices of $d'$;\\
(ii). all coefficient functions above are causal in $\mathbf{Y}$, i.e. for any $\phi= b,\sigma, f,f',h,h',$ $\phi(t,x,\BY)= \phi(t,x, \BY^{t})$, where $\BY^{t}_s := \BY_{t\wedge s}$.\\


\end{assumption}


\begin{assumption}{OWP} (Observation, well-posed)
\label{assum:R} Assume strong well-posedness of 
\be\label{eq:stoch-signal}
Y_t(\ome'')= \int_0^t k(r,Y_r) dW_r(\ome'').
\ee
\end{assumption}
\noindent (Of course, the usual Lipschitz condition on $k$ as contained in \cref{assumptionCPE} is sufficient.)
\medskip



By assumption, $Y$ is then a local martingale, hence (\cite{friz2023rough} and references therein) admits a rough path lift, $\BY^{\text{It\^o}}(\ome'') :=(Y,\Y^{\text{It\^o}} )(\ome'')$ be the It\^o (martingale) rough path, with 
\be\label{eq:def-rand-Y}
\Y^{\text{It\^o}}_{s,t}(\ome''):= (\int_s^t \delta Y_{s,r} \otimes dY_r)(\ome''), \ \ \forall \  0 \le s \le t \le T.
\ee
Note that $\mathbf{Y}^{\text{It\^o}}_t$ is $\MF''_t$-measurable and $\MF^{\mathbf{Y}^{\text{It\^o}}}_t \subset \MF_t''$.






\subsection{Stochastic vs rough Kallianpur--Striebel formula}

Recall, by definition, for any bounded Borel measurable function $\varphi $ on $\R^{d_X},$
\be \label{eq:rough-filters}
 \mu^{\mathbf{Y}}_t (\varphi)
=
\mathbb{E}  [ \varphi(X^{\mathbf{Y}}_t ) 
   Z^{\mathbf{Y}}_t ], \quad   \varsigma^{\mathbf{Y}}_t (\varphi)  = \frac{  \mu^{\mathbf{Y}}_t (\varphi) }{  \mu^{\mathbf{Y}}_t (1) },
\ee
where $(X^{\BY}, Z^{\BY})$ is the solution of rough SDE \eqref{eq:roughSDE_filtering}, \eqref{eq:Girsanovexponential_filtering}, with $(\Omega, \MF, \PP)$ there replaced by $(\Omega', \MF', \PP')$.

\begin{lem}\label{lem:jt-meausrable} Under \cref{assum:CME}, \\
(i) there is a $\MF'_t \otimes \mathfrak{C}_t$-progressively measurable version of $(X^{\mathbf{Y}}, Z^{\mathbf{Y}})$;\\
(ii) for any $t \ge 0,$ we also get $\mathfrak{C}_T \otimes \MB_{[0,t]}$-measurable mapping  $(\mathbf{Y},r) \mapsto \mu^{\mathbf{Y}}_r \in \mathcal{M}_F$. Similar for $\varsigma^{\mathbf{Y}}_t$ as a functional with values in $\mathcal P$.

\end{lem}

\begin{proof}
To prove (i), one only needs to check conditions from \cite[Theorem 3.4]{FLZ25} are satisfied under \cref{assum:CME}. (ii) follows by (i) and integrability of $Z^{\BY},$ which is proved in \cref{thm:wellposedness_roughfiltering}. 
\end{proof}
By measurability of $\ome'' \mapsto \BY^{\text{It\^o}}(\ome'')$
together with the previous lemma it is clear that 
\be\label{eq:randomised-sys+weight}
(\bar{X}, \bar{Z})_t (\ome',\ome''):= (X^{\BY}, Z^{\BY} )_t (\ome')|_{\BY=\BY^{\text{It\^o}}(\ome'')}
\ee
defines measurable processes. To connect with the original stochastic filtering problem, consider the It\^o SDE on $(\Omega, \MF, \PP)$,
\be\label{eq:rand-Ito-sde}
\begin{split}
 & d  {X}_t = b_{t, \omega} ( {X}_t) d t +  \sigma_{t,\ome}(X_t ) dB_t + f_{t,\ome} ({X_t})dY_t, \ X_0= X_0 (\ome');\\
 & dZ_t =Z_t h_{t,\ome} (X_t) d Y_t, \ Z_0=1,
\end{split}
\ee
with $\phi_{t, \omega} (x) = \phi (t, x ; \BY^{\tmop{It\hat{o}}}_{. \wedge t}
  (\omega''))$, $\phi=b,\sigma, f, h$. 
  Under \cref{assum:CME}, standard theory (e.g. \cite{Protter}) ensures a unique strong solution, denoted by $(X,Z)$.

\begin{theorem}\label{thm:mu_as_process}
Suppose \cref{assum:CME} and \cref{assum:R} hold. Then there exists a continuous modification of $(\bar X, \bar Z)$, denoted also by $(\bar X, \bar Z)$ such that
$(X,Z)$ and $(\bar X, \bar Z) $ are indistinguishable on $(\Omega, \MF, \PP).$ In particular, $  \bar Z_t$ is integrable for any $t\in [0,T].$ Moreover, let $\BY(\ome):= \BY^{\text{It\^o}}(\ome'')$ be given by \eqref{eq:def-rand-Y}, and we have

\begin{itemize}

\item[(i.1)] $\bar{\mu}_t (\omega) \assign \mu^{\mathbf{Y}}_t |
  \nobracket_{\mathbf{Y} = \mathbf{Y} (\omega'') }$ defines a
  $\mathcal{M}_F$-valued r.v.\ such that for any bounded measurable function $\varphi,$ a.s.
  \begin{equation} \label{e:42i1}
  \mathbb{E} [\varphi (\bar{X}_t) \bar{Z}_t |\mathcal{F}^{Y}_t \nobracket] = \nobracket  
     \mu^{\mathbf{Y}}_t (\varphi) |_{\mathbf{Y} =
     \mathbf{Y}  (\omega)} =     \bar{\mu}_t (\omega) (\varphi); 
     \end{equation} 

\item[(i.2)] the measure valued process $\{ \bar{\mu}_t : t \ge 0 \}$ is weakly continuous in $\mathcal{M}_F$ a.s., in the sense that for almost surely every $\ome,$ for any $\varphi \in \mathcal{C}^0_b(\R^{d_X};\R)$, the map $[0,T] \ni t \mapsto \bar{\mu}_t(\varphi)(\ome)  \in \R$ is continuous;
 
 \item[(ii.1)] Similarly, 
 $\bar{\varsigma}_t (\omega) \assign \varsigma^{\mathbf{Y}}_t |
  \nobracket_{\mathbf{Y} = \mathbf{Y}  (\omega)}$ defines a
  $\mathcal{P}$--valued r.v. such that a.s.
 \begin{equation} \label{e:42ii}
     \varsigma^{\mathbf{Y}}_t (\varphi)
     |_{\mathbf{Y} =
     \mathbf{Y} (\omega)} 
  =  
     \bar \varsigma_t (\omega) (\varphi);
\end{equation}

\item[(ii.2)] the measure valued process $\{ \bar{\varsigma}_t : t \ge 0 \}$ is weakly continuous in $\mathcal{M}_F$.

\end{itemize}

\end{theorem}

\begin{proof} 

According to \cite[Theorem 3.9]{FLZ25}, we see that there exists a continuous modification of $(\bar X, \bar Z)$, such that
$(X,Z)$ and $(\bar X, \bar Z) $ are indistinguishable. Moreover, by \cref{assum:CME} and \cref{assum:R}, we see that 
$
 \E [\sup_{t\in [0,T]} |Z_t|^p]< \infty, \ \ \forall \ p\ge 1,
$
which implies the integrability of $  \bar Z.$

(i.1) Clearly $\BY \mapsto \mu_t^{\BY} \in \mathcal{M}_F
(\mathbb{R}^d)$ is a \tmtextit{kernel} (terminology from e.g. \cite[Chapter 3]{Kal97}) from rough path space $\mathscr{C}$ to $\mathbb{R}^d$, both equipped with their Borel sets,
in the sense that $\BY \mapsto \mu_t^{\BY} (B)$ is
measurable, for any fixed Borel set $B$ in $\mathbb{R}^d$. By composition we
see that the randomization $\ome \mapsto \mu_t^{\tmmathbf{Y}} |_{\tmmathbf{Y} =
\tmmathbf{Y}  (\omega)} \nobracket \backassign \bar{\mu}_t (\ome)$ also gives
rise to a kernel, now from $(\Omega, \mathcal{F})$ to $\mathbb{R}^d$. Then according to \cite[Proposition 5.7]{MR3752669}, $\bar{\mu}_t (\omega)$ is a $\mathcal{M}_F$-valued random variable, when $\mathcal{M}_F$ is equipped with the weak topology.

For identities \eqref{e:42i1}, we immediately see the second one by definition. For the first one, by the definition \eqref{eq:randomised-sys+weight}, we have 
$\text{Law}((\bar X_t, \bar Z_t)|\MF^Y_t)= \text{Law}((X^{\BY}_t, Z^{\BY}_t))|_{\BY=\BY(\ome) }$, which means for any bounded measurable function $F: \R^{d_X} \times \R \rightarrow \R$, 
$$
\E[F(\bar X_t, \bar Z_t)|\MF^{Y}_t ](\ome) = \E[F(X^{\BY}_t, {Z}^{\BY}_t) ]|_{\BY=\BY(\ome) }.
$$
Then by a standard cutoff argument and the dominated convergence theorem, for any measurable function $F$ with $\E[|F(\bar X_t, \bar Z_t)| ]<\infty,$ the above identity also holds, which implies the first identity of \eqref{e:42i1}. 

(i.2) For any $\BY=\BY(\ome'')$, according to \cref{thm:existence} in the sense of \cref{def:solutionroughZakai}, we have that $\mu^{\BY}$ is a solution to \eqref{eq:roughZakai_CITE}, which implies the continuity we claim.

(ii) Similarly, just noting that 
$\tmmathbf{Y} \mapsto \varsigma_t^{\tmmathbf{Y}} \in \mathcal{M}_F
(\mathbb{R}^d)$ is a \tmtextit{probability kernel}, which as before leads to the statement that its randomization is a 
$\mathcal{P}$--valued random variable.
\end{proof}

To retrieve the original probability measure $\PP^o$ of the stochastic optimal filtering, let 
\be\label{def:ori-Po}
d \PP^o :=   Z_T   d\PP,
\ee
where $Z$ is given by It\^o SDE \eqref{eq:rand-Ito-sde}. 
Then we have the following rough stochastic Kallianpur--Striebel formula.

\begin{thm} \label{thm:RoughFilterFinal}
Suppose \cref{assum:CME} and \cref{assum:R} hold. Then we have for any $\varphi \in \mathcal{C}^0_b(\R^{d_X};\R)$,   
$$
\E^{\PP^o} [\varphi(  X_t)| \MF^{Y}_t] =  
      \varsigma^{\mathbf{Y}}_t (\varphi)
     |_{\mathbf{Y} =
     \mathbf{Y} (\omega)} = \frac{  \mu^{\mathbf{Y}}_t (\varphi) }{  \mu^{\mathbf{Y}}_t ( 1 )} \Big|_{\mathbf{Y}=
     \mathbf{Y} (\omega)}, \ \ \PP(\PP^o)\text{-a.s.}
$$

\end{thm}


\begin{proof}
Note that $\E[  Z_t]\equiv 1,$ $\forall \ t\in[0,T].$ We see that $  Z$ is a martingale on $(\Omega, \MF, \PP)$, and thus $\PP^o$ is a probability measure. Then by Bayes' formula and \cref{thm:mu_as_process},
$$
\E^{\PP^o}[\varphi(  X_t) | \MF^{Y}_t]= \frac{\E [\varphi(\bar X_t) \bar Z_t | \MF^{Y}_t]}{\E [\bar Z_t | \MF^{Y}_t]}  = \frac{  \mu^{\mathbf{Y}}_t (\varphi)}{  \mu^{\mathbf{Y}}_t ( 1)  } \Big|_{\mathbf{Y}=
     \mathbf{Y} (\omega)}. 
$$

\end{proof}




  

\begin{remark}

In the particular case $b_{t, \omega} (x) = b (t, x ; Y_t (\omega))$,
  where $Y_t (\omega) = \pi_1 \tmmathbf{Y}^{\tmop{It\hat{o}}}_t (\omega)$, we are in
  the setting of \cite{CP24} and see that our $(\bar{X}, \bar{Z})$ satisfies the SDE
  dynamics for $X, Z$ (see equations \eqref{eq:stoch-sig-2} and \eqref{eq:stoch-z-2}). Since \cref{assumptionCPE} guarantees well-posedness
  of these dynamics, we conclude that our randomized objects are
  indistinguishable from the stochastic objects considered in \cite{CP24}. One can then
  safely remove all bars in the the abvove robust representation formulae.     
\end{remark}


\begin{remark} (i) Following Lemma 3.7 and Remark 3.9 in \cite{CP24}, also Chapter 2 in \cite{BC09}, it is a non-trivial technical task to obtain a $\mathcal{M}_F$-valued process $\mu_t$ (and even a $\mathcal{M}_F$-valued random variable $\mu_t$ for fixed $t$)  such that
\begin{equation}
   \label{equ:CP_Rmk39} 
  \mu_t (\omega) ( \varphi) :=
  \mathbb{E} [\varphi (X_t) Z_t |\mathcal{F}^Y_t \nobracket] = \nobracket 
     \mu^{\mathbf{Y}}_t ( \varphi)  |_{\mathbf{Y} =
     \mathbf{Y}^{\tmop{It\hat{o}}} (\omega).}
 \end{equation}
Theorem \ref{thm:mu_as_process} provides such a process (and implicitly a weakly  continuous version of any $\mu$ for which whenever \eqref{equ:CP_Rmk39} holds).
 
(ii)  For fixed $t$, the random measure $\varsigma_t$ is an explicit construction of the regular conditional distribution of $X_t$, under $\mathbb{P}^o$, given $\mathcal{F}^Y_t$. The understanding of  $\varsigma_t$
as a process, achieved en passant in Theorem \ref{thm:mu_as_process},
is an important (if technical) aspect in filtering theory, cf.  Chapter 2 in \cite{BC09}.
\end{remark}

\subsection{Stochastic vs rough filtering equations}
\label{sec:SvsRFW}
In this subsection, we talk about the relation between stochastic Zakai equation \eqref{equ:ZakNonDeg} with our rough Zakai equation \eqref{eq:roughZakai_CITE}. Here we only talk about the Zakai equations since for the Kushner--Stratonovich equations, the argument is the same.

In the following, we assume all the coefficients $\phi(t,x,\BY)= \phi(t,x,Y_t),$ for $\phi=b,\sigma, f, h.$ Let $\mu$ be given by \eqref{equ:CP_Rmk39}

\begin{theorem} 
Let \cref{assum:CME} and \cref{assum:R} hold. Let $\mu^{\BY}$is be a solution to the rough Zakai equation in the sense of Definition \ref{def:solutionroughZakai}.
Then its randomization $\bar{\mu}$ is a solution to the classical Zakai equation, in the sense of \eqref{equ:ZakNonDeg} when the test function is choosen to be in $\mathcal{C}^{3}_b$. 
\end{theorem}
\begin{remark} (i) As noted in  \cref{rem:CP24_23}, one can switch from $\mathcal{C}^{3}_b$ to ${C}^2_b$, the classical space of bounded, twice continuously differentiable functions with bounded derivatives, and allow for time-dependent test functions. (ii) Strictly speaking, in \eqref{equ:ZakNonDeg} coefficients are of the form $\phi (t,x,Y_t)$ which amounts to a very explicit form of causality; our setup clearly allows to deal with more general form of the classical Zakai equation but to stay aligned with \cite{CP24} will not not discuss this further here.
\end{remark}

\begin{proof} 
For any test function $\varphi \in \mathcal{C}^3_b (\mathbb{R}^{d_X}
  ; \mathbb{R})$,
by Proposition \ref{prop:roughZakai_integralform} we have the integral equation

\[ \mu^\BY_t (\varphi) = \mu^\BY_0 (\varphi) + \int_0^t   \mu^\BY_r
      (A^\BY_r \varphi) \, dr + \int_0^t \mu^\BY_r (\Gamma_r^\BY \varphi) \, d\BY_r,
\]
where $(\mu^\BY  (\Gamma^\BY \varphi), \ \mu(\Gamma^{\BY} \Gamma^{\BY}\varphi+ (\Gamma^{\BY})' \varphi)) \in \SD^{2 \alpha}_Y$ is a deterministic $Y$-controlled rough path. Note that for $\phi(t,\BY)  = \mu^\BY_t
      (A^\BY_t \varphi),\  \mu^\BY_t (\Gamma_t^\BY \varphi),\  \mu_t(\Gamma^{\BY}_t \Gamma^{\BY}_t \varphi+ (\Gamma^{\BY}_t)' \varphi)$, $\phi:[0,T] \times \SC_T \rightarrow \R$ is measurable and causal. Besides, $\mu^{\BY}_0(\varphi)= \E[\varphi(X_0)]= \E'[\varphi(X_0(\ome'))]$ for any $\BY \in \SC_T.$ Then according to \cite[Theorem 3.9]{FLZ25}, we have 
\be\label{eq:eqiv}
\mu^{\BY}_t(\varphi)|_{\BY=\BY(\ome'')} = \E[\varphi(X_0)] + \int_0^t \overline{\mu_t^{\BY}
      (A_t^{\BY} \varphi)} \ dr + \int_0^t \overline{\mu_t^{\BY} (\Gamma_t^{\BY} \varphi)} \ d Y_r, 
\ee
where $\overline{ \phi_t^{\BY} } := \phi(t,{\BY})|_{\BY=\BY(\ome'')}= \phi(t,Y)|_{Y=Y(\ome'')}$ by \cref{assum:CME} (i). Note that for any measurable $F:[0,T] \times \R^{d_X} \times \SC_T \rightarrow \R,$ by \cref{thm:mu_as_process} (i.1), 
$$
\overline{\mu_t^{\BY}(F(t,\cdot, \BY)) }= \overline{ \E'[F(t,X^{\BY}_t, \BY) Z_t^{\BY}] }  = \E'[F(t,X^{\BY}_t, \BY) Z_t^{\BY} ]|_{\BY=\BY(\ome'')}=\langle \bar\mu_t(\ome'') , F(t,\cdot, \BY(\ome'')) \rangle.
$$
Then the right hand side of \eqref{eq:eqiv} is exactly the right hand side of \eqref{equ:ZakNonDeg}, which implies our conclusion.



  
\end{proof}

\begin{remark}
    One similarly shows that the randomized solution to the rough Kushner--Stratonovich equation \eqref{eq:roughKushnerStratonovich_CITE} solves the stochastic Kushner--Stratonovich equation \eqref{eq:KushnerStratonovich_CP24-nd}.
\end{remark}

\appendix

\section{Review on Stochastic filtering of diffusion processes} \label{sec:review}

\subsection{Stochastic dynamics of signal and observation }

Let $(\Omega, \mathcal{F}, (\mathcal{F}_t)_{t \in [0, T]}, \mathbb{P}^o)$ be a
filtered probability space, equipped with complete and independent
sub-filtrations $(\mathcal{F}^B_t), (\mathcal{F}^{B^{\perp}}_t)$, which
supports independent Brownian motions $B$ and $B^{\perp}$, of respective
dimension $d_B$ and $d_{W}$. 
One is given, under the {\em original} measure $\mathbb{P}^o$, joint dynamics of a signal process $X$ and observation process
$Y$, of respective dimensions $d_X$ and $d_Y$, 
\begin{eqnarray}\label{eq:stoch-sig}
 && d X_t (\omega)  =  \bar{b} (t, X_t, Y_t) d t + \sigma (t, X_t, Y_t) d B_t
  (\omega)  + \bar{f} (t, X_t, Y_t) d B^{\perp}_t (\omega)\\ \label{eq:stoch-obs}
 && d Y_t (\omega)  =  \bar{h} (t, X_t, Y_t) dt + k (t, Y_t) d B^{\perp}_t
  (\omega)
\end{eqnarray}
Conditions on the coefficients fields (of Lipschitz type) are given later and
will in particular guarantee strong well-posedness of this SDE system. (We will soon transition to a system in which all bar-decorated coefficients above are replaced by their unbarred counterparts, thereby minimizing the use of bar notation.)





\begin{definition}[Filtering Problem]
  The filtering problem consists in determining the conditional distribution
  $\varsigma_t$ of the signal $X$ at time $t$ given the information
  accumulated from observing $Y$ in the interval $[0, t]$; that is, for any
  $\varphi \in \mathcal{B}_b (\mathbb{R}^{d_X})$,
  \begin{equation} \label{defofvarsigma}
      \varsigma_t (\varphi) =\mathbb{E}^o [\varphi (X_t) \mid \mathcal{F}^Y_t],
  \end{equation} 
  where $(\mathcal{F}^Y_t)$ is the 
 observation filtration.
\end{definition}

Following \cite{CP24}, we assume
 $\bar{h} (t, x, y) = h_1 (t, y) + k (t, y) h_2 (t, x, y) $
which allows to absorb $h_2$ in 
\be \label{eq:W}
W_t : = \int_0^t h_2 (r, X_r, Y_r) dr + B^{\perp}_t . 
\ee
We make the standing assumption 
\begin{assumption}{CP-M} \label{assumptionCPM}
\[ \mathbb{E} \tilde Z_t = 1, \ \  t\in [0,T], \quad   \text{where } \tilde Z_t \assign \exp \left( - \int_0^t h_2 (r, X_r, Y_r)^{\top} \hspace{0.17em}
   \mathrm{d} B^{\perp}_r - \frac{1}{2}  \int_0^t \lvert h_2 (r, X_r, Y_r) \rvert^2 d
   r \right) \hspace{0.17em} . \]
\end{assumption}
By a Girsanov change of measure, $W$ then becomes a new Brownian motion, independent of $B$. 
%
Specify the probability $\mathbb{P}(\sim \mathbb{P}^o)$ via its Radon-Nikodym
derivative $Z_T (> 0)$, i.e.
\be \label{def:PP}
\frac{\mathrm{d} \mathbb{P}}{\mathrm{d} \mathbb{P}^o} \Big|_{\mathcal{F}_T} =
   \tilde{Z}_T, 
\ee
the inverse of which is given by
\begin{eqnarray}\label{eq:stoch-I}
  Z_T := \exp (I_T)  : =  \exp \left( \int_0^T h_2 (r, X_r, Y_r)^{\top} \hspace{0.17em}
  \mathrm{d} W_r - \frac{1}{2}  \int_0^T \lvert h_2 (r, X_r, Y_r) \rvert^2 d r
  \right) .
\end{eqnarray}

To proceed, we want to express $d B^{\perp}_t = d W_t - h_2 dt$ in \eqref{eq:stoch-sig}, \eqref{eq:stoch-obs} 
in terms of $d Y $=$h_1 d s + k dW$,
to the extent possible. In general, we can write\footnote{The good choice of
$k^+$ turns out to be the Moore-Penrose inverse of $k$, as understood in \cite{CP24}, although this plays no
role at this moment.}
\[ {\bar{f} \hspace{0.17em} dW } = \bar{f} \hspace{0.17em}
   k^+ (d Y - h_1 d s) + \bar{f} \hspace{0.17em} (I - k^+ k) d W = : f_1 (d Y
   - h_1 d s) + f_2 dW. \]
Similarly, rewrite the stochastic integral in the Girsanov factor $Z = \exp (I)$,
using
\[ h_2^{\top}  d {W} = : h_3 (d Y - h_1 d s) + h_4 d {W} . \]
This results in the system, with $b := \bar{b} - \bar{f} h_2$,
\begin{eqnarray}
\label{eq:stoch-sig-1}
&& dX_t (\omega) = b(t, X_t, Y_t) \, dt + \sigma(t, X_t, Y_t) \, dB_t(\omega) 
+ f_1(t, X_t, Y_t) \left( dY_t(\omega) - h_1(s, Y_s) \, ds \right) \nonumber \\
&& \quad\quad\quad\quad  + f_2(t, X_t, Y_t) \, dW_t(\omega) \\
\label{eq:stoch-obs-1}
&& dY_t(\omega) = h_1(t, Y_t) \, dt + k(t, Y_t) \, dW_t(\omega) \\
\label{eq:stoch-z-1}
&& dZ_t(\omega) = Z_t h_3(t, X_t, Y_t) \left( dY_t(\omega) - h_1(s, Y_s) \, ds \right) 
+ Z_t h_4(t, X_t, Y_t) \, dW_t(\omega).
\end{eqnarray}
In the {\em non-degenerate case}, when $k$ is invertible (take $d_Y = d_W$ and $k^+ =k^{-1}$) we can assume without loss of generality that $h_1=0$ (else absorbe in $h_2$). We also have $f_2 = 0$ and $h_4
= 0 $ and simplify notation accordingly by setting
$$f \assign f_1 = \bar{f} k^{-1}, h := h_3 := h_2^{\top} k^{-1},$$ 
which leads us, in the non-degenerate case, to following simpler form 
\begin{eqnarray}
\label{eq:stoch-sig-2}
&& dX_t(\omega) = b(t, X_t, Y_t) \, dt + \sigma(t, X_t, Y_t) \, dB_t(\omega) 
+ f(t, X_t, Y_t) \, dY_t(\omega) \\
\label{eq:stoch-obs-2}
&& dY_t(\omega) = k(t, Y_t) \, dW_t(\omega) \\
\label{eq:stoch-z-2}
&& dZ_t(\omega) = Z_t \, h(t, X_t, Y_t) \, dY_t(\omega).
\end{eqnarray}

We note that $(W,Y)$ and $B$ are $\mathbb{P}$-independent, 
$Y$ is a semi- (resp. local-)martingale with covariation dynamics
\[ d \langle Y, Y \rangle_t = k (t, Y_t) k (t, Y_t)^{\top} d t . \]

\subsection{Kallianpur--Striebel formula and Clark's robustness result}\label{sec:stochKSformula}

\begin{theorem}[Kallianpur--Striebel; Thm 3.5 in \cite{CP24}] \label{thm:KS}
Granted \cref{assumptionCPM}, 
and assume dynamics \eqref{eq:stoch-sig-1}, \eqref{eq:stoch-obs-1} and \eqref{eq:stoch-z-1} are  strongly well-posed under $\mathbb{P}$. Then, for every $\varphi \in
  \mathcal{B}_b (\mathbb{R}^{d_X})$ and $t \in [0, T]$,
  \begin{equation} \label{eq:KallianpurStriebelformula}
      \varsigma_t (\varphi) = \frac{\mathbb{E} [Z_t \varphi (X_t) \mid
     \mathcal{F}^Y_t]}{\mathbb{E} [Z_t \mid \mathcal{F}^Y_t] } \backassign
     \frac{\mu_t (\varphi)}{\mu_t (1)}  \quad \mathbb{P} (\mathbb{P}^o)
     \text{-a.s.} 
  \end{equation} 
  where $(\mathcal{F}^Y_t)$ is the observation filtration\footnote{Following Remark 3.9 in \cite{CP24} with completion, one has a cadlag version of $\vars(\varphi)$. Our rough path approach, if applicable, will effectively bypass such considerations.}
  and $\mu=\{\mu_t, t\ge 0\}$ is the so-called {\em unnormalised filter}.
\end{theorem}


It is a very natural question if and when such a filter depends continuously
on the observation. In the non-correlated case, when the statistics of the signal are not affected by the observation noise, the following result is classical.


\begin{theorem}[Clark robustness; \cite{CC05, BC09,davis2011pathwise, Yau24}]
  Granted \cref{assumptionCPM} 
  and strongly well-posed dynamics \begin{eqnarray}\label{eq:stoch-sig2}
 && d X_t (\omega)  =  {b} (t, X_t) d t + \sigma (t, X_t) d B_t
  (\omega)  \\ 
 && d Y_t (\omega)  =  h (t, X_t) dt + 
 d B^{\perp}_t
  (\omega),
\end{eqnarray}
  so that the signal process $X$ is independent of the observation noise. 
  Assume $\varphi$ is bounded continuous. Then it holds that $\mu_t
  (\varphi)$, and by consequence the filter $\varsigma_t (\varphi)$, admits a
  version that is a continuous function of the observation path $Y (\cdot)$.
  That is, there is continuous $\Phi$ on pathspace s.t.
  \[ {\mathbb{E}^o}   [\varphi (X _t) \tmcolor{red}{\tmcolor{red}{}} |
     \mathcal{F}^Y_t] \nobracket = \Phi (Y). \]
\end{theorem}

This type of robust representation is also possible when the signal and the
observation noise are correlated, provided the observation process is scalar.
However, in the general case of correlated noise and multidimensional
observation, $d_Y > 1$, such is impossible, as shown by explicit example in \cite{CDFO13}. By
using the theory of rough paths the same paper offered a first solution to
this deficiency, showing that under suitable conditions on the coefficients
fields it is possible to find a version of the filter that is a continuous
function of the observation rough path $\mathbf{Y} (\cdot)$. Our work also offers some improvements with regard to \cite{CDFO13} in removing excessive regularity demands, and dealing with the general dynamics seen in  \eqref{eq:stoch-sig}, \eqref{eq:stoch-obs} 
at least in the non-degenerate case.

\subsection{The (stochastic) filtering equation in spaces of measures} \label{sec:23}

In case $k=0$, 
we see that $X$, conditionally on $Y$, is Markov with ($\mathcal{F}^Y$-adapted random) generator $A_t = A_{t,\omega}$, so that the conditional law $\varsigma_t = \mu_t$ satisfies the random forward equation\footnote{This is consistent with the Kallianpur--Striebel formula since here indeed a.s. $\mu_t (1) \equiv 1$.}
\begin{equation}
  \mu_t (\varphi) = \mu_0 (\varphi) + \int_0^t \mu_r  (A_{r} \varphi) dr.
\end{equation}
For more general (and typically invertible) $k$, recalling its Feynman-Kac representation, $\mu_t$ is only a finite random measure. Following the general discussion \cite{CP24} its evolution described by the {\em Zakai equation}. That is, for any $\varphi \in C^2_b(\R^{d_X};\R)$, \footnote{Here and in the entire \cref{sec:review}, $C_b^2(\R^{d_X};\R)$ denotes the classical space of bounded and twice continuously differentiable functions with bounded derivatives. Note that in the rest of the paper we denote by $\mathcal{C}^2_b(\R^{d_X};\R)$ the space of bounded and continuously differentiable functions with bounded, Lipschitz continuous first derivative.  } 
\begin{align}
  \mu_t (\varphi)  = & \ \mu_0 (\varphi) + \int_0^t \mu_r  (A_r \varphi) dr \nonumber \\
  &   + \int_0^t \mu_r  (\nabla \varphi \bar{f} (r, \cdot, Y_r) + \varphi
  h_2^\top (r, \cdot, Y_r)) k^{+} (r, Y_r)  (dY_r - h_1 (r, Y_r) d r) \label{equ:ZakDeg} \\
   \underset{\tmop{ND}}{=} & \ \mu_0 (\varphi) + \int_0^t \mu_r  (A_r \varphi) dr \nonumber \\
  &   + \int_0^t \mu_r  (\nabla \varphi f (r, \cdot, Y_r) + \varphi h (r,
  \cdot, Y_r))  
  d Y_r,
  \quad \mu_0 = \tmop{Law} (X_0) ,\label{equ:ZakNonDeg}
\end{align}
which we have written for the reader's convenience in both the degenerate and non-degenerate (ND) case.
 In the above $k^+$ denotes the Moore-Penrose inverse of $k$, the stochastic differential $dY_t$ is understood in It\^o sense and the random operator $A_t$ is given by 
\begin{equation} \label{eq:operatorA}
  A_t \varphi (x) = \sum_{i = 1}^{d_X} \partial_i \varphi (x) \bar{b}^i (t, x,
  Y_t) + \frac{1}{2}  \sum_{i, j = 1}^{d_X} \partial_{i j}^2 \varphi (x)
  \bar{a}^{ij} (t, x, Y_t),
\end{equation}
with
\begin{align*}
  \bar{a}^{ij} (t, x, Y_t) & =  \sum_{\theta = 1}^{d_B} \sigma^i_\theta \sigma_\theta^j  (t,
  x, Y_t) + \sum_{\kappa = 1}^{d_Y} \bar{f}^i_{\kappa}  \bar{f}^j_{\kappa} 
  (t, x, Y_t) \\
   & \underset{\tmop{ND}}{=}  \sum_{\theta = 1}^{d_B} \sigma^i_{\theta} \sigma_{\theta}^j  (t, x,
  Y_t) + \sum_{\lambda, \eta = 1}^{d_Y} f^i_{\lambda}  (t, x, Y_t)  \frac{d
  \langle Y^{\lambda}, Y^{\eta} \rangle_t}{d t} f^j_{\eta}  (t, x, Y_t)
\end{align*} and
\begin{align*}
  \bar{b}^i (t, x, Y_t) & =  b^i (t, x, Y_t) + \sum_{\kappa  =
  1}^{d_Y} \bar f^i_{\kappa} (t, x, Y_t)
  h_2^{\kappa} (t, x, Y_t) \\
  & \underset{\tmop{ND}}{=}  b^i (t, x, Y_t) + \sum_{\kappa, \eta = 1}^{d_Y} f^i_{\kappa} (t, x,
  Y_t) \frac{d \langle Y^{\kappa}, Y^{\eta} \rangle_t}{d t} h_{\eta} (t, x,
  Y_t) .
\end{align*}
It will be convenient to define the first order differential operator $\Gamma_t
= ((\Gamma_t)_1, \ldots, (\Gamma_t)_{d_Y})$ as
\be\label{eq:stochGamma}
(\Gamma_t )_{\kappa} \varphi (x) : = \sum_{i = 1}^{d_X} \partial_i \varphi
   (x) f^i_{\kappa}  (t, x, Y_t) + \varphi (x) h _{\kappa} (t, x, Y_t), \qquad
   \kappa = 1, ..., d_Y . 
\ee

\begin{assumption}{CP-E} \label{assumptionCPE}
    Given are measurable functions \begin{align*}
        b &: [0,+\infty) \times \R^{d_X} \times \R^{d_Y}  \to \R^{d_X} \\
        \sigma &: [0,+\infty) \times \R^{d_X} \times \R^{d_Y}   \to \R^{d_X \times d_B} \equiv \mathrm{Lin}(\R^{d_B},\R^{d_X}) \\
        \bar f &: [0,+\infty) \times \R^{d_X}  \to \R^{d_X \times d_W} \equiv \mathrm{Lin}(\R^{d_W},\R^{d_X}) \\
        \bar h &: [0,+\infty) \times \R^{d_X}  \to \R^{d_Y} \\
        h_1 &: [0,+\infty) \to \R^{d_Y} \\
        h_2 &: [0,+\infty) \times \R^{d_X} \times \R^{d_Y}  \to \R^{d_W} \\
        k  &: [0,+\infty) \times \R^{d_Y}  \to  \R^{d_Y \times d_W}  
    \end{align*}
    with the following properties: 
    \begin{itemize}
        \item $ \bar{b}, \sigma, \bar{f}, \bar{h}, h_1, k $ are locally Lipschitz in the \( (x, y) \)-variables, uniformly in the $t$-variable; 
        That is, for any $R>0$, there exists $C_R >0$ such that $$ |\bar{b}(t, x_1, y_1) - \bar{b}(t, x_2, y_2) | \le C_R ( | x_1 - x_2 | + | y_1 - y_2 | ) $$ for all $t \in [0,+\infty)$, $|x_1|,|x_2|,|y_1|,|y_2| \le R$, and similarly for $\sigma, \bar{f}, \bar{h}, h_1, k$ .
        \item The functions  $\bar{b}, \sigma, \bar{f}, \bar{h}, h_1, h_2, k$  satisfy a linear growth condition in $(x, y)$. Namely, there exists $K > 0$  such that $$ |\bar{b}(t, x, y) | \le K (1 + | x| + | y| ) $$
        for all $(t, x, y)$, and similarly for \( \sigma, \bar{f}, \bar{h}, h_1, h_2, k \).
    \end{itemize}
    The initial conditions \( X_0 \), \( Y_0 \) have finite second moments, that is $\E(|X_0|^2 + |Y_0|^2) < +\infty$. 
\end{assumption}

\begin{assumption}{CP-U} \label{assumptionCPU}
    For any arbitrary $T>0$, the functions $\bar b, \sigma, \bar f, h_2$ are bounded on $[0,T] \times \R^{d_X} \times \R^{d_Y}$ and the functions $h_1$ and $k$ are bounded on $[0,T] \times \R^{d_Y}$. 
    Moreover, for some integer $$n > \frac{d_X}{2}+2 , $$ all the partial derivatives of the functions $\bar b, \sigma, \bar f, \bar h$ in the $x$-variable with multi-index $\alpha$, such that $|\alpha| \le n$, are bounded on $[0,T] \times \R^{d_X} \times \R^{d_Y}$.
\end{assumption}

\begin{theorem}
[{cf. \cite[Theorem 3.11 and Theorem 5.5]{CP24} }] \label{thm:CP24ZakaiWP} 
  Under \cref{assumptionCPE}, the measure-valued process $\mu_t$ defined in \eqref{eq:KallianpurStriebelformula} solves the Zakai equation \eqref{equ:ZakDeg} for any $\varphi \in C^2_b(\R^{d_X};\R)$. 
  Moreover, under \cref{assumptionCPE} and \cref{assumptionCPU}, it is the unique $(\mathcal{F}^Y_t)_t$-adapted solution to \eqref{equ:ZakDeg} satisfying \begin{equation*}
      \E \Big( \sup_{t \in [0,T]} \mu_t(1)^2 \Big) < +\infty \qquad \text{for any $T>0$}. 
  \end{equation*}
\end{theorem} 

\begin{theorem}[{cf. \cite[Theorem 3.11 and Theorem 3.17]{CP24}}]
  Under \cref{assumptionCPE}, the measure-valued process $\varsigma_t$ defined as in \eqref{defofvarsigma} satisfies the following \textup{Kushner--Stratonovich equation}, for all $\varphi \in C_b^2  (\mathbb{R}^{d_X};\R) :$ 
  \begin{align}
      \varsigma_t (\varphi )   = & \  \varsigma_0 (\varphi ) + \int_0^t \varsigma_r
    (A_r \varphi) \hspace{0.17em} \mathrm{d} r \notag  \\
      &   + \int_0^t (\varsigma_r (\nabla \varphi  \bar{f} + \varphi
    h_2^{\top}) - \varsigma_r (\varphi) \varsigma_r (h_2^{\top})) k^{+} (r,
    Y_r)  \left( \mathrm{d} Y_r - \varsigma_r (\bar{h}) \, dr \right) \label{eq:KushnerStratonovich_CP24} \\ 
     \underset{\tmop{ND}}{=} & \  \varsigma_0 (\varphi) + \int_0^t \vars_r  (A_r \varphi) dr \notag \\
        & + \int_0^t \left( \varsigma _r  (\nabla \varphi f  +
    \varphi h ) - \varsigma_r(\varphi)\varsigma_r(h) \right) \left( \mathrm{d} Y_r - \frac{d\langle Y, Y \rangle_r}{dr} \varsigma_r (h^T) \, dr \right) .  \label{eq:KushnerStratonovich_CP24-nd}
  \end{align}
  Moreover, the uniqueness of the solution of equation \eqref{eq:KushnerStratonovich_CP24} is equivalent to that of equation \eqref{equ:ZakDeg}. 
\end{theorem}

\begin{remark} \label{rem:CP24_23}
    In \cite{CP24}, the authors consider time-dependent test functions in the class $C^{1,2}_b([0,T] \times \R^{d_X};\R)$, in which case, denoting by $\varphi_t = \varphi(t,\cdot)$, the Zakai equation is given as 
    \begin{align}
  \mu_t (\varphi_t)  = & \ \mu_0 (\varphi_0) + \int_0^t \mu_r  (\partial_t \varphi_r + A_r \varphi_r) dr \nonumber \\
  &   + \int_0^t \mu_r  (\nabla \varphi_r \bar{f} (r, \cdot, Y_r) + \varphi_r
  h_2^T (r, \cdot, Y_r)) k^{+} (r, Y_r)  (dY_r - h_1 (r, Y_r) d r) \label{eq:ZakDeg_testfunctions} \\
   \underset{\tmop{ND}}{=} & \ \mu_0 (\varphi_0) + \int_0^t \mu_r  (\partial_t \varphi_r + A_r \varphi_r) dr \nonumber \\
  &   + \int_0^t \mu_r  (\nabla \varphi_r f (r, \cdot, Y_r) + \varphi_r h (r, \cdot, Y_r))  
  d Y_r,
  \quad \mu_0 = \tmop{Law} (X_0). \label{eq:ZakNonDeg_testfunctions}
\end{align}
    By generalizing \cite[Lemma 4.8]{BC09} to the stochastic filtering framework \eqref{eq:stoch-sig}-\eqref{eq:stoch-obs}, it is possible to prove that, at least in the non-degenerate case and under mild regularity and growth assumptions on the coefficients (which are satisfied, for instance, by taking $n=2$ in \cref{assumptionCPU}, regardless of $d_X$), for any given $\mathcal{F}^Y_t$-adapted, $\mathcal{M}_F(\R^{d_X})$-valued $\mu=(\mu_t)_{t \in [0,T]}$ with continuous paths, such that $\E(\int_0^T \mu_r(1) \, dr) < +\infty$, the following are equivalent:  
    \begin{enumerate}   
    \item $\mu$ satisfies \eqref{equ:ZakDeg},\eqref{equ:ZakNonDeg} for any $\varphi \in C^2_b(\R^{d_X};\R)$; 
    \item $\mu$ satisfies \eqref{equ:ZakDeg},\eqref{equ:ZakNonDeg} for any $\varphi \in C^\infty_c(\R^{d_X};\R)$;   
    \item $\mu$ satisfies \eqref{eq:ZakDeg_testfunctions}, \eqref{eq:ZakNonDeg_testfunctions} for any $\varphi \in C^{1,2}_b([0,T] \times \R^{d_X};\R)$.  
    \end{enumerate} 
\end{remark}

\section{Elements of rough (stochastic) analysis} \label{sec:elem}



\subsection{Rough paths} \label{sec:RPs}
We consider a class of level-2 rough path over $\mathbb{R}^{d_Y}$, say
$\mathbf{Y} = (Y, \mathbb{Y}) \in (\mathscr{C}, \rho)$ where $(\mathscr{C}, \rho)$  
is a complete and metric space of rough paths, such that $(\mathscr{C}, \rho) \hookrightarrow (\mathscr{C}^{\alpha} ,
\rho_{\alpha})$, the usual $\alpha$-H\"older rough path space\footnote{ $\rho_{\alpha} (\mathbf{Y}, \overline{\mathbf{Y}}) = |
  \delta Y - \delta \bar{Y} |_{\alpha} + |\mathbb{Y}- \bar{\mathbb{Y}} |_{2
  \alpha} := \sup_{s < t} \frac{|\delta Y_{s,t} - \delta \bar Y_{s,t}|}{|t-s|^\alpha}+ \sup_{s < t} \frac{|\mathbb{Y}_{s,t} - \bar{\mathbb{Y}}_{s,t}|}{|t-s|^{2\alpha}} $. Recall also the inhomogeneous rough path norm $|\BY|_\alpha:= |\delta Y|_\alpha + |\mathbb{Y}|_{2\alpha}$ and the
  homogeneous rough path norm
$ \nn{\mathbf{Y} }_\alpha = | \delta Y|_{\alpha} \vee
   \sqrt{|\mathbb{Y}|_{2 \alpha } } $. }, for some fixed
$\alpha \in (1 / 3, 1 / 2]$. We write $\mathscr{C}_T$ and
$\mathscr{C}^{\alpha}_T$ respectively when we want to  emphasize the time-horizon
$[0, T]$, the notation $\mathscr{C}^{\alpha}  ([0, T], \mathbb{R}^{d_Y})$ will also be used. \ For $\mathbf{Y} \in \mathscr{C}^\alpha$, the bracket
$[\mathbf{Y}]$ is a $2 \alpha$-H\"older path defined by
\[ (\delta Y) \otimes (\delta Y) =\mathbb{Y}+\mathbb{Y}^\top + \delta
   [\mathbf{Y}] . \]
Set also  $\mathbf{Y}^\circ = (Y, \mathbb{Y}^{\circ}) \assign
(Y, \mathbb{Y} +  \delta [\mathbf{Y}] / 2)$. 
In this work, we find it convenient to work with spaces where bracket draft corrections can be written in terms of Lebesgues integrals. Here are some examples of rough path spaces we have in mind, notation along \cite{FH20,FLZ25}:
\begin{itemize}
  
  
  
  
  \item  $(\mathscr{C}^{\alpha, 1}, \rho_{\alpha, 1})$, the complete (i.g.\ non-separable) space of $\alpha$-H\"older rough paths with Lipschitz bracket,
  $\mathscr{C}^{\alpha,1} := \{\mathbf{Y} \in
  \mathscr{C}^{\alpha} \mid | [ \mathbf{Y} ] |_{Lip} := \frac{|\delta [\BY]_{s,t}|}{|t-s|} < \infty\}$ with 
  $$
  \rho_{\alpha, 1} (\mathbf{Y}, \tmmathbf{Z}) \assign \rho_{\alpha}
  (\mathbf{Y}^{\circ}, \tmmathbf{Z}^{\circ}) + | [\mathbf{Y}] -
  [\tmmathbf{Z}] |_{Lip}.$$
  \item $(\mathscr{C}^{0, \alpha, 1}, \rho_{\alpha, 1})$, the Polish space
  of $\alpha$-H\"older rough paths with continuously differentiable bracket,
  with $\mathscr{C}^{0,\alpha, 1} :=  \{\mathbf{Y} \in
  \mathscr{C}^{0 , \alpha} \mid | [ \mathbf{Y} ] |_{C^1} < \infty\}$. 
\end{itemize}
In the above, $\mathscr{C}^{0 , \alpha}$ denotes the Polish space of $\alpha$-H\"older rough paths obtained as $\rho_\alpha$-closure of smooth rough paths (cf.\ \cite[Exercise 2.8]{FH20}). 
Writing $\mathscr{C}^{0, \alpha}_g$ for the classical (Polish) space of {\em geometric $\alpha$-H\"older rough paths}, obtained as $\rho_\alpha$-closure of lifted smooth paths. We have 
$$
(\mathscr{C}^{0, \alpha}_g, \rho_{\alpha}) 
\hookrightarrow 
(\mathscr{C}^{0, \alpha, 1}, \rho_{\alpha, 1})
\hookrightarrow 
(\mathscr{C}^{ \alpha, 1}, \rho_{\alpha, 1}). 
$$ 

\subsection{Rough stochastic differential equations} \label{app:RSDE}
Following \cite{FHL21, FLZ25},
\textit{rough It{\^o} processes} are continuous adapted stochastic processes, on some stochastic basis $\mathbf{\Omega}:=(\Omega, \mathcal{F}, \{\mathcal{F}_t\}_{t \in [0,T]}, \mathbb{P})$, of the form $X = X (\omega, \mathbf{Y}) = X^{\mathbf{Y}} (\omega)$ with dynamics
\begin{equation}
  \label{RIP1} d X_t = A_t (\omega, \mathbf{Y}) d t + \Sigma_t (\omega,
  \mathbf{Y}) d B_t + (F_t, F'_t) (\omega, \mathbf{Y}) d \mathbf{Y}_t,
  \qquad t \in [0,T],
\end{equation}
where $B = B (\omega)$ is Brownian motion and $\mathbf{Y} = (Y, \mathbb{Y})
\in \mathscr{C}   {  \equiv \mathscr{C}_T  \subset  \mathscr{C}^{\alpha}  ([0, T], \mathbb{R}^{d_Y})}$, a class of deterministic rough paths, on a fixed H\"older 
scale with exponent $\alpha \in (1/3,1/2]$. All coefficients are assumed progressive
in $(t, \omega)$ and $(F_t, F'_t) (\omega, \mathbf{Y})$ is a stochastic
controlled rough path in the sense of \cite[Definition 3.1]{FHL21}; this allows to give
bona fide integral meaning to \eqref{RIP1} and the reader can also find an It{\^o} formula for such processes in \cref{eq:roughItoformula_roughItoprocess} below. 

\begin{definition} \label{def:stochasticcontrolledroughpaths} (see \cite[Definition 3.1]{FHL21})
    Let $p \in [2,\infty)$ and $q \in [p,\infty]$. 
    We say that $(F,F')=(F_t,F'_t)(\omega,\BY)$ is a stochastic controlled rough path in $\mathbf{D}_Y^{2\alpha}L_{p,q}([0,T], \mathbf{\Omega};\mathrm{Lin}(\R^{d_Y},\R^{d_X}))$ if \begin{equation*}
        (F,F') : \Omega \times [0,T] \to \mathrm{Lin}(\R^{d_Y},\R^{d_X}) \times \mathrm{Lin}(\R^{d_Y}, \mathrm{Lin}(\R^{d_Y},\R^{d_X}))
    \end{equation*}
    are progressively measurable and\footnote{As usual, $L_q$ are the classical Lebesgue spaces equipped with norm $\| \cdot \|_q = \E( |\cdot|^q)^\frac1q$.}
    \begin{multline*}
    \|(F,F')\|_{Y;2\alpha;p,q} := \sup_{0 \le s < t \le T} \Big( \frac{\| \E(|\delta F_{s,t}|^p \mid \mathcal{F}_s)^{1/p}\|_q}{|t-s|^\alpha} +  \|F'_t\|_q  \\ 
    + \frac{\| \E(|\delta F'_{s,t}|^p \mid \mathcal{F}_s)^{1/p}\|_q}{|t-s|^\alpha} + \frac{\| \E(\delta F_{s,t} - F'_s \delta Y_{s,t} \mid \mathcal{F}_s)\|_q}{|t-s|^{2\alpha}} \Big) <+\infty. 
\end{multline*} 
If $p=q$, we write $\mathbf{D}_Y^{2\alpha}L_{p}([0,T],\mathbf{\Omega};\mathrm{Lin}(\R^{d_Y},\R^{d_X}))$ instead of $\mathbf{D}_Y^{2\alpha}L_{p,p}([0,T],\mathbf{\Omega};\mathrm{Lin}(\R^{d_Y},\R^{d_X}))$. 
If $F,F'$ are deterministic, we write $(F,F') \in \mathscr{D}_Y^{2\alpha}([0,T],\mathrm{Lin}(\R^{d_Y},\R^{d_X}))$ and we retrieve the notion of controlled rough paths of \cite[Definition 4.6]{FH20}. 
\end{definition}

From a stochastic calculus perspective, \eqref{RIP1} generalizes classical
It{\^o}-diffusions $X = X (\omega)$, with dynamics
$ d X_t = a_t (\omega) d t + \sigma_t (\omega) d B_t$,
which accommodates solutions to It{\^o} stochastic differential equations, $$ X_t = b_t (X_t, \omega) \, d t +  \sigma_t (X_t, \omega) \, d B_t.$$
From a rough path perspective, (\ref{RIP1}) generalizes rough integration, with dynamics
\begin{equation*}
  \label{Dav} d X_t = (\phi_t, \phi'_t) \,  d \mathbf{Y}_t \Leftrightarrow \begin{cases}
      \delta  X_{s, t} \approx_{3 \alpha} \phi_s \delta Y_{s, t} + \phi'_s \mathbb{Y}_{s, t} \\
      \delta \phi_{s, t} \approx_{2 \alpha} \phi_s' \delta Y_{s, t} \\
      \delta \phi'_{s,t} \approx_{\alpha} 0
  \end{cases} 
\end{equation*}
where $(\phi_t, \phi'_t) = (\phi_t, \phi'_t) (\mathbf{Y})$ is a given (deterministic) $Y$-controlled rough path and by $ \Psi_{s,t} \approx_{\gamma} \Phi_{s,t}$ we mean $| \Psi_{s,t} - \Phi_{s,t} | \le C |t-s|^\gamma$ for some constant $C$, uniformly on $[0,T]$.
This accommodates solutions to the ``contextualized'' (by which we mean explicit
$\mathbf{Y}$-dependence in the coefficient fields) rough differential equation (RDE)
\begin{equation}
  \label{RDE1} d X_t = f (X_t ; \mathbf{Y}) d \mathbf{Y}_t
  \end{equation}
via the usual Davie expansion $\phi_t := f (X_t ; \mathbf{Y})$, $  \phi'_t := ((D f) f) (X_t ; \mathbf{Y})$. This also accommodates the case of regular time-dependence $f (t,X_t ; \mathbf{Y})$  in \eqref{RDE1}; and more generally RDEs of the form
\begin{equation*}
  \label{RDE2} d X_t = (f_t, f_t') (X_t ; \mathbf{Y}) d {\mathbf{Y}_t}  ,
\end{equation*}
  in which case $\phi_t =f_t (X_t ; \mathbf{Y})$, $\phi'_t = ((D f_t) f_t + f_t')(X_t ; \mathbf{Y})$. 
 Main interest in $\eqref{RIP1}$ comes from rough SDEs 
\cite{FHL21}
\begin{equation}
  \label{RSDE1} d X_t = b_t (X_t, \omega ; \mathbf{Y }) d t + \sigma_t (X_t,
  \omega ; \mathbf{Y }) d B_t + (f_t, f'_t) (X_t, \omega ; \mathbf{Y}) d
  \mathbf{Y}_t,
\end{equation}
solutions $X = X^{\mathbf{Y}}$ of which provide natural examples of rough
It{\^o} processes, with
\[ (F_t, F'_t) (\omega, \mathbf{Y}) = (f_t (X_t, \omega ; \mathbf{Y}),\ ((D  f_t) f_t + f_t') (X_t, \omega ; \mathbf{Y})) . \]

The (random) coefficients $f$ and $f'$ in \eqref{RSDE1} are typically assumed to be stochastically $Y$-controlled in the sense of one of the two following definitions. 

\begin{definition} \label{def:stochasticcontrolledvectorfields_intro}
    (see \cite[Definition 2.3]{BFS24})
    Let $p \in [2,\infty)$, $q \in [p,\infty]$ and let $\beta = N + \kappa > 1$ for some $N \in \mathbb{N}_{\ge 0}$ and $\kappa \in (0,1]$.
    We say that $(f,f') = (f_t(x, \omega;\BY), f'_t(x, \omega;\BY))$ is a stochastic controlled vector field in $\mathbf{D}_Y^{2\alpha}L_{p,q} \mathcal{C}^\beta_b([0,T], \mathbf{\Omega}; \R^{d_X})$ if \begin{equation*}
        (f,f') : \Omega \times [0,T] \to \mathcal{C}^\beta_b(\R^{d_X}; \mathrm{Lin}(\R^{d_Y}, \R^{d_X})) \times \mathcal{C}^{\beta-1}_b(\R^{d_X}; \mathrm{Lin}(\R^{d_Y},\mathrm{Lin}(\R^{d_Y}, \R^{d_X})))
    \end{equation*}
   is progressively measurable, $\|(f , f') \|_\beta := \sup_{t \in [0,T]} \| |f_t|_{\mathcal{C}^\beta_b} \|_q + \sup_{t \in [0,T]} \| |f'_t|_{\mathcal{C}^{\beta-1}_b} \|_q < +\infty$ and 
    \begin{equation*}
        \llbracket f,f' \rrbracket_{Y;2\alpha;\beta;p,q} := \sum_{j=0}^N \llbracket \delta D_x^j f \rrbracket_{\alpha;p,q} + \sum_{j=0}^{N-1} ( \llbracket \delta D_x^j f' \rrbracket_{\alpha;p,q} + \llbracket R^{D_x^j f} \rrbracket_{2\alpha;p,q} ) < +\infty.  
    \end{equation*}
    In the above, $D^0_x f = f$, $\delta f_{s,t} := f_t - f_s$,  $R^{D^j_xf}_{s,t} := D^j_x f_t - D^j_x f_s - D^j_x f'_s \delta Y_{s,t}$ and, for any $\zeta=\zeta_{s,t}(x,\omega)$, $$\llbracket \zeta \rrbracket_{\gamma;p,q} := \sup_{0 \le s < t \le T} \frac{ \|\E(\sup_{x \in \R^{d_X}} |\zeta_{s,t}(x,\omega)|^p \mid \mathcal{F}_s )^{1/p}\|_q}{|t-s|^\gamma}. $$   
    If $f, f'$ and $\zeta$ are deterministic, we simply write $(f,f') \in \mathscr{D}_Y^{2\alpha} \mathcal{C}^\beta_b ([0,T]; \R^{d_X})$, $\llbracket f,f' \rrbracket_{Y;2\alpha;\beta}$ and $\llbracket \zeta \rrbracket_{\gamma}$. 
\end{definition}

\begin{definition} \label{def:stochasticcontrolledlinearvectorfield} (see \cite[Definition 3.2]{BCN24})
    Let $p \in [2,\infty)$, $q \in [p,\infty]$. 
    We say that $(f,f') = (f_t(x, \omega;\BY), f'_t(x, \omega;\BY))$ is a stochastic controlled linear vector field in $\mathbf{D}_Y^{2\alpha}L_{p,q} \mathrm{Lin}([0,T], \mathbf{\Omega}; \R^{d_X})$ if
    \begin{equation*}
        (f,f') : \Omega \times [0,T] \to \mathrm{Lin}(\R^{d_X}; \mathrm{Lin}(\R^{d_Y}, \R^{d_X})) \times \mathrm{Lin}(\R^{d_X}; \mathrm{Lin}(\R^{d_Y},\mathrm{Lin}(\R^{d_Y}, \R^{d_X})))
    \end{equation*}
    are progressively measurable, $\|(f,f')\|_\infty := \sup_{t \in [0,T]} \| |f_t|_{\mathrm{Lin}}\|_\infty + \sup_{t \in [0,T]} \| |f'_t|_{\mathrm{Lin}}\|_\infty < +\infty$ and \begin{multline*}
        \sup_{0 \le s < t \le T} \Big( \frac{\| \E(|\delta f_{s,t}|_{\mathrm{Lin}}^p \mid \mathcal{F}_s)^{1/p}\|_q}{|t-s|^\alpha} + \frac{\| \E(|\delta f'_{s,t}|_{\mathrm{Lin}}^p \mid \mathcal{F}_s)^{1/p}\|_q}{|t-s|^\alpha} + \frac{\| \E(\delta f_{s,t} - f'_s \delta Y_{s,t} \mid \mathcal{F}_s)\|_q}{|t-s|^{2\alpha}} \Big) < +\infty. 
    \end{multline*}    
    In case of deterministic $f$ and $f'$, we write $\mathscr{D}_Y^{2\alpha}\mathrm{Lin}([0,T];\R^{d_X})$
\end{definition}

We briefly recall the notion of solution for an RSDE \eqref{RSDE1} and we also collect three versions of the so-called rough stochastic It\^o formula. 

\begin{definition} (\cite[Definition 4.2]{FHL21}) \label{def:integrablesolutionsRSDEs}
    We say that $X=X^\BY_t(\omega)$ is an $L_{p,q}$-integrable solution to \eqref{RSDE1} if it is continuous and adapted, and the following are satisfied \begin{equation*}
    \begin{cases}
        \text{$\int_0^T |b_r(X_r,\omega;\BY)| \, dr$ and $\int_0^T |\sigma_r \sigma^\top_r(X_r,\omega;\BY)| \, dr$ are finite $\mathbb{P}$-a.s. } \\
        (f_t (X_t, \omega ; \mathbf{Y}),\ ((D  f_t) f_t + f_t') (X_t, \omega ; \mathbf{Y})) \in \mathbf{D}_Y^{2\alpha}L_{p,q}([0,T] , \mathbf{\Omega}; \mathrm{Lin}(\R^{d_Y},\R^{d_X})) \\
        \delta X_{s,t} = \int_s^t b_r(X_r,\omega;\BY) \, dr + \big( \int_s^t \sigma_r(X_r, \cdot; \BY) \, dB_r \big)(\omega) + f_s(X_s,\omega;\BY) \delta Y_{s,t} \\
        \qquad \qquad  + ((Df_s)f_s + f'_s)(X_s,\omega;\BY) \mathbb{Y}_{s,t} + X^\natural_{s,t}
    \end{cases}
\end{equation*}
    with $ \sup_{0 \le s < t \le T} \frac{\| \E(|X^\natural_{s,t}|^p \mid \mathcal{F}_s)^{1/p}\|_q}{|t-s|^{2\alpha}} < +\infty$ and $ \sup_{0 \le s < t \le T} \frac{\|\E(X^\natural_{s,t} \mid \mathcal{F}_s)\|_q}{|t-s|^{3\alpha}} < +\infty$. 
\end{definition}
Well-posedness results for \eqref{RSDE1} were obtained, with $p\in [2, \infty), \,  q=\infty$ for bounded non-linear coefficients in \cite{FHL21}. The case of linear coefficients with $p=q \in [2,\infty)$ was considered in \cite[Section 3]{BCN24}, see also \cite[Section 3]{HZ25}.

\noindent {\em Ito's formula.}  We say that a {\em rough stochastic It\^o formula} holds for a rough It\^o process of the form \eqref{RIP1} if, for test function $\varphi,$ 
 $\varphi(X)$ is again a rough It\^o process with \begin{multline*} 
    d \varphi(X_t) = \Big( D\varphi(X_t) A_t(\omega,\BY) + \frac{1}{2} D^2 \varphi(X_t) : (\Sigma_t \Sigma^\top_t)(\omega,\BY) \Big) \, dt  + \frac{1}{2} D^2 \varphi(X_t) : (F_t \, d[\BY]_t F_t^\top) \\ + D\varphi(X_t) \Sigma_t(\omega,\BY) \, dB_t + \big(D\varphi(X_t) F_t, D^2\varphi(X_t): (F_t F_t^\top) + D\varphi(X_t) F'_t \big)(\omega,\BY) \, d\BY_t
\end{multline*}
where $:$ denotes the Frobenius matrix inner product and  $d[\BY]_t$ is understood as a Young integral. Both the rough stochastic integral and the Young integral in this identity are defined as limits of (compensated) Riemann sums in $L_2$.

\begin{theorem} \label{eq:roughItoformula_roughItoprocess}
  
A \emph{rough stochastic It\^o formula} holds if one of the three sets of assumptions is satisfied: 
 \begin{itemize}
        \item[$1$.]{(\cite[Theorem 4.13]{FHL21})} Take test functions $\varphi \in \mathcal{C}^\gamma_b(\R^{d_X};\R)$ for some $\gamma \in (\frac{1}{\alpha},3]$.
        Let $A, \Sigma$ be bounded, and let $(F,F') \in \mathbf{D}_Y^{2\alpha}L_{4,\infty}$ be with $\|F_0\|_4 <+\infty$.  
        \item[$2$.]{(\cite[Proposition 2.16]{BCN24})}  Take test functions $\varphi \in \mathcal{C}^\gamma_{pol,m}(\R^{d_X};\R)$ \footnote{For $N,m \in \mathbb{N}_{\ge 0}$ and let $\kappa \in (0,1]$, we define \begin{equation} \label{eq:testfunctionspolynomialgrowth}
        \mathcal{C}^{N+\kappa}_{pol,m} (\R^{d_X};\R) := \left\{ \varphi \in \mathcal{C}^{N+\kappa}(\R^{d_X};\R) \, \bigg| \,  \sup_{x \in \R^{d_X}} \frac{|\varphi(x)| + \sum_{j=1}^N |D^j \varphi(x)|}{1+|x|^m} 
        < +\infty \right\}
        \end{equation}
        } 
        for some $\gamma \in (\frac{1}{\alpha},3]$ and $m \in \mathbb{N}_{\ge 1}$.
        For any $p \in [2,\infty)$, assume $\sup_{t \in [0,T]} \|A_t\|_p + \| \Sigma_t\|_p < + \infty$ and $(F,F') \in \mathbf{D}_Y^{2\alpha}L_{p}$ with $\|F_0\|_p <+\infty$.
        \item[$3$.]{(\cite[Proposition 4.14]{BFS24})} Take as a test function $\varphi(x) := e^x$. Let $\E(\int_0^T |A_t|^p \, dt) < +\infty$ for any $p \in [2,\infty)$, $\Sigma \equiv 0$ and $(F,F') \in \mathbf{D}_Y^{2\alpha}L_{1,\infty}$.
        Then $\sup_{t \in [0,T]} \|e^{X_t}\|_p < +\infty$ for any $p \in [2,\infty)$ and the rough stochastic It\^o formula holds.  
\end{itemize}
       
\end{theorem}

\subsection{Regular solutions to rough backward equations} \label{appendix}
In \cref{section:Zakai_uniqueness} we show how uniqueness for the measure-valued rough Zakai equation \eqref{eq:roughZakai_CITE} follows from the existence of a solution to the function-valued backward equation \eqref{eq:dualroughequation}. 
The framework developed in \cite{BFS24} covers backward rough PDEs driven by weakly geometric rough paths only. In this section we show how to extend their existence result to general (possibly non-weakly geometric) drivers.
Let the setting of \cref{section:Zakai_uniqueness} be in force. 
Define $\BY^{\circ} = \BY + \frac{1}{2} (0, \delta
[\BY])$ and write $\circ d \BY $:= $d
\tmmathbf{\BY^{\circ}}$. Notice that $\BY^{\circ} \in
\mathscr{C}^{0,\alpha, 1}([0, T] , \mathbb{R}^{d_Y})$ is a weakly geometric rough path, that is \begin{equation*}
    (\mathbb{Y}^\circ_{s,t})^{\kappa \lambda} + (\mathbb{Y}^\circ_{s,t})^{\lambda \kappa} = \delta Y^\kappa_{s,t} \delta Y^\lambda_{s,t} 
\end{equation*}
for any $0 \le s \le t \le T$ and for any $\kappa,\lambda=1,\dots,d_Y$. 
Notice that the RSDE \eqref{eq:roughSDE_filtering} can be written as  \begin{equation} \begin{aligned}
    d X^{\BY}_t &=  b^{\circ,[\BY]} (t, X^{\BY}_t; \BY )     d t + \sigma (t, X^{\BY}_t;\BY) d B_t +     f (t, X^{\BY}_t ; \BY) \,\circ d \BY_t ,
\end{aligned} \label{eq:RSDE_geometric}
\end{equation}
with \begin{equation*}
    b^{\circ, [\BY]} (t, x; \BY)  :=  b (t, x; \BY) - \frac{1}{2}   \sum_{\kappa, \lambda = 1}^{d_Y} \Big(\sum_{j=1}^{d_X}\partial_{x^j} f_{\lambda} (t, x; \BY)   f^j_{\kappa} (t, x; \BY)  +   f'_{\lambda \kappa} (t, x; \BY)\Big) \dot{[\BY]}_t^{\kappa \lambda}.
\end{equation*}
Furthermore, the RSDE \eqref{eq:Girsanovexponential_filtering} can be written as
\begin{equation} \label{eq:Girsanovexponential_geometric}
    d Z^{\BY}_t = Z^\BY_t c^{\circ,[\BY]} (t, X^{\BY}_t ; \BY) \, dt + Z^\BY_t h (t, X_t^{\BY} ;\BY) \,  \circ d\BY_t ,
\end{equation}
where 
\begin{align*}
    c^{\circ, [\BY]} (t, x; \BY) & := 
    - \frac{1}{2} \sum_{\kappa, \lambda = 1}^{d_Y} \Big( h_{\kappa}   (t, x; \BY)   h_\lambda (t,   x; \BY) + \sum_{j=1}^{d_X} \partial_{x^j} h_{\lambda}    (t, x; \BY) f^j_{\kappa} (t, x;\BY) \\
    & \qquad + h'_{\lambda \kappa} (t, x; \BY)\Big) \dot{[\BY]}_t^{\kappa \lambda}
\end{align*}
Define the time dependent second order differential operator $L^{\circ, \BY}_t$ as 
\be\label{def:Lo-operator}
L^{\circ, \BY}_t \varphi (x) \assign \frac{1}{2} \sum_{i, j = 1}^{d_X}     \partial_{i j}^2 \varphi (x) a^{i j} (t, x; \BY) + \sum_{i = 1}^{d_X}     \partial_i \varphi (x) (b^{\circ, [\BY]})^i (t, x ; \BY) + \varphi     (x) c^{\circ, [\BY]} (t, x; \BY) 
\ee
  with $a^{i j} (t, x; \BY) :=  \sum_{\theta = 1}^{d_B} \sigma_{\theta}^i(t, x; \BY) \sigma^j_{\theta} (t, x; \BY)$. 


\begin{proposition} \label{prop:equivalentformofroughZakai-back} $u$ is a function-valued solution to
  \begin{equation*} 
    u_t (x) = u_T(x) + \int_t^T \big( A^{\BY}_r u_r(x) - (\Gamma_r^\BY)' u_r(x) \dot{[\BY]}_r \big) \, dr + \int_t^T \Gamma^{\BY}_r u_r(x)  \, d \BY_r 
  \end{equation*}
  if and only it solves 
  \begin{equation} \label{eq:roughKolmogorov_Stratoform}
    u_t(x) = u_T(x) + \int_t^T L^{\circ,\BY}_r u_r(x) \, dr + \int_t^T \Gamma^{\BY}_r u_r(x) \,  \circ d \BY_r .
  \end{equation}
  
\end{proposition}

\begin{proof}
  For sake of a more compact notation, we adopt Einstein's repeated indeces summation convention. The proof follows by comparing the terms multiplying $\mathbb{Y}$ in the Davie expansion of \cref{def:solutionbackwardroughPDE}.\textit{(iii)} with the Lebesgue integral. 
  Notice that \begin{align*}
      &L^{\circ,\BY}_t u_t (x)  = A^{\BY}_t u_t (x)  -
    \frac{1}{2} \partial_{i j}^2 u_t (x) f_{\kappa}^i (t, x; \BY) \dot{[\BY]}_t^{\kappa \lambda} f_{\lambda}^j (t, x, Y_t)  \\
    &   \quad  - \partial_i u_t (x) \Big(f_{\kappa}^i (t, x; \BY)h_{\lambda} (t, x, Y_t)  +  \frac{1}{2} (\partial_{x^j} f_{\lambda}^i (t, x; \BY) f^j_{\kappa}(t, x; \BY) +(f'_{\lambda \kappa})^i (t, x; \BY)) \Big) \dot{[\BY]}_t^{\kappa \lambda}    \\
    &   \quad   - \frac{1}{2} u_t (x) \Big(h_{\kappa} (t, x; \BY) h_{\lambda} (t, x; \BY)  + \nobracket (\partial_{x^j} h_{\lambda} (t, x, Y_t) f^j_{\kappa}(t, x; \BY) +  h'_{\lambda\kappa} (t, x; \BY)) \Big) \dot{[\BY]}_t^{\kappa \lambda} .   
  \end{align*}
  Recall that, by definition, $\delta [\BY]_{s, t} = \int_s^t \dot{[\BY]}_r \, dr$ and $(\delta \dot{[\BY]}_{s, t})^{\kappa \lambda} = (\delta  \dot{[\BY]}_{s, t})^{\lambda \kappa}$. Therefore, \begin{align*}
      & ((\Gamma^{\BY}_t)_{\kappa}(\Gamma_t^{\BY})_{\lambda} u_t(x) -(\Gamma_t^\BY)'_{\kappa \lambda}  u_t(x)) \mathbb{Y}^{\kappa \lambda}_{s, t}  \\
      & =  ((\Gamma^{\BY}_t)_{\kappa}(\Gamma_t^{\BY})_{\lambda} u_t(x) -(\Gamma_t^\BY)'_{\kappa \lambda}  u_t(x)) (\mathbb{Y}^{\circ})^{\kappa \lambda}_{s, t}  - \frac{1}{2} ((\Gamma^{\BY}_t)_{\kappa}(\Gamma_t^{\BY})_{\lambda} u_t(x) -(\Gamma_t^\BY)'_{\kappa \lambda}  u_t(x)) \int_s^t \dot{[\BY]}_r^{\kappa \lambda} \, d r\\
    & =  ((\Gamma^{\BY}_t)_{\kappa}(\Gamma_t^{\BY})_{\lambda} u_t(x) -(\Gamma_t^\BY)'_{\kappa \lambda}  u_t(x)) (\mathbb{Y}^{\circ})^{\kappa \lambda}_{s, t}   - \frac{1}{2} \int_s^t ((\Gamma^{\BY}_r)_{\kappa}(\Gamma_r^{\BY})_{\lambda} u_r(x) -(\Gamma_r^\BY)'_{\kappa \lambda}  u_r(x)) \dot{[\BY]}_r^{\kappa \lambda} \, d r + R_{s, t} 
  \end{align*}
  where $| R_{s, t} | \leqslant C_{\varphi} | t - s |^{1 + \alpha}$. Indeed,  recall that $| ((\Gamma^{\BY}_t)_{\kappa}(\Gamma_t^{\BY})_{\lambda} u_t(x) -(\Gamma_t^\BY)'_{\kappa \lambda}  u_t(x))- ((\Gamma^{\BY}_s)_{\kappa}(\Gamma_s^{\BY})_{\lambda} u_s(x) -(\Gamma_s^\BY)'_{\kappa \lambda}  u_t(x))
  | \le C_{\varphi} | t - s |^{\alpha}$.
  We can write \begin{align*}
    &\frac{1}{2} (\Gamma^{\BY}_t)_{\kappa}    (\Gamma_t^{\BY})_{\lambda} u_t (x) \dot{[\BY]}_t^{\kappa \lambda}  \\
    & =  \frac{1}{2} \Big(\partial_i  ((\Gamma_t)_{\lambda} u_t) (x) f_{\kappa}^i (t, x; \BY) +   (\Gamma_t)_{\lambda} u_t (x) h_{\kappa} (t, x; \BY) \Big)   \dot{[\BY]}_t^{\kappa \lambda} \\
    &  =  \frac{1}{2} \Big(\partial^2_{i j} u_t (x) f_\kappa^i (t, x;\BY) f_{\lambda}^j (t, x; \BY) + \partial_i u_t (x) \partial_{x^j}  f_{\lambda}^i (t, x; \BY) f_{\kappa}^j (t, x; \BY)  \\
    & \quad + 2 \partial_i u_t (x) h_{\lambda} (t, x; \BY)
    f_{\kappa}^i (t, x ; \BY) + u_t (x) \partial_{x^i} h_{\lambda} (t, x;\BY) f_{\kappa}^i (t, x; \BY) \\
    & \quad + u_t (x) h_{\lambda} (t, x; \BY) h_{\kappa} (t, x; \BY ) \Big) \dot{[\BY]}_t^{\kappa \lambda} \quad  
  \end{align*}
  and
  \[
  \begin{split}
    \frac{1}{2} (\Gamma_t')_{\kappa \lambda } u_t (x)
    \dot{[\BY]}_t^{\kappa \lambda} & =  \frac{1}{2} \Big(\partial_i u_t (x) (f_{\kappa \lambda}')^i (t, x; \BY) + u_t (x) h_{\kappa \lambda}' (t, x ; \BY) \Big)    \dot{[\BY]}_t^{\kappa \lambda} .     
  \end{split}
  \]
  Putting everything together we have proved that \begin{multline*}
      \int_s^t \big( A^{\BY}_r u_r(x) -(\Gamma^\BY_r)_{\lambda \kappa}'u_r(x) \dot{[\BY]}^{\kappa\lambda}_r \big) \, dr
     + ((\Gamma^{\BY}_t)_{\kappa}
     (\Gamma_t^{\BY})_{\lambda} - (\Gamma_t^{\BY})'_{\kappa \lambda}) u_t(x) \mathbb{Y}_{s, t}^{\kappa \lambda} 
       \\ =  \int_s^t L^{\circ,\BY}_r u_r(x)
     \, dr 
      + ((\Gamma^{\BY}_t)_{\kappa} (\Gamma_t^{\BY})_{\lambda}
     - (\Gamma_t^{\BY})'_{\kappa \lambda }) u_t(x) 
     (\mathbb{Y}^{\circ})_{s, t}^{\kappa \lambda} + u_{s, t}^{\natural}    
  \end{multline*}
  with $u_{s, t}^{\natural} = o (| t - s |)$ as $| t - s | \rightarrow 0$. 
\end{proof}

The following assumptions (cf.\  \cite[Assumptions 4.4]{BFS24}) guarantee the existence of a (unique) solution to equation \eqref{eq:roughKolmogorov_Stratoform} . Notice that under \cref{assumptionU} they are automatically satisfied. 

\begin{assumption}{BFS-U} \label{assumptionBFSU} 
Let $\beta \in (\frac{1}{\alpha},3]$ and let $\theta \in [4,5]$ such that $2\alpha + (\theta-4)\alpha >1$. 
Let the following hold: \begin{itemize}
    \item $ [t \mapsto \bar b^{\circ,[\BY]}(t,\cdot \, ; \BY) :=  b^{\circ, [\BY]} (t, \cdot \, ; \BY) - \sum_{\kappa=1}^{d_Y} f_\kappa(t,\cdot \, ;\BY) h_1^\kappa(t;\BY)] \in  \mathcal{B}_b([0,T];\mathcal{C}^\beta_b(\R^{d_X};\R^{d_X}))$ and $$ t \mapsto |b^{\circ, [\BY]} (t, \cdot \, ; \BY)|_\infty + |D_x b^{\circ, [\BY]} (t, \cdot \, ; \BY)|_\infty + |D^2_{xx}b^{\circ, [\BY]} (t, \cdot \, ; \BY)|_\infty \quad \text{is continuous}, $$ with similar conditions for   $[ t \mapsto \sigma(t,\cdot \, ; \BY)] $ and $[ t \mapsto c^{\circ, [\BY]} (t, \cdot \, ; \BY)] $;
    \item $[t \mapsto (f (t, \cdot \, ;\BY),  f' (t, \cdot \, ; \BY))] \in \mathscr{D}_{Y}^{2 \alpha} \mathcal{C}_b^\theta$ 
    in the sense of \cref{def:stochasticcontrolledvectorfields_intro} ; 
    \item $[t \mapsto (h(t,\cdot \, ;\BY), h' (t,\cdot \, ; \BY))] \in \mathscr{D}_{Y}^{2 \alpha} \mathcal{C}_b^\theta$. 
\end{itemize}
\end{assumption} 
With a very similar argument as in the proof of \cref{prop:equivalentformofroughZakai-back}, one can deduce the following

\begin{proposition} \label{prop:equivalentformofroughZakai} $\mu$ is a measure-valued solution to
  \begin{equation*} 
    \mu_t (\varphi) = \mu_0 (\varphi) + \int_0^t \mu_r 
    (A^{\BY}_r \varphi) dr + \int_0^t \mu_r (\Gamma^{\BY}_r
    \varphi)  \, d \BY_r,
  \end{equation*}
  if and only it solves 
  \begin{equation*} \label{eq:roughZakai_Stratoform}
    \mu_t (\varphi) = \mu_0 (\varphi) + \int_0^t \mu_r 
    (L^{\circ,\BY}_r \varphi) dr + \int_0^t \mu_r (\Gamma^{\BY}_r
    \varphi) \, \circ d \BY_r .
  \end{equation*}
  
\end{proposition}

\section{Exponentiation of second Wiener chaos functionals} \label{app:EWC}

\begin{proposition}
 \label{prop:exp2chaos}
 Let $(W_t)_{t\in[0,T]}$ be a standard Wiener process.\footnote{The case of $d$-dimensional Wiener process is analogous, suffices to work with $L^2([0,T],\mathbb{R}^d)$.} 
 For $h\in L^2[0,T]$ and symmetric $f\in L^2([0,T]^2)$, let
 \[
 I_1(h)=\int_0^T h(t)\,dW_t,\qquad
 I_2 (f)=\iint_{[0,T]^2} f(s,t)\,dW_s\,dW_t ,
 \]
 and set $X=I_1(h)+I_2(f)$.
 Let $A:L^2[0,T]\to L^2[0,T]$ be the self-adjoint Hilbert--Schmidt operator with kernel $f$,
 \[
 (Ag)(t)=\int_0^T f(t,s)\,g(s)\,ds .
 \]
 If $(I-2A)$ is strictly positive, equivalently $\sup\sigma(A)<\tfrac12$, then
 \[
 \E\big[e^{\,X}\big]
 = \det\nolimits_{2}(I-2A)^{-\frac12}\;
 \exp\!\Big(\frac12\,\big\langle (I-2A)^{-1}h,\,h\big\rangle\Big),
 \]
where $\det_2$ is the Carleman–Fredholm determinant.
\end{proposition}

 \begin{proof}[Proof (sketch)]
 Diagonalize $A$ in $L^2[0,T]$: there exist eigenpairs $(\lambda_i,e_i)$ with $Ae_i=\lambda_i e_i$ and i.i.d.\ $\zeta_i:=\int_0^T e_i\,dW\sim N(0,1)$. Writing $\beta_i:=\langle h,e_i\rangle$, one has the chaos decomposition
 \[
 X=\sum_i\big(\beta_i\zeta_i+\lambda_i(\zeta_i^2-1)\big).
 \]
 Independence and the one-dimensional Gaussian identity
 \[
 \E\big[e^{\,\beta Z+\lambda Z^2}\big] \;=\; (1-2\lambda)^{-1/2}\,
 \exp\!\Big(\frac{\beta^2}{2(1-2\lambda)}\Big),\qquad Z\sim N(0,1),\; 1-2\lambda>0,
 \]
 give
 \[
 \E[e^{\,X}]
 =\prod_i e^{-\lambda_i}(1-2\lambda_i)^{-1/2}\,
 \exp\!\Big(\frac{\beta_i^2}{2(1-2\lambda_i)}\Big).
 \]
 Recognizing $\det\nolimits_2(I-2A)=\prod_i e^{2\lambda_i}(1-2\lambda_i)$ and
 $\langle (I-2A)^{-1}h,h\rangle=\sum_i \beta_i^2/(1-2\lambda_i)$ yields the claim for finite-rank $A$. The general case follows by Hilbert–Schmidt approximation of $A$, continuity of the Carleman–Fredholm determinant, and $L^2$-limits.
 \end{proof}

\bibliographystyle{alpha}
\bibliography{processes_martingale}

\end{document}